\documentclass[openright,a4paper,american,11pt]{amsart}

\usepackage[latin1]{inputenc}

\usepackage{amsmath}
\usepackage{amssymb}
\usepackage{babel}

\usepackage{amsfonts}
\input xy
\xyoption{all}
\usepackage[dvips]{graphicx}
\usepackage{boxedminipage}
\usepackage{color}

\newcommand{\ZZ}{{\mathbb Z}}
\newcommand{\Z}{{\mathbb Z}}

\newcommand{\CC}{{\mathbb C}}
\newcommand{\RR}{{\mathbb R}}
\newcommand{\R}{{\mathbb R}}

\newcommand{\TT}{{\mathbb T}}
\newcommand{\T}{{\mathbb T}}
\newcommand{\TP}{{\mathbb{T}P}}
\newcommand{\CP}{{\mathbb{C}P}}

\newcommand{\A}{{\mathcal A}}

\newcommand{\D}{{\mathcal D}}
\newcommand{\C}{{\mathcal C}}
\newcommand{\E}{{\mathcal E}}
\newcommand{\F}{{\mathcal F}}
\newcommand{\X}{{\mathcal X}}
\newcommand{\Y}{{\mathcal Y}}
\newcommand{\M}{{\mathcal M}}

\newcommand{\U}{{\mathcal U}}
\renewcommand{\H}{{\mathcal H}}
\renewcommand{\S}{{\mathcal S}}
\renewcommand{\P}{{\mathcal P}}
\renewcommand{\M}{{\mathcal M}}
\renewcommand{\L}{{\mathcal L}}

\newcommand{\Log}{\text{Log}}
\newcommand{\Ed}{\text{Edge}}
\renewcommand{\div}{\text{div}}

\newcommand{\trp}{\text{Trop}}

\newcommand{\aff}{\text{aff}}
\newcommand{\Aff}{\text{Aff}}

\newtheorem{thm}{Theorem}[section]
\newtheorem{defi}[thm]{Definition}
\newtheorem{defn}[thm]{Definition}
\newtheorem{definition}[thm]{Definition}

\newtheorem{lemma}[thm]{Lemma}
\newtheorem{cor}[thm]{Corollary}
\newtheorem{corollary}[thm]{Corollary}

          {\theoremstyle{definition}
\newtheorem{rem}[thm]{Remark}}
          {\theoremstyle{definition}
\newtheorem{exa}[thm]{Example}
\newtheorem{example}[thm]{Example}}

\newtheorem{ques}[thm]{Question}

\newcommand{\comment}[1]{}

\begin{document}
\title{Obstructions to approximating tropical curves in surfaces via
  intersection theory}
\author{Erwan Brugallé}
\author{Kristin  Shaw}
\address{Erwan Brugallé, 
École polytechnique,
Centre Mathématiques Laurent Schwatrz, 91 128 Palaiseau Cedex, France}
\email{erwan.brugalle@math.cnrs.fr}
\address{Kristin Shaw, Departement of Mathematics, University of
  Toronto, 40 St. George St., Toronto, Ontario, CANADA
M5S 2E4.}
\email{shawkm@math.toronto.edu }
\date{\today}



\begin{abstract}
We provide some new local obstructions to 
approximating
tropical curves in
smooth tropical surfaces.  These obstructions are based on 
a
relation between tropical and complex intersection theories which is 
also established here. We give
two applications of the methods developed in this paper.
First we classify all locally irreducible approximable 3-valent fan tropical
curves in a 
fan tropical plane.
Secondly,  we prove that a generic non-singular
tropical surface 
in tropical projective 3-space contains finitely
many approximable tropical lines 
if 
it is of degree 3, and contains no approximable tropical lines if 
it is of degree 4 or more.
\end{abstract}
\maketitle

\hspace{45ex}\textit{Dedicated to the memory of Mikael Passare}

\begin{flushright}

\end{flushright}

\tableofcontents
\renewcommand{\L}{{\mathcal L}}
\begin{section}{Introduction/Main results}
\subsection{Background}
Tropical geometry is a recent field of mathematics which provides
powerful 
tools
to study classical algebraic varieties. 
The most striking example is undoubtedly the use of tropical methods in
real and complex enumerative geometry initiated by Mikhalkin in \cite{Mik1}. 
Tropical
varieties are piecewise polyhedral objects and satisfy the so-called
balancing condition (see \cite{Mik3} or \cite{St7} for a precise definition). 
One possible way, among others, to relate tropical
geometry to classical algebraic geometry is via \textbf{amoebas}
of complex varieties. 
Given a family $(\X_t)_{t\in\RR_{+}}$ of algebraic subvarieties of the
complex torus $(\CC^*)^N$, one may consider the corresponding family
of amoebas $\Log_t (\X_t)$ in $\RR^N$ where
$\Log_t$ is the map defined by
$$\begin{array}{cccc}
\Log_t: &(\CC^*)^N & \longrightarrow & \RR^N
\\ & (z_i) &\longmapsto & (\frac{\log |z_i|}{\log t})
\end{array}.$$
When the family $\Log_t (\X_t)$ converges (in the sense of the Hausdorff
metric on compact sets of $\RR^N$) as $t$ goes to infinity,  it
is known  that the limit set $X$ is a
 rational polyhedral complex (see \cite{BiGr}); moreover the facets of $X$
 come naturally equipped with positive integer weights making $X$ balanced
 (see \cite{Spe1} or \cite{St7}).

In this paper, we call a positively weighted, 
 balanced, rational, polyhedral complex a \textbf{tropical variety}. 
We say a tropical variety is \textbf{approximable} when it is the limit of
amoebas of a family of complex algebraic varieties in the above sense.
Not all tropical varieties are approximable. 
The first example was given by
  Mikhalkin who constructed
in \cite{Mik1} a spatial elliptic tropical cubic $C$  which is not
 tropically planar: by the Riemann-Roch Theorem 
any classical spatial elliptic cubic
 is planar,  therefore the tropical curve $C$ cannot be approximable. 
One of the challenging problems  in tropical geometry is to
understand which tropical varieties are approximable.
It follows from the works of Viro, Mikhalkin, 
and Rullg{\aa}rd 
(see \cite{V9},
\cite{Mik12}, 
\cite{Rullgard1}, see also \cite{Kap1}) that any 
 tropical hypersurface in $\RR^N$ is approximable. In addition, 
many nice partial results about approximation of tropical curves
 in $\RR^N$ have been proved by different authors (see \cite{Mik1},
 \cite{Mik3}, \cite{Spe2}, \cite{NS}, \cite{Mik08}, \cite{Nishinou},
 \cite{Tyomkin},  \cite{Kat1}, \cite{Br12}, \cite{Br9}).

Tropical varieties in $\RR^N$ 
are 
related to
classical subvarieties of toric
varieties.
When considering non-toric varieties, or when working in  tropical models of the torus different from $\RR^N$,
one is naturally led to  the approximation problem for 
pairs. That is to say, given $X\subset Y$ two tropical varieties in
$\RR^N$, does there exist two families $\X_t\subset \Y_t \subset (\CC^*)^N$ of complex
varieties approximating respectively $X$ and $Y$? 

Non-approximable pairs of tropical objects show up even in very simple
situations. Some well known pathological examples of such
pairs were given by Vigeland, who constructed in
\cite{Vig1} examples of generic non-singular tropical surfaces in $\RR^3$ of
any degree $d\ge 3$ containing infinitely many tropical lines (
called Vigeland
lines throughout this text). Moreover, the surfaces constructed by
Vigeland form an \textit{open} subset of the space of all tropical surfaces of the
given degree $d$, which means that these families of  lines
survive when perturbing the coefficients of a tropical  
equation of the surface.
Vigeland's construction dramatically contrasts with
 Segre's Theorem (see \cite{Segre43})
asserting that  any non-singular complex surface of degree $d\ge 3$ in $\CC
P^3$ can contain only finitely many lines.

\vspace{1ex}

A very important feature in tropical geometry is the so-called
\textbf{initial degeneration} or \textbf{localization} property (see  
{\cite[Chapter 2]{St7}}): let $X$ be a tropical variety approximated by a family of
amoebas $\Log_t (\X_t)$; then given any point $p$ of $X$ we may also 
 produce
from
$\X_t$
 an approximation of $Star_p(X)$ by a \textit{constant}
  family of complex algebraic varieties. Recall that the \textbf{star}
  $Star_p(X)$ 
of $X$ at $p$
is the  
fan composed 
of all vectors $v\in\RR^N$
such that $p+\varepsilon v$ is contained in $X$ for $\varepsilon$
a  small enough
positive real number. 
In other words, any approximable tropical
variety is locally approximable by constant families. 
This already produces non-trivial obstructions to globally
approximating 
tropical
subvarieties 
of dimension and codimension greater than one
in $\RR^N$.
For example,  tropical linear 
fans
are 
in correspondence with 
matroids (see \cite{Spe3}), and it is well known that there exist
 matroids which are not
realisable over any field.

It follows from the same argument used in the case of a 
single tropical variety that a globally approximable pair is locally
approximable by constant families. Therefore there are local
obstructions to globally approximating pairs. 
The main motivation  for this
paper is to provide combinatorial local obstructions
in the case of curves in surfaces.
Results in this direction were previously obtained by the first author
and Mikhalkin (see \cite{Br12} and Theorem \ref{RH} below),  by Bogart and Katz
(see \cite{BogKat}),
and subsequently by Gathman, Schmitz and Winstel (see \cite{GathSchW}). 

Our strategy in this paper
is 
to use the relation between
tropical and complex
intersection theories in order to translate classical results
(e.g. adjunction formula) into combinatorial formulas involving only
tropical 
data.
In the case of stable intersections, 
such a
relation 
has been previously  obtained 
in \cite{Rab1}, \cite{Br16},  \cite{Rab2}, and \cite{KatzInt}. 
Our situation is  reduced to the induced intersections of
\cite{KatzInt} thanks to 
{\cite[Lemma 2.23]{Shaw}} 
saying that any fan
tropical divisor in a matroidal fan can be expressed as the divisor of
a tropical rational function.

\subsection{Overview of the paper}
Let us now describe more precisely the main results of this paper. As mentioned
above, our main goal is
to provide local obstructions to the global approximability of a pair
$(S,C)$ where $C$ is a tropical curve contained in a non-singular
tropical surface $S$. 
By definition, a non-singular tropical variety is
locally a tropical linear space, and a tropical curve is locally a fan. 
This motivates the following
problem.

\begin{ques}\label{ques:approx1}
Let $\P\subset (\CC^*)^N$ be a linear space, 
and 
$C \subset \trp(\P) \subset \RR^N$ be a fan tropical  curve.
Does there exist a complex algebraic curve $\C \subset \P
\subset (\CC^*)^N$ such that $\trp(\C) = C$?
\end{ques}

Precise definitions  of  fan tropical curves and of
 \textbf{tropicalisations}  are given in Section
\ref{definitions}. For the moment, given an algebraic subvariety $\X$
of $(\CC^*)^N$,
one can think of $\trp(\X)$ as $\lim_{t\to \infty}\Log_t(\X)$.

When the answer to Question \ref{ques:approx1} is positive, we say that the tropical curve $C$ is
\textbf{coarsely approximable} in $\P$. If in addition the curve $\C$ is
irreducible and reduced, we say that $C$ is \textbf{finely
approximable}.

As it appears in many works by different authors,  
it is more natural to consider the  approximation 
problem from the point of view of parameterised tropical curves instead
of embedded tropical curves. 
Hence we refine 
 Question \ref{ques:approx1} as follows.

\begin{ques}\label{ques:approx2}
Let $\P\subset (\CC^*)^N$ be a linear space, 
and 
$f:C \to \trp(\P)$ be a tropical morphism from 
an abstract
 fan tropical  curve $C$.
Does there exist a non-singular Riemann surface $\C$ and a proper
algebraic map $\F: \C\to \P$ 
 such that $\trp(\F) = f$?
\end{ques}
When the answer is positive, we say that the tropical morphism $f:C\to
\trp(\P)$ is
\textbf{coarsely approximable} in $\P$. If in addition the morphism $\F$ is
irreducible (i.e. does not factor 
through a holomorphic map of degree at least 2
 between Riemann surfaces $\F': \C\to
\C'$), 
we say that $f$ is \textbf{finely approximable} in $\P$.  

\vspace{2ex}
Our results  provide combinatorial obstructions to the approximation problems
posed in Questions
\ref{ques:approx1} and \ref{ques:approx2}, when $\P$ is a plane. As mentioned above, 
they are based on the
relation
between tropical and complex
intersection theories established in Section \ref{sec:int}.
In particular, in this section we 
define
the tropical intersection product of
fan curves in a 
fan 
tropical plane. 

In Section \ref{sec:adj}, we
combine  Theorem \ref{thm:realiseInt} with the adjunction  formula
for algebraic curves in surfaces.
The following theorem is a weak but easy-to-state version of Theorem
\ref{thm:Adj}.
A plane $\P$ in $(\CC^*)^N$
is called \textbf{uniform} if
its compactification $\overline \P\subset \CC P^N$ as a projective
linear subspace
does not meet any $N-k$-coordinate linear space with $k\ge
3$. 
 The geometric genus of a reduced
algebraic curve $\C$ (i.e. the genus of its normalization) is denoted
by $g(\C)$. Underlying a tropical curve $C$ is a $1$-dimensional polyhedral fan equipped with positive integer weights on the edges. We denote by
$\Ed(C)$ the set of edges of $C$ and $w_e$ the weight of an edge $e \in \Ed(C)$. Each curve $C$ in a uniform plane has a well defined degree which is described in Definition \ref{def:degree}.

\begin{thm}\label{thm:simpadjunction}
Let $\P\subset (\CC^*)^N$ be a uniform 
plane, 
and $C \subset \trp(\P) \subset \RR^N$ be a fan tropical curve of
degree $d$.
If there exists an irreducible and reduced complex curve $\C \subset \P $
such that $\trp(\C)=C$, then 
$$C^2 + (N-2)d - \sum_{e_i \subset \Ed(C)} w_{e_i} + 2 \geq 2g(\C).$$
In particular, if the left hand side is negative then $C$ is not
finely 
approximable in $\P$.  
\end{thm}

In Section \ref{sec:Hess}, we
combine  Theorem \ref{thm:realiseInt} and intersections of a plane
algebraic curve with its Hessian curve. Since the statement is quite
technical, we refer to Theorem \ref{obstruction: hessian} for the
precise details.  As an application we prove Corollary \ref{4-valent},
which 
will be used  in Section \ref{sec:lines} to prohibit all
 4-valent 
Vigeland lines 
in a degree
$d\ge 4$ non-singular tropical surface.

It is worth stressing that 
 Theorem \ref{thm:Adj} is not contained in Theorem \ref{obstruction: hessian} and vice versa. 
  This is not surprising since it is already the case in
 complex geometry: the adjunction formula prohibits the existence of an
 irreducible quartic with 4 nodes while intersection with the Hessian curve
 does not; on the other hand, intersection with the Hessian curve 
 prohibits the
 existence of an 
 irreducible quintic with 6 cusps, while the adjunction formula does not.
Note also that Theorem \ref{thm:Adj} provides an obstruction to
approximate an embedded tropical curve, although Theorem
\ref{obstruction: hessian}  provides an obstruction to approximate a
tropical morphism.

\vspace{1ex}
The last two sections are devoted to
applications of the general obstructions proved in Theorems 
 \ref{thm:Adj} and \ref{obstruction: hessian}.

In Section \ref{sec:aff}, given a non-degenerate plane
$\P\subset(\CC^*)^N$, we classify all 2 or 3-valent fan tropical curves
finely approximable in $\P$ (Theorem \ref{trivalent thm}). 
Here we give two simple instances of this classification.
\begin{thm}\label{plane cycles}
Let $\P\subset(\CC^*)^N$ be a non-degenerate plane, and let $C\subset
\trp(\P)$ be a reduced  2 or 3-valent fan tropical curve. 
\begin{enumerate}
\item If $N=3$, then $C$ is finely approximable in $\P$ if and only if
  $C^2=0$ or $C^2=-1$;
\item If $N\ge 6$ and $C$ is of degree at least 2, then $C$ is not
  approximable in $\P$.
\end{enumerate}
\end{thm}

In Theorem \ref{plane curves}, the finely approximable tropical morphisms
from point (1) 
are described. The intermediate cases $N=4$ and $5$ are
described in Lemma \ref{exceptional conics}. The classification of
degree 1 fan tropical curves is given in Lemma \ref{lines}.

\begin{figure}
\includegraphics[scale=0.3]{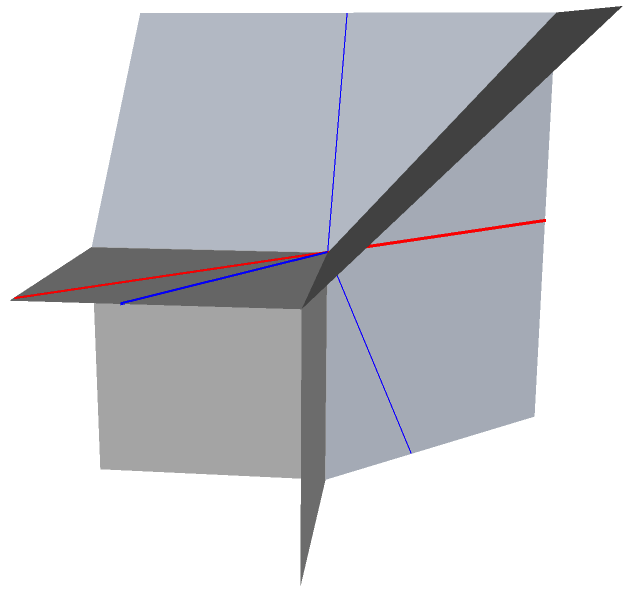}
\hspace{1cm}
 \includegraphics[scale=0.45]{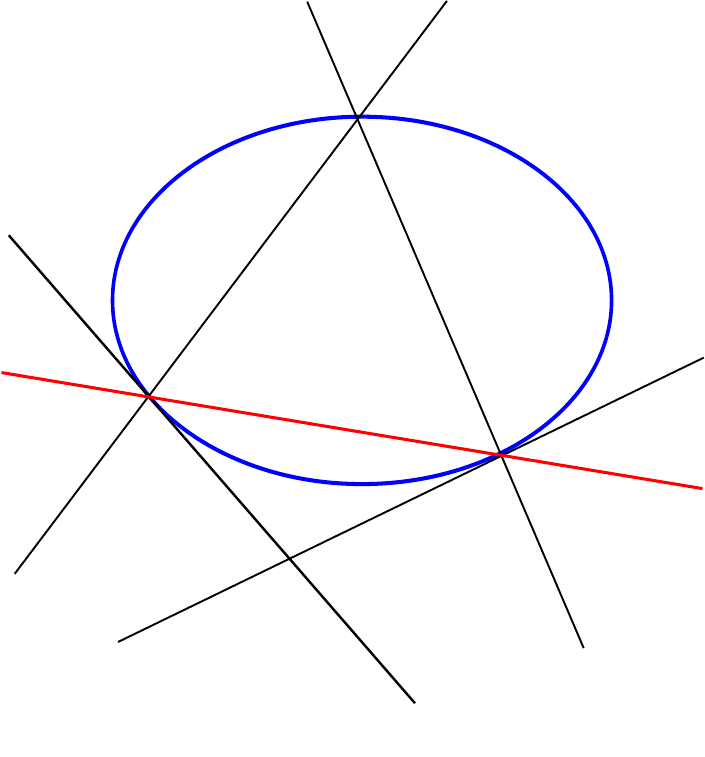}
\put(-255, 93){\small{$\trp(\L)$}}
\put(-315, 115){\small{$\trp(\C)$}}
\put(-300, 150){\small{$d-1$}}
\put(-300, 75){$u_2$}
\put(-210, 180){$u_0$}
\put(-365, 105){$u_1$}
\put(-280, 25){$u_3$}
\put(-210, 55){$P$}
\put(-290, 177){\small{$u_0+u_1$}}
\put(-253, 35){\small{$u_0+ du_3$}}
\put(-348, 78){\small{$u_1+du_2$}}
\put(-212, 110){\small{$u_0+ u_3$}}
\put(-30, 125){$\C$}
\put(-155, 120){$\L_2$}
\put(-158, 50){$\L_1$}
\put(-2, 90){$\L_3$}
\put(-20, 25){$\L_0$}
\put(-85, 75){$\L$}
 \caption{On the left are two tropical curves in the standard plane from Case $(1)$ of Theorem \ref{plane curves}. In red is the $2$-valent curve in the case $d =1$.  On the right are  the  complex line and conic which approximate the tropical curves with their positions drawn relative to the four lines in  $\CC P^2 \backslash \P$  drawn in $\RR P^2$.}
\label{fig:planecurves}
\end{figure}
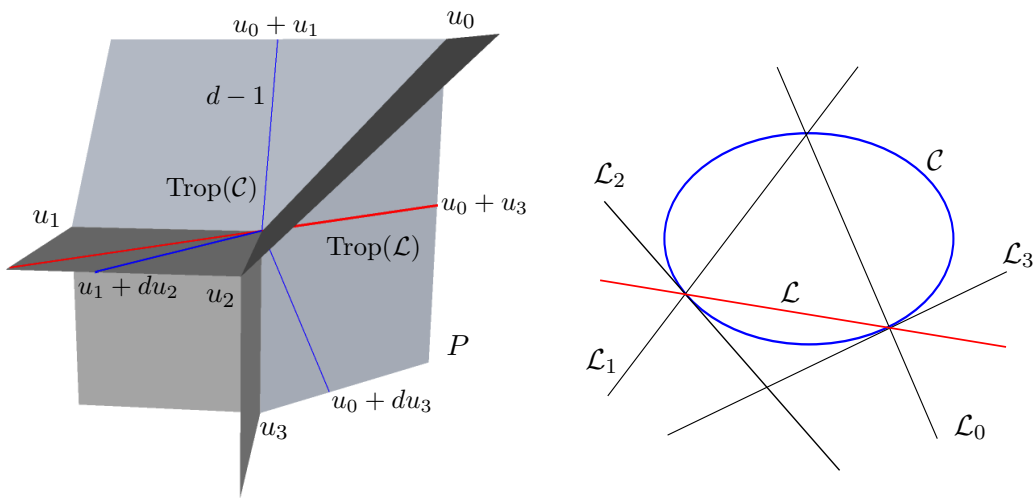
 
\begin{exa}
A fan tropical  curve 
from case $(1)$ of Theorem \ref{plane cycles}
inside the standard tropical plane  
in $\RR^3$ is depicted on the left side of Figure \ref{fig:planecurves}. 
\end{exa}

At this point, it is interesting to 
note
that 
all
fan tropical curves $C\subset\RR^3$  
known 
to us
to be finely approximable
in a plane $\P\subset (\CC^*)^3$ 
satisfy $C^2\ge -1$ in $\trp(\P)$. This leads us to
the following open question.

\begin{ques}
Does there exist 
 a 
fan tropical
curve $C\subset \RR^3$
which is
 finely
approximable in a plane $\P\subset (\CC^*)^3$
and  satisfies
$C^2 \le -2$ in $\trp(\P)$? 
\end{ques}

\begin{rem}\label{ex:minus2}
If the ambient torus has dimension bigger than 3, then such
fan tropical curves
can exist: consider an arrangement of 6 lines
$\L_1,\ldots,\L_6$ 
in $\CC P^2$ such that the 3 points $\L_1\cap\L_2$, $\L_3\cap\L_4$,
$\L_5\cap\L_6$ lie on the same line $\L$; one can embed
$\CC P^2\setminus\{\L_1,\ldots,\L_6\}$ as a plane $\P$ in $(\CC^*)^5$
(see Section
\ref{definitions}), and if we denote $L=\trp(\L)$ we have $L^2=-2$ in
$\trp(\P)$. 
Note that if $\P'\subset (\CC^*)^5$  is the complement in $\CC P^2$ of
 6 lines chosen generically, the tropical 
$-2$-line $L$ is still in $\trp(\P')=\trp(\P)$ 
yet it is no longer approximable in $\P'$. 
See  {\cite[Section 7]{BogKat}}
for another example (though based on the same observation) of 
 approximation problem for pairs $(\trp(\P),C)$ which
cannot be resolved using solely the combinatorial information of $\trp(\P)$. 
\end{rem}

Finally 
in Section  \ref{sec:lines}, 
we apply our methods 
to the study of
tropical lines in tropical surfaces.
In \cite{Vig1} and \cite{Vig2}, Vigeland exhibited generic
non-singular tropical surfaces of degree $d\ge 4$ containing tropical
lines, and  generic
non-singular tropical surfaces of degree $d=3$ containing infinitely
many tropical
lines. The next theorem shows that when we restrict our attention to
the tropical lines which are approximable in the surface,  the situation turns
out to be analogous to the case of complex algebraic surfaces.
This solves the  problem raised in  \cite{Vig1} of generic tropical surfaces 
of degree $d\geq 4$ containing tropical lines, and  tropical surfaces of degree $d=3$
containing infinitely many tropical lines. 
\begin{thm}\label{prohib Vigeland intro}
Let $S$ be a
generic
non-singular 
tropical surface in $\TT P^3$ of degree $d$.
If $d=3$,  
then there exist
finitely many tropical lines $L\subset S$ such that the pair $(S,L)$
is approximable. 

If $d\ge 4$, 
then there exist
no tropical lines $L\subset S$ such that the pair $(S,L)$
is approximable.
\end{thm}

Except in the case of Vigeland lines, all 1-parametric families
contained in a generic non-singular tropical surface $S$ in $\TT P^3$
are \textit{singular} in $S$ (see Section \ref{concluding remark}).
By this we mean that  according to the
adjunction formula a 
line in a surface $S$  of degree $d$ should have
self-intersection $2-d$, which is not the case for such tropical lines. 
 This raises a strange tropical
 phenomenon: $L$ and $S$ are both non-singular in $\RR^3$, but $L$ is
 singular as a subvariety of $S$. In other words, being
 non-singular does not seem to be 
 an 
 intrinsic 
 property of tropical subvarieties.

\begin{rem}
There are obstructions to approximating tropical curves in surfaces that are not local, as the
next two examples show. 
Using Theorem \ref{RH} below and tropical modifications,
 it is proved in \cite{Br12} that  the 
tropical curve $C$ depicted in Figure \ref{global} is approximable in the
tropical hyperplane $P$ by an algebraic curve of genus 0 with 4 punctures
if and only if the vertex of $P$ is the middle of the weight 2 edge of $C$. 
Nevertheless, it is easily verified that this tropical curve is locally approximable in $P$. 

\begin{figure}[h]
\includegraphics[scale=0.02]{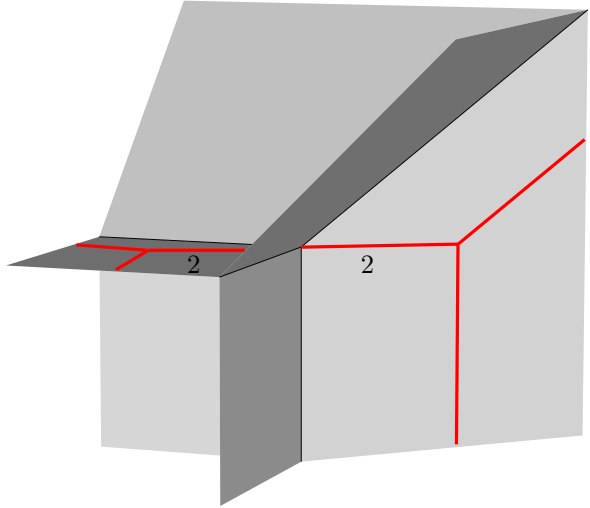}
\put(-150, 88){\small{$2$}}
\put(-85, 88){\small{$2$}}
\caption{The above tropical curve is approximable in $P$ by a rational curve
  with 4 punctures if and only if the vertex of the tropical plane is the midpoint of the edge of weight $2$.}
\label{global}
\end{figure}

Another example of a locally approximable tropical cubic in the
standard tropical plane in $\RR^3$, which is not globally
approximable, even by non-rational cubics,
is depicted in Figure \ref{fig:global2}. The directions to infinity
are
$$ (-2,1,1), \quad (1,-2,1),  \quad (0,0,-1), 
\quad \text{and} \quad (1,1,-1) $$
 The
pair is locally approximable (see Theorem \ref{plane curves}),  but global approximability would
imply the existence of an algebraic cubic in the projective plane
together with a line passing through exactly two of its inflection
points (see Section \ref{sec:related}).

Notice that the two curves  above are both rational.  This contrasts with the fact that any  tropical rational curve in $\RR^N$  is globally approximable (see \cite{Mik3},
\cite{Mik08}, \cite{Br9}, \cite{Spe2}, \cite{NS},

\begin{figure}
\includegraphics[scale=0.02]{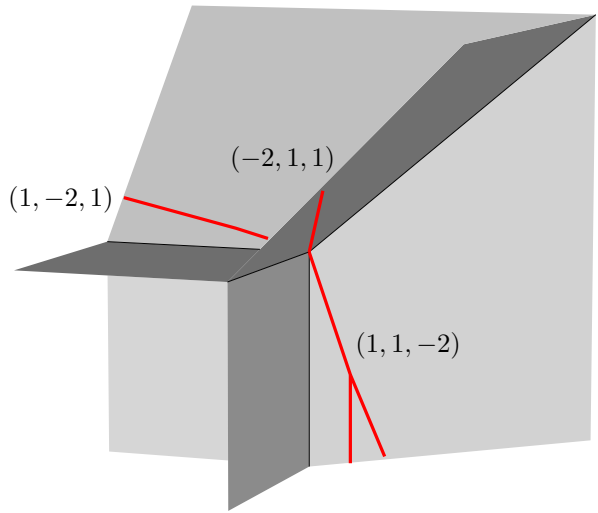}
\put(-90, 60){\small{$(1, 1, -2)$}}
\put(-220, 115){\small{$(1, -2, 1)$}}
\put(-137, 130){\small{$(-2, 1, 1)$}}
\caption{The above tropical curve is not approximable in $P$ by a 
cubic curve since this would contradict the fact that a line passing
through two inflection points of a cubic actually intersects this cubic
in three inflection points.}
\label{fig:global2}
\end{figure}

\end{rem}

\subsection{Related works}\label{sec:related}
In \cite{BogKat}, Bogart and Katz used the relation between stable
intersections in $\RR^n$ and complex intersections to study the
realisation 
problem for pairs in the case of a
trivalent fan tropical  curve $C$ contained 
in the standard tropical hyperplane  in
$\RR^3$. By considering the stable intersection of this latter with the
classical affine plane spanned by $C$, they 
  reduce this situation to the approximation problem for a 
tropical curve
in $\RR^2$ by a reducible complex algebraic curve.
As an application of their method, 
they
proved 
a particular case of
Theorem \ref{prohib Vigeland intro},
where $L$ is a 3-valent Vigeland
line, 
a particular instance of a 1-parameter
family of tropical lines in a non-singular tropical surface.

\vspace{1ex}
  Gathmann,
Schmitz and 
Winstel 
 proved in \cite{GathSchW} some
 general obstructions  when $\P$ is a uniform
plane in $(\CC^*)^3$. These obstructions are of a 
different nature from the ones presented here, and it
seems worthwhile  to stress the differences between the two approaches.

The approximation 
 of pairs in the case  of a fan tropical curve in a
tropical plane is by definition equivalent to the following problem:
consider a line arrangement $\A=\{ \L_0,\ldots,\L_N\}$ in $\CC P^2$;
does there exists a complex algebraic
curve $\C$  of degree $d$ in $\CC P^2$ such that 
for all lines $\L_1$ and $\L_2$ in $\A$, the curve $\C$ has a prescribed 
local Newton polygon (i.e. the polygon $\Gamma^c(\C)$ defined at the beginning of Section \ref{sec:Hess}) at $\L_1\cap\L_2$ in some affine chart of $\CC P^2$  for which 
both $\L_1$ and $\L_2$ are coordinate axes?

One easily  reduces this problem to the
study of a system  $(E)$ of linear equations and inequations in the
coefficients of the hypothetic complex algebraic curve. Hence 
 given  a specific plane $\P\in(\CC^*)^N$ and a specific tropical curve
$C\subset\trp(\P)$, 
the approximation problem of the pair $(\trp(\P),C)$ is, in
principle, solvable. 
As an illustration, \cite{GathSchW}  classified all realisable tropical curves of low degree in the standard tropical plane in $\RR^3$ by solving the corresponding finitely many systems $(E)$.
Note
that as soon as $N\ge 4$, 
there are moduli in line arrangements in $\CC P^2$
giving rise to the same tropical plane,
so the  system 
$(E)$  depends 
on $\A$
and not only on $\trp(P)$ (see Remark \ref{ex:minus2}).

The  obstructions proved in this paper  identify
tropical curves whose approximation would carry too many
singularities: standard tools from algebraic geometry are sometimes
sufficient to ensure that the above mentionned system $(E)$ contains
more independent equations than variables, and hence has no 
solutions. 
We point out that
these obstructions are valid for any $N$ and
only depend on $\trp(\P)$.

However, it may happen that the pair $(\trp(\P),C)$ is not
approximable because the set of equations of $(E)$ and the set of
inequations are dependant. This second kind of obstructions are 
more subtle to understand than the previous ones, and definitely deserve attention.
In particular, studying those non-transversality issues, one cannot only remember  $\trp(P)$ when there are 
moduli in $\A$.
The following classical statement is an example of 
such a non-transversality:
  given a non-singular cubic curve
$\C$ in $\CC P^2$ and a line $\L$ passing through two inflexion points
of $\C$, then the third intersection point of $\C$ and $\L$ is also an
inflexion point of $\C$. The tropical version of this 
 statement is 
 that the spatial tropical fan curve of degree 3 whose directions to
 infinity are
$$(-2,1,1), \quad (1,-2,1), \quad \text{and} \quad (1,1,-2) $$
is approximable in
 the standard tropical
plane in $\RR^3$, while  the spatial tropical fan curve of degree 3 whose 
 directions to infinity are 
$$ (-2,1,1), \quad (1,-2,1),  \quad (0,0,-1), 
\quad \text{and} \quad (1,1,-1) $$
is not.
In addition to their classification in low degree,
Gathmann
Schmitz and 
Winstel focused in \cite{GathSchW} on  the study of
non-transversality issues in the system $(E)$ when 
$\A$ is a generic arrangement of $N+1=4$ lines in $\CC P^2$. 
Therefore among other examples, their methods could generalize the previous 
observation about cubics to curves of any degree.

In conclusion, the present paper and \cite{GathSchW} study two
different kinds of general obstructions to the approximation of pairs, 
and results proved in
both papers barely 
overlap (except for a very particular case of Corollary
\ref{cor:CD<0}). 
In addition, we also address in this paper
the approximation of a
tropical morphism to a tropical plane, which is no longer
linear. In particular, we do not see how to prove
 Theorem \ref{obstruction: hessian} 
 with the methods from \cite{GathSchW}.
Altogether, it is not clear to us that,
quoting {\cite[Abstract]{GathSchW}},
\cite{GathSchW} includes and 
generalizes the main  obstructions presented here, and
quoting an anonymous referee report
on a previous version of this text, 
\cite{GathSchW} recovers all of
the obstructions presented here as special cases.

\vspace{1ex}
Lastly, the techniques presented
in this paper should  generalise to the
study of tropical curves
or morphisms in higher dimensional tropical varieties. At present,
 this problem is widely unexplored. Up to our knowledge, the
following tropical Riemann-Hurwitz condition
 is the only general 
obstruction to the approximation of a tropical morphism to a tropical variety.
Recall that $g(\C)$ denotes the geometric genus of
 a reduced
algebraic curve $\C$.

\begin{thm}[Brugall\'e-Mikahlkin see \cite{Br12},
   \cite{Br13}, or \cite{Br18}]\label{RH}
Let $\P\subset (\CC^*)^N$ be a  linear space of dimension $n$
such that $\trp(\P)$ is composed of $k$ faces of dimension $n$ and one
face $F$ of dimension $n-1$. Let $f:C\to \trp(\P)$ be a tropical
morphism from a fan tropical curve $C$, let $d$ be the tropical
intersection number of $f(C)$ and $F$ in $\trp(\P)$, and suppose that $C$ has
exactly $l$ edges which are not mapped entirely in $F$. Then if $f:C\to
\trp(\P)$ is coarsely approximable by an algebraic map $\F:\C\to\P$,
one has
$$g(\C)\ge \frac{d(k-2) - l +2}{2} .$$
\end{thm}

\vspace{2ex}

\noindent \textbf{Acknowledgment.} This  work  was started during a common invitation
of both authors to the Universidad Nacional Aut\'onoma de Mexico in
december 2010
(Instituto de Matem\'aticas. Unidad Cuernavaca. UNAM-PAPIIT IN117110). We thank this
institution for the  support and excellent working conditions 
 they provided us. We 
 especially
 thank Lucia  L\'opez de Medrano for her
 kind invitation 
to
Cuernavaca and Guanajuato. We are also grateful to Assia
 Mahboubi and Nicolas Puignau who drove the car, and to Jean-Jacques
 Risler and Erendira Munguia Villanueva for their judicious comments on a preliminary version of the text.
E.B. is 
partially supported by the ANR-09-BLAN-0039-01 and ANR-09-JCJC-0097-01.

\vspace{0.5cm}

\section{Preliminaries}\label{definitions}
In this section we recall some well known facts in order to make the paper
self-contained, and to introduce notations used in the following.

\subsection{Linear spaces and their tropicalisations}\label{sec:linspaces}

A \textbf{linear space} $\P$  in $(\CC^*)^N$ is a
subvariety which is given, up to 
the action of
$Aut\left( (\CC^*)^N\right)
=Gl_N(\Z)$, by a system
of equations of degree $1$.
Equivalently, a subvariety $\P\subset (\CC^*)^N$ is a
linear space
if it is given by a system of equations with support 
$\Delta(\P)$
contained in
a primitive simplex $\Delta$.  
Here, \textit{primitive} indicates that $\Delta$ has the same volume as
the standard simplex in $\RR^N$. 
Such a  simplex $\Delta$ induces  a toric compactification 
of $(\CC^*)^N$ to $\CC
P^N$ such that $\P$ compactifies to a projective linear subspace
$\overline \P = \CC P^r$; the simplex $\Delta$ will be said to give a
degree $1$ compactification of $\P \subset (\CC^*)^N$. Note that
neither $\Delta(\P)$ nor $\Delta$ may be uniquely chosen (see Examples 
\ref{ex:compdeg2} and \ref{ex:cremona}).
However once $\Delta(\P)$ is
 fixed, any choice of $\Delta$ produces the same
pair $(\CC P^N,\overline \P)$ up to toric isomorphism.
  
A linear space $\P\subset (\CC^*)^N$ is said to be \textbf{non-degenerate}
if it is not contained in any translation of a strict sub-torus of
$(\CC^*)^N$. 
If $\P \subset (\CC^*)^N$ is
a non-degenerate linear space, 
then for any choice of defining equations we must have $\dim ( \Delta(\P) ) \geq N+1-\dim(\P)$.  
A \textbf{plane} is a 
non-degenerate 
linear space of dimension 2, and in this case 
$\dim ( \Delta(\P) ) \geq N-1$.

\begin{defi}
The tropicalisation of a linear space $\P$  in $(\CC^*)^N$, denoted by
$\trp(\P)$ 
and called a tropical linear fan, 
is defined as
$\lim_{t \to \infty} \Log_t(\P)$. 
\end{defi}
We  say that $\P$ approximates $\trp(\P)$. 
The tropicalisation $\trp(\P)$ is a rational polyhedral fan of
pure dimension 
$\dim \P$. 
Moreover, it is
naturally equipped with a constant weight function equal to 
 1 on each face of maximal dimension
making 
 $\trp(\P)$ into a balanced polyhedral fan (see \cite{St7} and
\cite{Ard}).  
Note that if $\dim (\Delta(\P) ) = N-k$ then the
fan $\trp(\P)$ contains an affine space of dimension $k$.

The above mentioned  degree $1$ compactification 
$\overline \P = \CC P^r$  of a  
linear space
$\P\subset (\CC^*)^N$ by way of a simplex $\Delta$ 
defines a 
hyperplane
arrangement $\A=\overline \P\setminus \P$  in 
$\overline \P=\CC P^r$. 
 Two hyperplane arrangements $\A$ and $\A^{\prime}$ in $\CC P^r$ are
isomorphic if there is an automorphism $\phi \in \text{Aut}(\CC P^r)$
inducing a bijection $\A\to\A'$.
When there is a choice in the simplex $\Delta(\P)$, 
the arrangements defined by the various compactifications may not be
isomorphic.  
However in the case of planes one can easily describe all those different
compactifications: if $(\overline \P,\A)$ and 
$(\overline\P',\A')$ are two non-isomorphic compactifications of
$\P$, then $(\overline\P',\A')$ is obtained from $(\overline \P,\A)$
by a standard quadratic transformation $\sigma$ of $\overline \P=\CC
P^2$ such
that any line of $\A$ passes through at least one of the three
base points of $\sigma$. In particular, the arrangements $\A$ and
$\A'$ are
combinatorially isomorphic,  meaning they determine the same matroid.
Also, the pair $(\P,\A)$ is unique when $\P$ is a
uniform plane.

\begin{figure}[b]
\includegraphics[scale = 0.3]{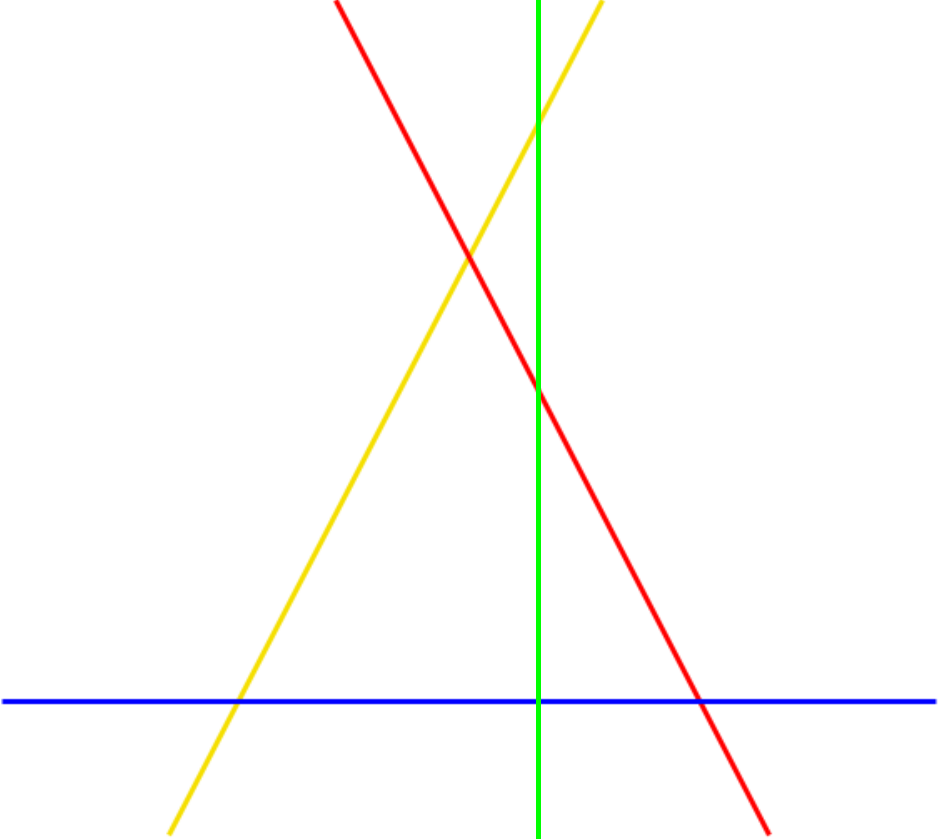}
\hspace{2cm}
\includegraphics[scale = 0.3]{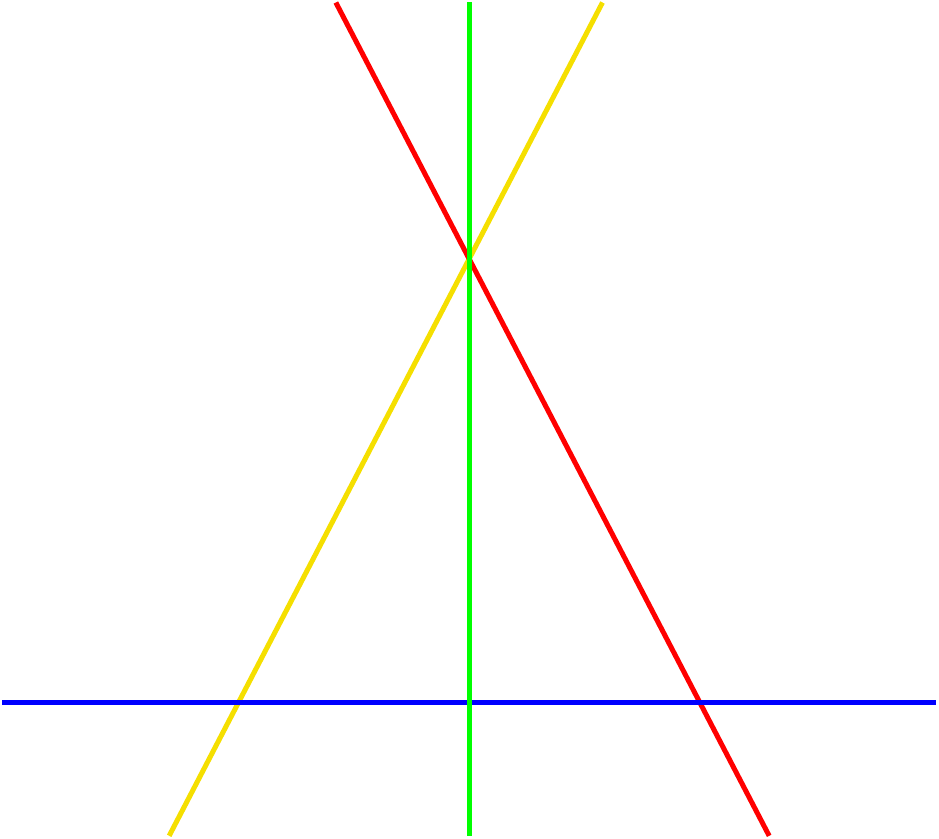}
\put(-60, 80){$\textbf{p}_{i, j, k}$}
\caption{The line arrangements  corresponding to the 
 two types of non-degenerate planes in $(\CC^*)^3$} 
\label{fig:someplanes2}
\end{figure}

\begin{example}\label{ex:compdeg2}  
Up to a change of coordinates, there exist only two non-degenerate
planes in $(\CC^*)^3$.
\begin{itemize}
\item The plane with equation $z_1+z_2
+ z_3 +1 = 0$. The corresponding line arrangement is depicted on the
left of Figure \ref{fig:someplanes2}.
\item The plane $\P \subset (\CC^*)^3$ with equation $z_2
+ z_3 +1 = 0$. 
The corresponding  line arrangement is drawn on 
the right of Figure \ref{fig:someplanes2} and $\trp(\P) \subset \R^3$
contains the 
affine line in direction $(1, 0, 0)$.  
The support of $\P$ is  $$\Delta(\P) = Conv\{(0, 0, 0), (0, 1, 0), (0,
0, 1)\} \subset \RR^3.$$ 
The two different simplicies  
$$\Delta = Conv\{(0, 0, 0), (0, 1, 0), (0, 0, 1), (1, 0, 0)\}, $$
and
$$ \Delta^{\prime} =  Conv\{(0, 0, 0), (0, 1, 0), (0, 0, 1), (-1, 0,
0)\}$$
define two degree one compactifications of $\P\subset (\CC^*)^3$, and
the map $(z_1,z_2,z_3)\mapsto (\frac{1}{z_1},z_2,z_3)$ induces a toric
isomorphism between $(\overline P,\A)$ and  $(\overline P',\A')$.
\end{itemize} 
\end{example}

\begin{example}\label{ex:cremona}
Consider the plane $\P \subset (\CC^*)^4$ defined by the two degree
$1$  equations 
$$ x_3 - x_1 - x_2 = 0 \qquad \text{and}  \qquad x_4 - x_1  - x_2  - 1
= 0.$$ 
The polytope $\Delta$ determined by
this system of equations is the standard simplex in $\R^4$, and the
line arrangement obtained by compactifying by way of $\Delta$ is shown
on the 
right of Figure \ref{fig:VinR4}. 

The plane $\P$ may also be defined by the system of equations:
$$x_3 - x_1 - x_2 = 0 \qquad \text{and} \qquad
x_1x_4 - x_1x_3 - x_3 - x_2 = 0.$$
The support of this system is a polytope $\Delta^{\prime} \subset
\R^4$ different from the standard simplex but which is also primitive.   
\end{example}
Conversely, a 
hyperplane
 arrangement  
$\A = \{\H_0, \dots , \H_{N}\}$ in $\CC P^r$ 
satisfying 
$\cap_{i=0}^N \H_i = \emptyset$ 
defines 
an embedding 
$$
\begin{array}{cccc}
\phi_\A: & \CC P^r & \longrightarrow & \CC P^N \\
  & z  &\longmapsto &[f_0(z):  \dots : f_N(z) ]
\end{array}
$$
where $f_i$ is a linear form defining the hyperplane $\H_i$. 
Up to a rescaling of each coordinate in $\CC
P^N$, the map $\phi_\A$
depends only on $\A$. 
The  
linear space
$\P = \phi_{\A}(\CP^r) \cap (\CC^*)^N$ is 
non-degenerate and is
the complement $\CC P^r \backslash \A$ embedded in $(\CC^*)^N$. 
A 
hyperplane
arrangement $\A$ in $\CC P^r$ is \textbf{uniform} if 
any $m$ hyperplanes in $\A$ intersect in a codimension $m$ linear space.
So we say a
linear space in $(\CC^*)^N$
is uniform if 
its corresponding 
 hyperplane
 arrangement is. 
In this text, all line arrangements 
$\A = \{\L_0, \dots , \L_{N}\}$
are assumed to contain a
uniform sub-arrangement  of three lines, i.e. $\cap_{i=0}^N \L_i = \emptyset$.

\vspace{1ex}
The tropicalisation  $\trp(\P)$ of a 
linear space
$\P\subset (\CC^*)^N$
 is  the Bergman fan of
the matroid corresponding to $\P$, and has a very
nice combinatorial construction 
 as described in \cite{Ard}. We recall  now this construction
in the case when $\P$ is a \textbf{plane}.
 Without
 loss of generality we may assume that $\P$ is non-degenerate.

As previously, 
let us consider the degree 1 toric compactification of $(\CC^*)^N$ in $\CC
P^N$ defined by a simplex
$\Delta$. 
Denote by $u_0,\ldots,u_N\in\ZZ^N$ the outward primitive integer normal
  vectors  to the faces of $\Delta$. 
There is a natural correspondence between the vectors $u_i$, the
hyperplanes  of $\CC P^N\setminus (\CC^*)^N$, and
the lines in the arrangement $\A=\overline\P\setminus
\P$. In particular we can write $\A=\{\L_0,\ldots,\L_N\}$ where $\L_i$ lies in
the coordinate hyperplane corresponding to $u_i$.
A \textbf{point} of an arrangement $\A$  is a point in $\overline \P$ 
 contained in
at least two lines of $\A$. To a point of $\A$ we  associate  the
maximal  subset $I \subset \{0, \dots , N\}$ such that $\textbf{p}=
\cap_{i \in I } \L_i$, as well as the vector
$u_I = \sum_{i \in I} u_i$. Thus we may denote a point of $\A$ by $\textbf{p}_I$, 
and denote the set of points of $\A$ by $\textbf{p}(\A)$.
From the construction of the Bergman fan in 
\cite{Ard},  
as a set $\trp(\P)$ is the union  of all cones
 $$\{\lambda u_i+ \mu u_I \ | \ \lambda, \mu \in \RR_{\geq0} \}$$ 
where $i$ is contained in  $I$ and  $I \subset \{0, \dots, N\}$ 
corresponds  to a  point of $\A$. 
An \textbf{edge} of $\trp(\P)$ is an edge of the coarse
polyhedral structure on $\trp(\P)$ (i.e. it is a ray made of points
where  $\trp(\P)$ is not locally homeomorphic to $\RR^2$), see \cite{Ard}.
 The coarse polyhedral structure may be obtained from the Bergman fan structure
by removing all rays in the direction $u_i +u_j$   corresponding to points 
 $\textbf{p}_{i, j}$ and  all rays in the direction $u_k$ for all $\L_k$ which contain
 only  two points $\textbf{p}_I$, $\textbf{p}_J$ of the arrangement.  
In particular
the set of edges
in the coarse polyhedral structure on $\trp(\P)$ is contained in
$\{u_0,\ldots,u_N,u_I\ | \ p_I\in \textbf{p}(\A)\}$, 
this inclusion is usually strict. The tropicalisation $\trp(\P)$ 
does not depend on the choice of polytope $\Delta$ used to 
compactify $(\CC^*)^N$, therefore neither does this
coarse polyhedral structure. However, the fine polyhedral structure on
$\trp(\P)$ does depend on the choice of $\Delta$, when  a 
choice exists.

\begin{figure}[b]
\includegraphics[scale = 0.25]{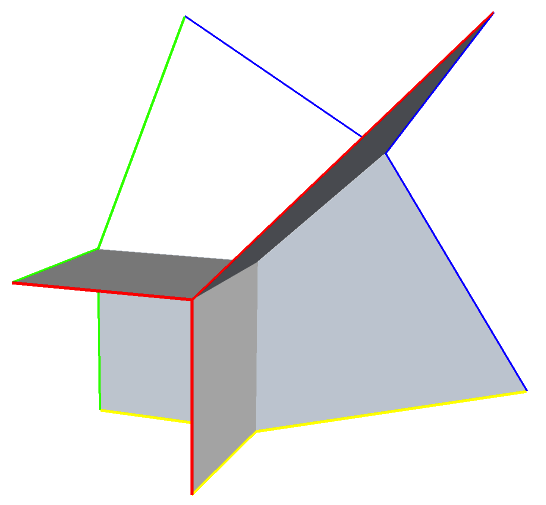}
\hspace{2cm}
\includegraphics[scale = 0.33]{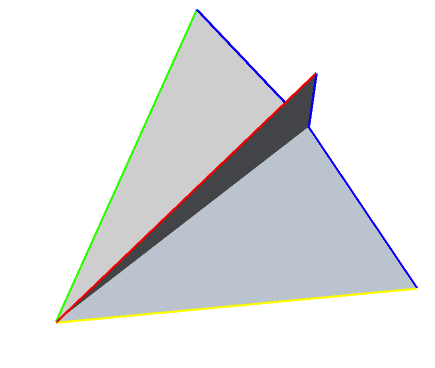}
\put(-135, 7){$p_{i, j, k}$}
\caption{The compactifications of two tropical planes in $\TT P^3$. On the right there is a corner point $p_{i, j,k}$ corresponding to a triple of lines. 
\label{fig:someplanes} }
\end{figure}

It follows from this construction that
 the tropical fan $\trp(\P)$ depends  
only on the intersection lattice of the arrangement
$\A$.
Thus, non-isomorphic line arrangements on $\CC P^2$ 
may have the same tropicalisations.  This 
leads to the phenomena explained in Remark \ref{ex:minus2} and
in \cite[section 7]{BogKat},
where the approximation problem of a tropical curve in $\trp(\P)$
depends on $\P$ and not just on $\trp(\P)$.

\vspace{1ex}
 By declaring $\Log(0) = -\infty$ we can extend the map $\trp$
 continuously to varieties in $\CC^N$ and even in $\CC P^N$. The images of
 $\CC^N$ and $\CC P^N$ under the extended $\Log$ map being respectively
 tropical affine space, $\T^N = [-\infty, \infty)^N$, and tropical projective space $\TT P^N$. 
 Here we will use tropical projective space as it appears in  \cite{Mik3}. This space is compact and 
is obtained in
   accordance with classical geometry.     
 It is equipped with tropical homogeneous coordinates
  $$[x_0: \dots : x_N] \sim [x_0 +a : \dots : x_N +a]$$ where $x_i \not = -\infty $ for at least one $0\leq i \leq N$, and $a \in \R$. Moreover, it is covered by the affine charts 
 $$U_i = \{ [x_0: \dots : x_N] \ | \ x_i = 0 \}  \cong \T^N.$$
  Given a plane $\P \subset (\CC^*)^N$, 
  and a degree $1$ compactification $\overline{\P} = \CC P^2$ given by a simplex $\Delta$, 
  the tropicalisation $P = \trp(\P) \subset \R^N$ may be compactified to $\overline P \subset \TT P^N$, and by continuity
$\trp(\overline{\P}) = \overline P$. 
 Clearly, any line $\L_i\in\overline\P\setminus\P$ tropicalises to a
boundary component $L_i$ of $\overline P$, and 
$\overline P\setminus P=\bigcup L_i$.

 \begin{figure}
 \includegraphics[scale=0.2]{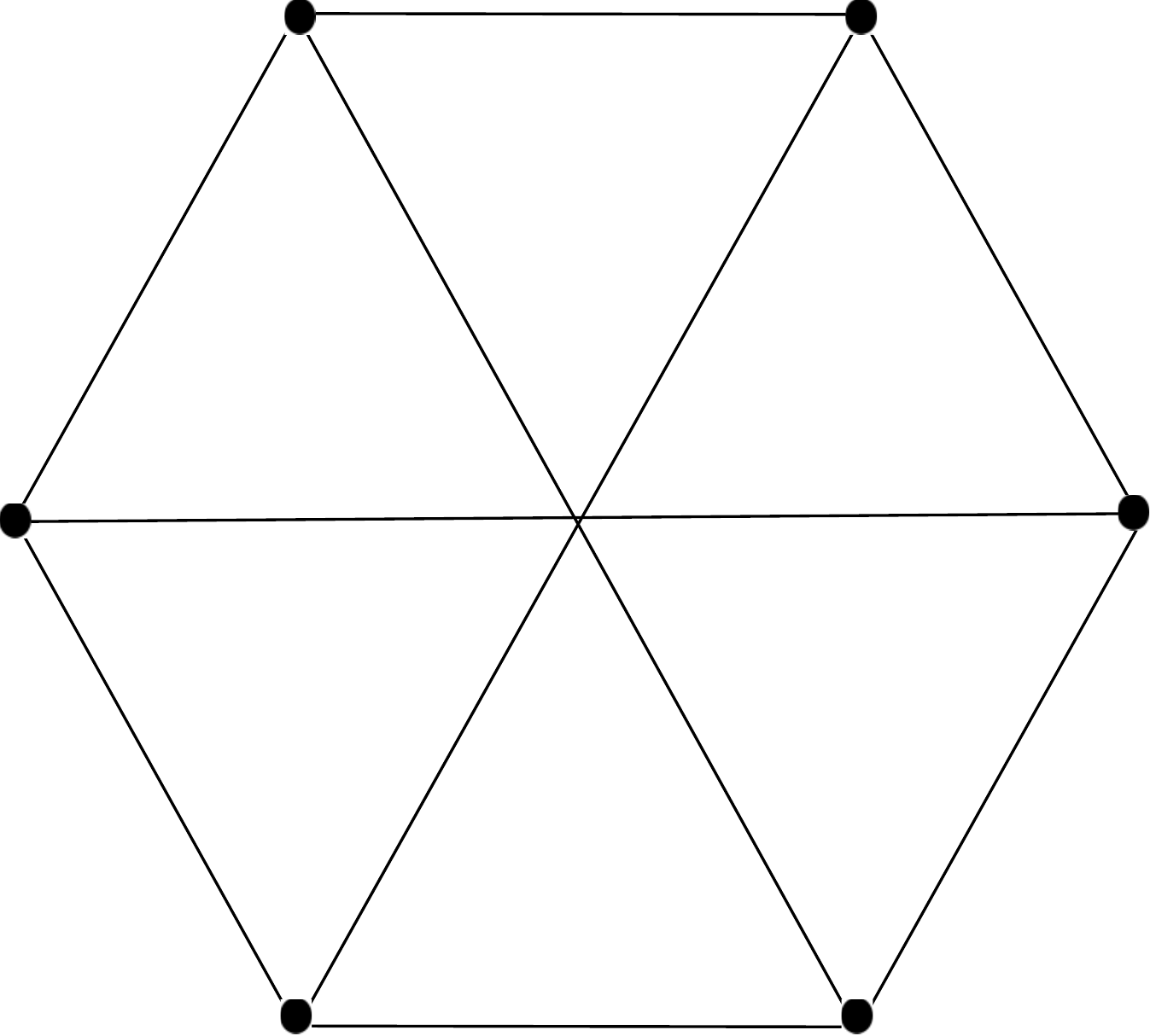}
 \hspace{2cm}
 \includegraphics[scale = 0.45]{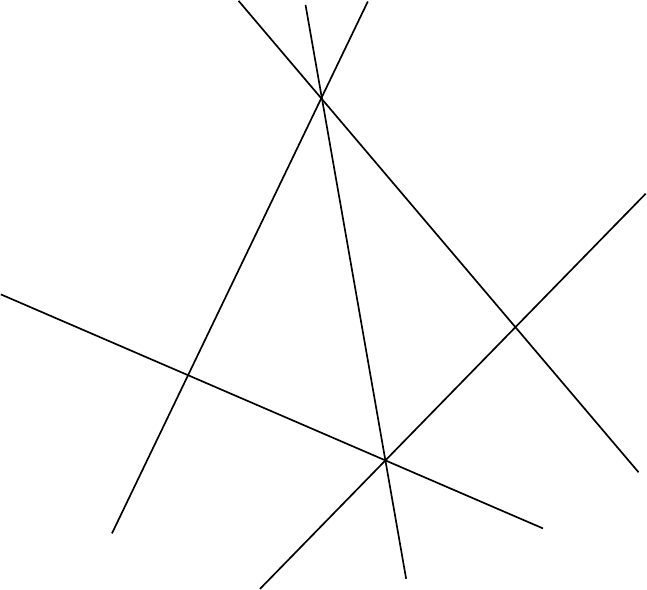}
 \put(-250, 120){$u_0+u_3+u_4$}
  \put(-240, -7){$u_1$}
 \put(-200, 55){$u_0$}
 \put(-305, 120){$u_3$}
 \put(-345, 55){$u_2$}
   \put(-330, -7){$u_1+u_2+u_4$}
\put(-105, 120){$\L_2$}
\put(-130, 12){$\L_1$}
\put(-2, 90){$\L_3$}
\put(-20, 12){$\L_0$}
\put(-78, 60){$\L_4$}
 \caption{The link of singularity of a tropical plane  in 
$\RR^4$
and the corresponding line arrangement.}
 \label{fig:VinR4}
 \end{figure} 
 \begin{figure}
 \includegraphics[scale=0.2]{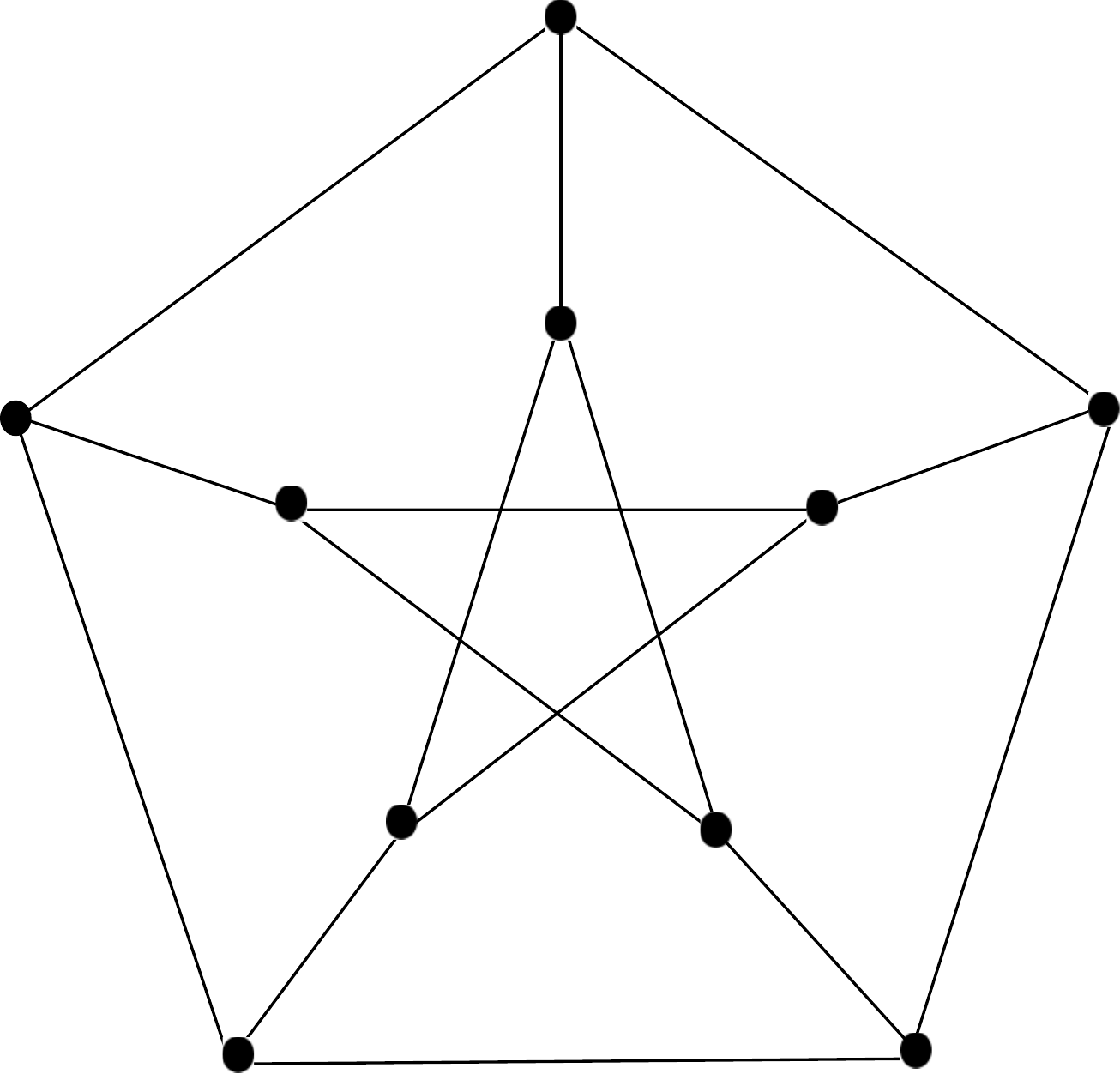}
 \hspace{2cm}
 \includegraphics[scale= 0.65]{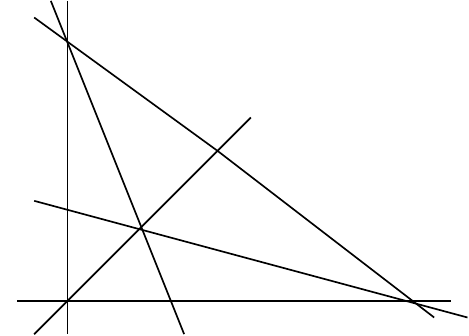}
 \caption{The Petersen graph is the link of singularity of 
   $\trp(\M_{0,5}) \subset \RR^5$, it corresponds to  the braid arrangement
  of lines drawn
 on the right,  the directions of the rays are 
   omitted since they will not be used here. Since the  line
   arrangement is  symmetric we choose not to label it. } 
 \label{fig:M05}
 \end{figure}
 
\begin{exa}
Figure \ref{fig:someplanes} shows the compactifications in $\TT P^3$ of 
the 
standard tropical plane in $\R^3$, and another plane with only three
  faces. 
The first plane corresponds to the complement of a uniform
arrangement of four lines in  
 $\CC P^2$ whereas the second plane
 corresponds to the arrangement where three of the four lines belong to 
 the same pencil (see Figure \ref{fig:someplanes2}).  
 \end{exa}

\begin{example}
Figures \ref{fig:VinR4} and \ref{fig:M05} are again examples of tropical
planes and their corresponding arrangements. In both figures the
graphs on the left are the link of singularity of the  
tropical fan, so that the fans are the cones over the graphs. The rays
of the coarse polyhedral structure on the fans correspond to the thick vertices. 
The line arrangements corresponding to the two planes are drawn on the
right in each figure.  The arrangement in Figure  \ref{fig:M05} is
known as the Braid arrangement. 
The complement of this arrangement in $\CC P^2$ is $\M_{0,5}$,  the
moduli space of complex rational curves with $5$ marked points.  In
\cite{Ard}, this is also shown to be the moduli space of $5$-marked 
rational
tropical curves.  
Since both arrangements are determined by a generic configuration of
four points in $\CC P^2$, in each case there
 is a unique arrangement up to automorphism of $\CC P^2$ with the
corresponding intersection lattice.
\end{example}

To a point $\textbf{p}_I$ of $\textbf{p}(\A)$ corresponds a point
$p_I=\trp(\textbf{p}_I)$ in  
 $\overline P$, called a \textbf{corner} of $\overline P$. 
Locally at such 
a point $p_I$ the tropical plane $\overline P$  is determined by $|I|$. 
If $|I| = 2$,  a neighborhood of $p_I \in \overline{P}$ is a neighborhood of $(-\infty, -\infty) \in \T^2$. If $|I| = k >2$, a
neighborhood of $p_I$  is the cone over a $k$-valent  vertex,
with $p_I$ the vertex of the cone, see 
the left side of Figure \ref{fig:corners}.
At a point  $\textbf{p}_I \in \overline{\P} \subset \CC P^N $ we may choose an affine chart $\U_l \cong \CC^N$ for $l \not \in I$, then for any pair $i, j \in I$ define the projection $\pi_{i,j}: \U_l \longrightarrow \CC^2$ by $(z_0, \dots , z_{l-1}, z_{l+1}, \dots z_N) \mapsto (z_i,z_j)$.  
We use the same notation $\pi_{i,j}$ for the induced projection on
$U_l \subset \TT P^N$ since the projection commutes with
tropicalisation. Then $\pi_{i,j}(\textbf{p}_I) = (0, 0)$, and of
course $\pi_{i,j}(p_I) = (-\infty, -\infty)$. Throughout the article
these projections will be applied in order  to work locally over
$\CC^2$
and $\T^2$.

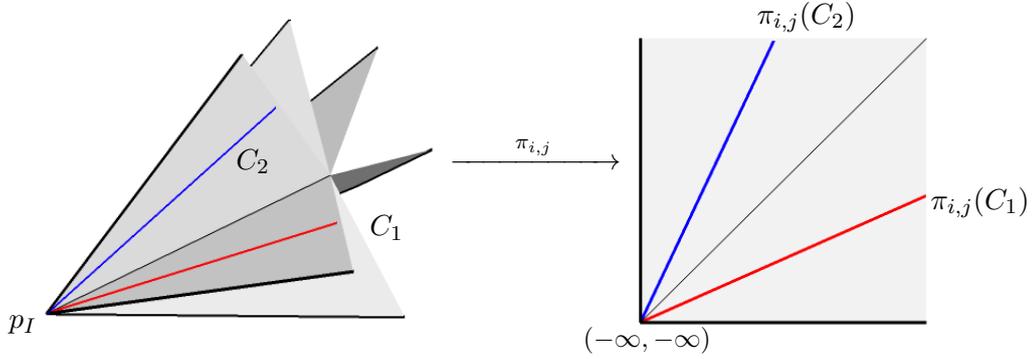
\begin{figure}
\includegraphics[scale =0.4]{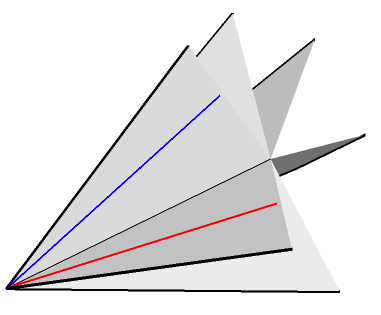}
\hspace{2cm}
\includegraphics[scale =0.38]{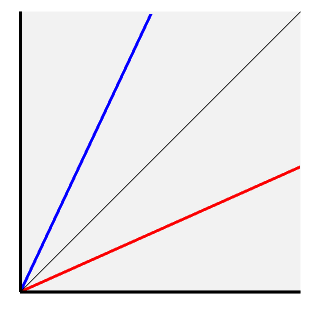}
\put(-350, 5){$p_I$}
\put(-185, 65){$\xrightarrow{\quad \quad \pi_{i, j} \quad \quad }$} 
\put(-215, 40){$C_1$}
\put(-265, 65){$C_2$}
\put(-135, -2){\small{$(-\infty, -\infty)$}}
\put(-5, 50){$\pi_{i, j}(C_1)$}
\put(-70, 120){$\pi_{i, j}(C_2)$}
\caption{On the left is a  neighborhood of a corner point $p_I \subset
  \overline P$ for $|I| = 6$, along with rays of two tropical curves 
$\overline C_1,\overline C_2\subset \overline P$ passing through
$p_I$.  
On the left is the image under the projection $\pi_{i,j}: U_l
\longrightarrow \T^2$.}
\label{fig:corners}
\end{figure}

\subsection{Tropical curves and morphisms}\label{section:definitions curves}
Since we are mainly concerned with the local situation in this paper, for
the sake of simplicity we restrict ourselves to \textit{fan}
tropical curves.
Given a $1$-dimensional polyhedral fan $C \subset \R^N$, we denote by $\Ed(C)$ the set of its edges.

\begin{definition}\label{def:fancurve}
A fan tropical curve $C \subset \RR^N$ is a $1$-dimensional  rational polyhedral
fan such that  each edge $e \in \Ed(C)$ is equipped with positive integer weight $w_e$ 
and 
satisfying the balancing condition
$$\sum_{e \in \Ed(C)} w_e v_e = 0$$ 
where $v_e$ is the 
 primitive
integer direction of $e$ 
pointing away from the vertex of $C$.
\end{definition}
Note that contrary to 
some other
definitions of tropical curves, the vertex
of a fan tropical curve can be 2-valent.

An algebraic curve $\C \subset (\CC^*)^N$ produces a fan tropical
curve $ \trp(\C) \subset \R^N$, the support of $\trp(\C)$ being
$\lim_{t \to \infty} \Log_t(\C)$. In order to define the weights 
on edges
of
$\trp(\C)$, consider any toric compactification $\X$ of $(\CC^*)^N$
such that the closure $\overline \C$ of $\C$ in $\X$ does not intersect
any $\D_i\cap\D_j$ where $\D_i$ and $\D_j$ are any two toric divisors
of $\X$. Such a compactification of 
$(\CC^*)^N$
 will be called  \textbf{compatible}
with $\C$.
For
every edge $e \subset \trp(\C)$ there is a corresponding toric
divisor $\D_e \subset \X$. The curve
$\overline{\C}$ intersects $\D_e$ in a finite number of points and we
define the weight $w_e$ of $e$ as the sum of the multiplicities of
these intersection points.
The fan $\lim_{t \to \infty} \Log_t(\C)$  enhanced with the
weights $w_e$ on its edges is a tropical curve $\trp(\C)$.

\begin{definition}
The tropicalisation of a curve  $\C \subset (\CC^*)^N$ is the tropical
curve $\trp(\C)$.
\end{definition}

\begin{exa}\label{cusp cubic}
The tropicalisaton $C$ of the curve $\C$ in $(\CC^*)^2$ with equation 
$-Y^2 -X^3 + 4 X^2Y -5XY^2 + 2Y^3=0$   is depicted in Figure \ref{fig:cusp in R2}. The
tropical curve $C$ has 3 rays $e_1,e_2,$ and $e_3$ with
$$w_{e_1}v_{e_1}=(-2,-3),\quad  w_{e_2}v_{e_2}=(-1,0),\quad
\text{and}\quad
w_{e_3}v_{e_3}=(3,3). $$
Note
that the presence of the edge $e_1$ of $C$  is
equivalent to the fact that the compactification of the curve $\C$ in
$\CC P^2$ has a cusp at $[0:0:1]$.
\end{exa}

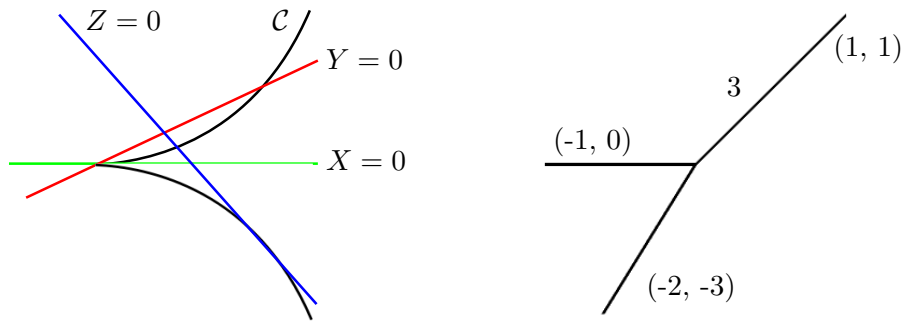
\begin{figure}
\includegraphics[scale=0.4]{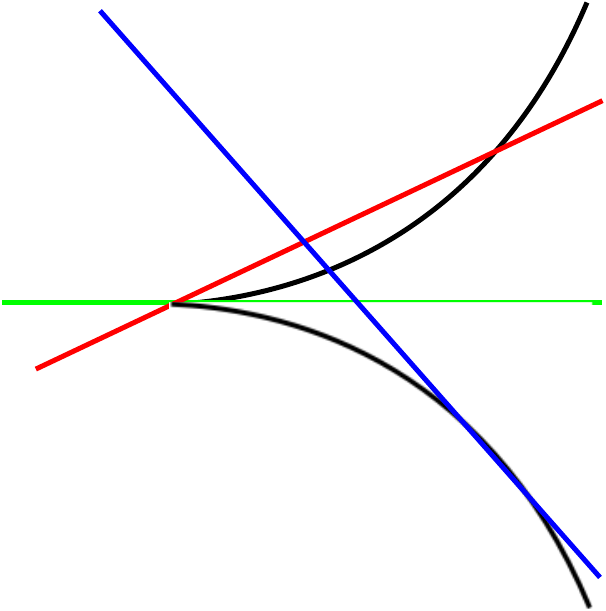}
\hspace{2cm}
\includegraphics[scale=0.4]{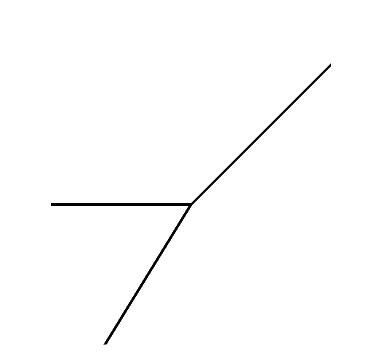}
\put(-130, 65){(-1, 0)}
\put(-95, 10){(-2, -3)}
\put(-25, 100){(1, 1)}
\put(-65, 85){3}
\put(-235, 110){$\C$}
\put(-215, 57){$X=0$}
\put(-215, 95){$Y=0$}
\put(-305, 110){$Z=0$}
\caption{On the left is the real cubic curve from Example \ref{cusp cubic} with its position with respect to the three coordinate axes. On the right is the corresponding tropical curve in $\RR^2$. }
\label{fig:cusp in R2}
\end{figure}

\vspace{1ex}
Let us now turn to tropical morphisms. For a graph $\Gamma$, let $l(\Gamma)$ denote its set of leaves. Recall 
that a star graph is a tree with a unique non-leaf vertex.

\begin{definition}
A punctured abstract fan tropical curve $C$ 
is $\Gamma \backslash l(\Gamma)$ equipped with a  
complete inner metric, where $\Gamma$ is a star graph.

A continuous map  $f : C\to \RR^n$ from a 
punctured abstract fan tropical curve $C$
is a
  tropical morphism if
\begin{itemize}

\item for any edge $e$ of $C$ 
with unit tangent vector $u_e$, 
the restriction $f_{|e}$ is a
  smooth map with 
$df(u_e)=w_{f,e}u_{f,e}$ 
where
 $u_{f,e}\in\ZZ^n$ is a
  primitive vector, and $w_{f,e}$ is a 
positive
integer;

\item it satifsies the balancing condition
$$\sum_{e\in\Ed(C)}  w_{f,e}u_{f,e}=0 $$
where $u_{f,e}$ is chosen so that it points away from the vertex of $C$.
\end{itemize}

\end{definition}

Note that there may be 
several 
rays of $C$ 
having the same image in $\RR^N$ by $f$.
 A tropical morphism $f:C\to \RR^N$
induces a tropical curve in $\RR^N$ in the sense of Definition
\ref{def:fancurve}: the set $f(C)$ is a 
$1$-dimensional rational fan in $\RR^N$, and
an edge $e$ of $f(C)$ is  equipped  with the weight 
$$w_e=\sum_{\substack{e' \in \Ed(C) \\  f(e') = e}} w_{f,e'}.$$

Given a proper algebraic map
$\F:\C\to(\CC^*)^N$ from a punctured Riemann surface $\C$,
we
construct a tropical morphism 
$f:C\to\RR^N$
as follows. Consider the punctured abstract fan tropical curve
$C$ such that the edges $e$ of $C$ are in 
one to one
correspondence with  punctures $p_e$ of
$\C$, 
 and set $f(v)=0$, 
 where $v$ is the vertex of $C$; 
 for each edge $e$ of $C$,  consider a small punctured 
disc $D_e\subset\C$ around the corresponding puncture $p_e$ of $\C$; the set
$\lim_{t\to \infty}\Log_t(\F(D_e))$ is a half line in $\RR^N$ with
primitive integer direction 
$v_e$
and we can define its weight  $w_{f,e}$ as in the
case of 
the tropicalisation of an algebraic curve in $(\CC^*)^N$ 
(in this case 
$\D_e\cap \overline{\F(D_e)}$ 
is a single point and $w_{f,e}$ is the intersection multiplicity
of the toric divisor $\D_e$ and
 $\overline{\F(D_e)}$ 
at that point);
we define $f$ on $e$ by $df(u_e)=w_{f,e}v_{e}$ where $u_e$ is a unit
tangent vector on $e$ pointing away from the vertex $v$.

\begin{definition}\label{def:tropicalisationmorphism}
The tropical morphism $f:C\to\RR^N$ is the
tropicalisation  of $\F:\C\to(\CC^*)^N$, and 
is denoted
by $\trp(\F)$.
\end{definition}

Note that the definitions of  tropicalisations of morphisms and curves are consistent since
we have
$$\trp(\F)(C)=\trp(\F(\C)). $$

\begin{exa}
The map
$$ \begin{array}{cccc}
\F : &\C=\CC^*\setminus\{-1,1\} & \longrightarrow & (\CC^*)^2
\\ & z &\longmapsto & (\frac{z^2(z+1)}{z-1},\frac{z^3}{z-1})
\end{array}.$$
tropicalises to the tropical morphism $f:C\to\RR^2$ where $C$ has 4
edges $e_1,e_2,e_3,$ and $e_4$ with
$$w_{f,e_1 }v_{f,e_1}=(-2,-3), \ w_{f,e_2}v_{f,e_2}=(-1,0),\
w_{f,e_3}v_{f,e_3}=(1,1),\ \text{and, }\ 
w_{f,e_4}v_{f,e_4}=(2,2). $$
The image $\F(\C)$ is the algebraic curve in $(\CC^*)^2$ of Example
\ref{cusp cubic}, (see Figure \ref{fig:cusp in R2}). The weights of $e_3$ and $e_4$ come from
the factorization $-X^3 + 4 X^2Y -5XY^2 + 2Y^3=-(X-2Y)(X-Y)^2$.
\end{exa}

\begin{exa}\label{ex:Anarchy}
Consider the map
$$ \begin{array}{cccc}
\F : &\C=\CC^*\setminus\{1,2,3\} & \longrightarrow & (\CC^*)^2
\\ & z &\longmapsto &
(-\frac{z}{3(z-1)},-\frac{(z-2)(z-3)}{6(z-1)}) 
\end{array}.$$

Figure \ref{fig:Anarchy} shows the curve $\C$ with respect to the $3$ lines in 
$\CC P^2 \backslash (\CC^*)^2$. This morphism tropicalises to
$f: C \longrightarrow \RR^2$, where $C$ is a tropical curve with 
$5$ edges $e_1,e_2,e_3,e_4$ and $e_5$
where, 
$$w_{f,e_1 }v_{f,e_1}=(-1,0),\quad  w_{f,e_2}v_{f,e_2}=(0,1),\quad
w_{f,e_3}v_{f,e_3,}=(1,1),\quad \text{and}\quad$$
$$w_{f,e_4}v_{f,e_4}=w_{f,e_5}v_{f,e_5}=(0,-1).$$

Now consider another embedding of $\C$ into a torus of higher 
dimension given by
$$ \begin{array}{cccc}
\F^{\prime} : &\C=\CC^*\setminus\{1,2,3\} & \longrightarrow & (\CC^*)^3
\\ & z &\longmapsto &
(-\frac{z}{3(z-1)},-\frac{(z-2)(z-3)}{6(z-1)}, -\frac{z}{6}) 
\end{array}.$$
The curve $\F^{\prime}(\C)$ is contained in the plane $\P$ given by 
$X+Y+Z+1=0$.
In Figure \ref{fig:exconic4lines} the curve $\F^{\prime}(\C)$ is drawn in $\overline{\P}$ with 
respect to the $4$ lines in $\overline{\P} \backslash \P$.

\begin{figure}
\includegraphics[scale = 0.4]{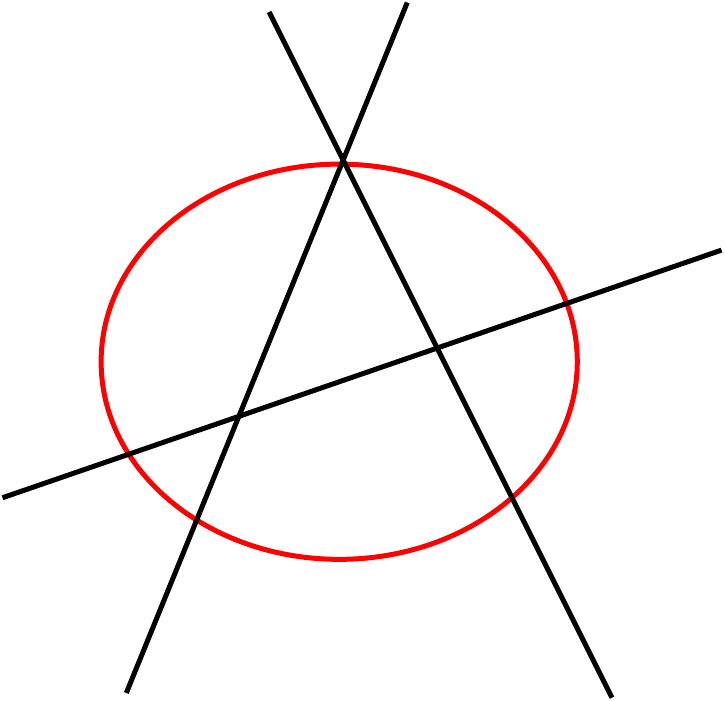}
\hspace{2cm}
\includegraphics[scale = 0.47]{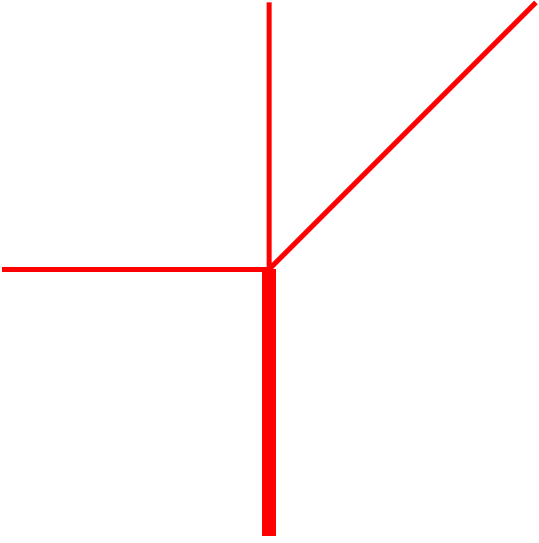}
\put( -340,5 ){\small{$\{z_1 = 0\}$}}
\put( -200, 95 ){\small{$\{z_2 = 0\}$}}
\put( -205, 5){\small{$\{z_0 = 0\}$}}
\put( -330, 85){$\F(\C)$}
\put( -120,70){$f(e_1)$}
\put( -90,110){$f(e_2)$}
\put( -15,90){$f(e_3)$}
\put( -55,10){$f(e_4), f(e_5)$}
\caption{a) The image of the morphism  $\F(\C)$ from Example \ref{ex:Anarchy}  drawn with respect to the
three lines in $\overline{\P} \backslash \P$. b) The tropicalisation of the image  $\F(\C) \subset \RR^2$. 
 \label{fig:Anarchy}}
\end{figure}

\begin{figure}[b]
\includegraphics[scale=0.45]{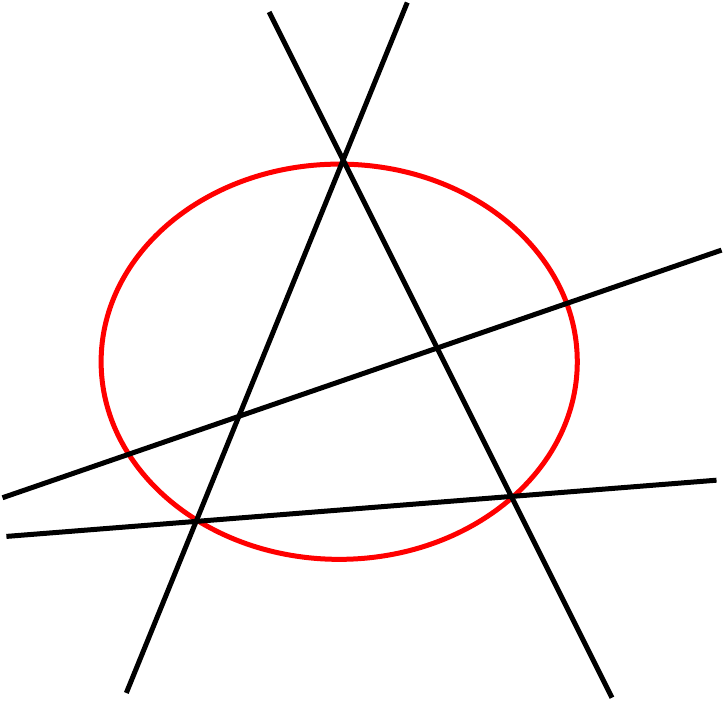}
\hspace{1.3cm}
\includegraphics[scale=0.28]{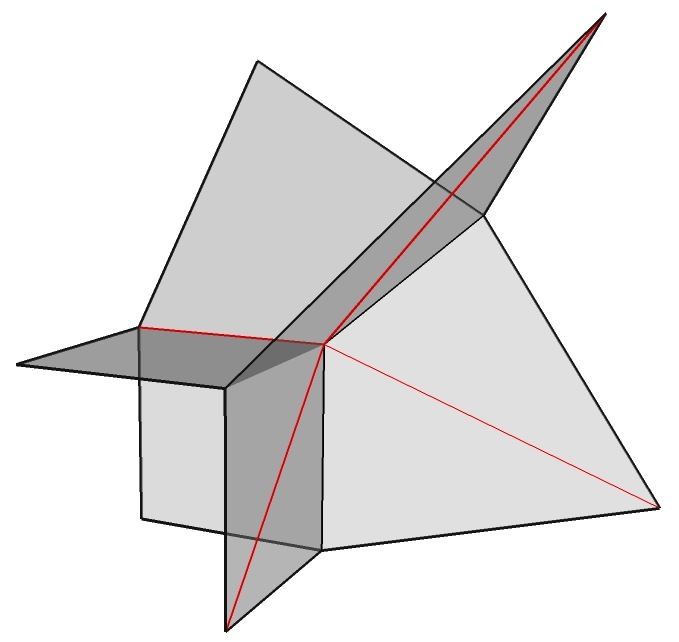}
\put( -370, 115){$\F^{\prime}(\C)$}
\put( -410,5 ){\small{$\{z_1 = 0\}$}}
\put( -420, 55 ){\small{$\{z_2 = 0\}$}}
\put( -255, 5){\small{$\{z_0 = 0\}$}}
\put( -250, 58){\small{$\{z_3 = 0\}$}}
\put( -155,10){\small{$f(e_1)$}}
\put( -55,170){\small{$f(e_3)$}}
\put( -2,40){\small{$f(e_2)$}}
\put( -206,95){\small{$f(e_4), f(e_5)$}}
\caption{On the left is the conic $\F^{\prime}(\C)$ from Example \ref{ex:Anarchy} with its position with respect to the $4$ lines in $\overline{\P} \backslash \P$. On the right the tropicalisation drawn in the compactification $\overline P \subset \TT P^3$ of the  standard tropical plane.}
\label{fig:exconic4lines}
\end{figure}

The map
$\F'$
tropicalises to the tropical morphism $f^{\prime}:C\to\RR^3$ where $C$ has 5
edges $e_1,e_2,e_3,e_4$ and $e_5$ with
 (see Figure \ref{fig:exconic4lines})
$$w_{f,e_1 }v_{f,e_1}=(-1,0,-1),\quad  w_{f,e_2}v_{f,e_2}=(0, 1,1),\quad
w_{f,e_3}v_{f,e_3,}=(1,1,0),\quad \text{and}\quad$$
$$w_{f,e_4}v_{f,e_4}=w_{f,e_5}v_{f,e_5}=(0,-1,0). $$
\end{exa}

\end{section}

\renewcommand{\L}{{\mathcal L}}

\begin{section}{Intersections of tropical curves in planes.}\label{sec:int}

Here we define the intersection of two fan tropical curves $C_1, C_2$
in a fan tropical plane 
$\trp(\P)\subset \R^N$ 
by defining first intersections at the corner
points $p_I$ of a compactification $\overline{\trp(\P)} \subset \TP^N$
defined in Section \ref{sec:linspaces}.
This definition is equivalent to the ones given in \cite{AlRa1} and \cite{Shaw} by  {\cite[Theorem 8.10]{FranRau}}, 
but has
the advantage of rendering the necessary lemmas more transparent.  
For a primitive simplex $\Delta$ giving a degree $1$ compactification of 
$\P \subset (\CC^*)^N$, let   $u_0, \dots , u_N$ denote 
the outward primitive integer normal
vectors  to the faces of $\Delta$. 

\begin{defn}\label{def:cornerInt}

Let $\P  \subset (\CC^*)^N$ be a non-degenerate  plane and $\Delta$ a 
primitive $N$-simplex 
giving a degree $1$ compactification of $\P$. 
Given two fan tropical curves 
$C_1, C_2 \subset \trp(\P)$,
let $\overline{C}_i$ denote their compactifications in 
$\overline{\trp(\P)}
\subset \TP^N$. Let $p_I \in \overline{\trp(\P)}$ be a corner point and
suppose that $\overline{C}_1$ and $\overline{C}_2$ 
both pass through $p_I$.

In the case when both curves 
have each exactly one ray
passing through
$p_I$, we define the intersection multiplicity 
of $\overline{C}_1$ and $\overline{C}_2$
at the
corner $p_I$ as follows: 
\begin{enumerate}
\item If  $I = \{i,j\}$ choose an affine chart $U_l$ for $l \not \in I$ and let $ \pi_{i,j}: U_l \longrightarrow \TT^2 $ be the projection from Section \ref{sec:linspaces}.
Suppose the ray of $\pi_{i,j} (\overline{C}_1 \cap U_l) \subset \TT^2$
has weight 
$w_1$
 and primitive integer direction 
$(p_1,q_1)$, and
similarly the ray of $\pi_{i,j}(\overline{C}_2 \cap U_l) \subset
\TT^2$ has weight 
$w_2$ 
and primitive integer direction 
$(p_2,q_2)$
 then, 
$$(\overline{C}_1. \overline{C}_2)_{p_I} = w_1w_2\min \{p_1q_2, q_1p_2\}.$$
 
\item If $|I| > 2$ choose an affine chart, $U_l \ni p_I$ for $l \not
  \in I$, and a  projection $\pi_{i, j}: U_l \longrightarrow \T^2$
  where $i, j \in I$ such that the rays of $\overline{C}_1$ and 
$ \overline{C}_2$ are contained in the union 
of the closed faces generated by $u_i, u_j$, 
see Figure \ref{fig:corners}. Then 
$$(\overline{C}_1. \overline{C}_2)_{p_I} = (\pi_{i, j}( \overline{C}_1 \cap U_l). \pi_{i, j}( \overline{C}_2 \cap U_l))_{(-\infty, -\infty)}.$$ 
\end{enumerate}

We extend this intersection multiplicity by distributivity in the case
when $\overline{C}_1$ and $\overline{C}_2$ have several rays passing through $p_I$.
\end{defn}

In part $(2)$  of the above definition, if the two rays of $C_1$ and $C_2$
are contained in the same 
open 
face generated 
by $u_i$  and $u_I$ then  we are free to
choose $j$ as we wish. In this case even the result of the projection
$\pi_{i, j}$ does not depend on the choice of $j$.

Next we define the degree of a tropical curve $C \subset \trp(\P)
\subset \R^N$. 
Suppose that a  
vector
$v \in \Z^N$ is contained in $\trp(\P)$. By the construction of
$\trp(\P)$  described in Section \ref{sec:linspaces}, $v$ is contained in a 
cone generated by $u_i, u_I$ for some $I \in \textbf{p}(\A)$ and $i \in I$. 
Therefore, there is a unique expression, 
$v = \rho_i(v)u_i + \rho_I(v)u_I$ where $\rho_i(v), \rho_I(v)$ are non-negative integers.
Set $r_i(v)= \rho_i(v) + \rho_I(v)$ if the direction $v$ is contained in a cone generated by 
$u_i$ and $u_I$ for some $I \ni i$, and $r_i(v) = 0$ otherwise.
Given an edge $e \in \Ed(C)$, we denote by $v_e$  the
primitive integer vector  of $e$ pointing outward from the vertex
of $C$.

\begin{definition}\label{def:degree}
Let $\P  \subset (\CC^*)^N$ be a non-degenerate  plane and $\Delta$ a 
primitive $N$-simplex 
giving a degree $1$ compactification of $\P$. 
Let $C \subset \trp(\P)$ be a fan
tropical curve
and $i \in \{0, \dots, N\}$.  
We define the degree of $C$ with respect to $\Delta$ as
$$\deg_{\Delta}(C) = \sum_{e \in \Ed(C)} w_er_i(v_e).$$
\end{definition}

It follows from the balancing condition that the above definition is
independent  of the choice of $u_i$. 
In fact
as we will see in the proof of Lemma \ref{lem:degree}, 
the degree of a curve $C \subset \R^N$ with respect to $\Delta$ is the 
multiplicity of the tropical  stable intersection in $\RR^N$ of the curve $C$ 
and  the tropical 
variety 
dual to $\Delta$.
The next examples demonstrate the dependence of $\deg_{\Delta}(C)$ on the choice of 
$\Delta$ 
when such a choice exists.

\begin{example}\label{ex:compdeg}
Recall the plane $\P \subset (\CC^*)^3$ from Example
\ref{ex:compdeg2}, which admits two compactifications to $\CC P^2$
given by polytopes $\Delta$, $\Delta^{\prime}$. 
Consider the $4$-valent fan 
tropical curve $C \subset \trp(\P)$ with edges of weight $1$ in directions:
$$(1, 1, 1), \quad (-1, 0, 0), \quad (0, -1, 0) \quad \text{and} \quad (0, 0,  -1).$$
Calculating the degree using Definition \ref{def:degree} we obtain,  $\deg_{\Delta}(C) = 1$ and  $\deg_{\Delta^{\prime}}(C) = 2$. 
\end{example}

\begin{figure}
\includegraphics[scale = 0.4]{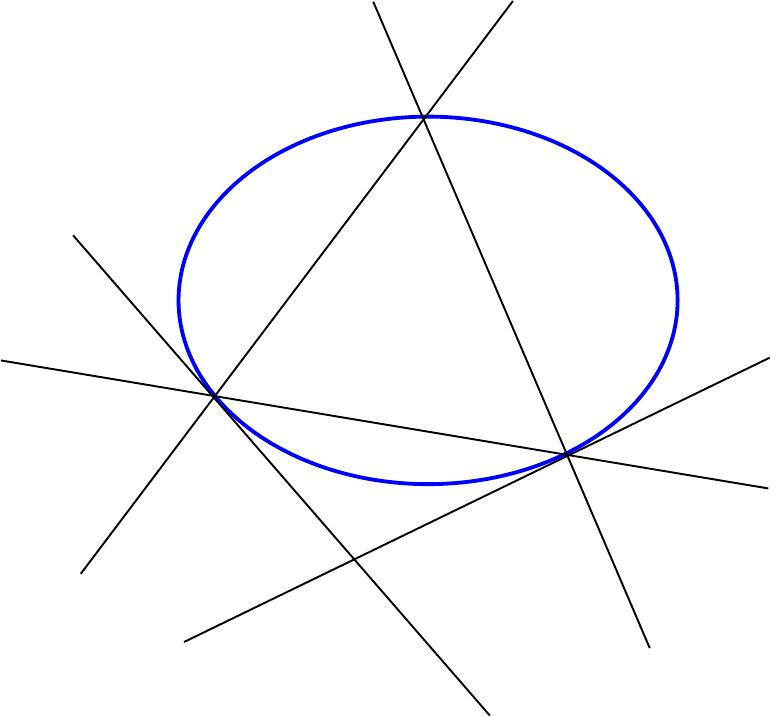}
\put(-25, 100){$\overline{\C}$}
\put(-5, 75){$\L_4$}
\put(-165, 68){$\L_3$}
\put(-150,  95){$\L_2$}
\put(-150, 20){$\L_1$}
\put(-25, 3){$\L_0$}
\hspace{2cm}
\includegraphics[scale = 0.4]{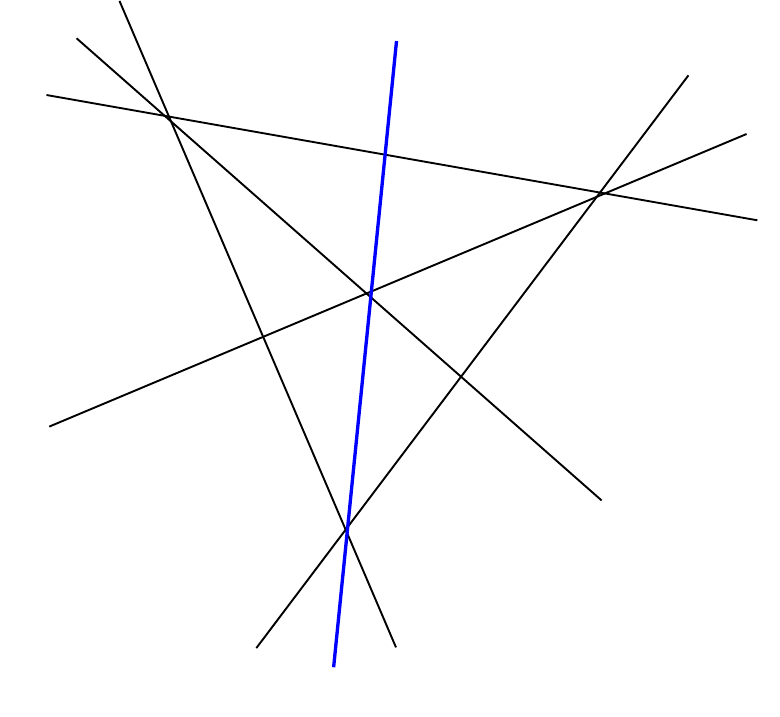}
\put(-70, 122){$\overline{\C}^{\prime}$}
\put(-18, 82){$\L_3^{\prime}$}
\put(-35, 30){$\L_1^{\prime}$}
\put(-127,  90){$\L_2^{\prime}$}
\put(-115, 20){$\L_0^{\prime}$}
\put(-3, 110){$\L_4^{\prime}$}
\caption{Two compactifications of the plane $\P$ from Example \ref{ex:cremona} along with the compactification of a curve $\C \subset \P$ from Example \ref{ex:cremonaCurve}.  \label{fig:cremona}}
\end{figure}
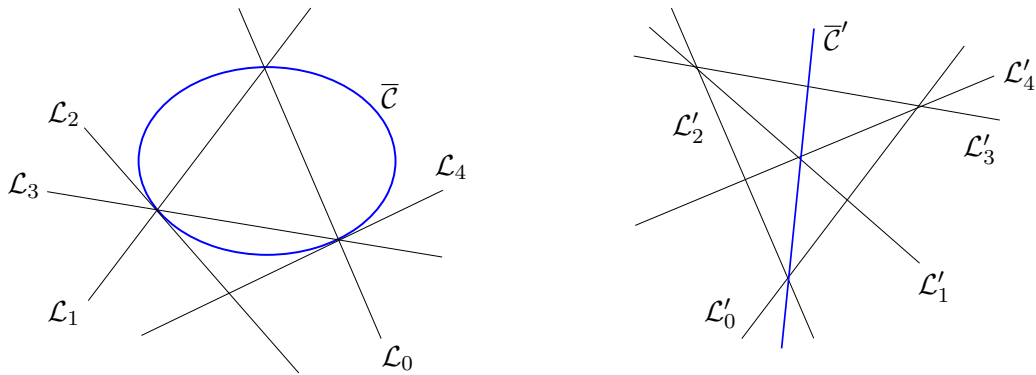

\begin{example}\label{ex:cremonaCurve}
Recall the plane $\P \subset (\CC^*)^4$ from Example \ref{ex:cremona}, which admits two compactifications to $\CC P^2$ given by polytopes $\Delta$, $\Delta^{\prime}$. Consider the curve $\C \subset \P$ given by the additional equation $x_1x_3 - x_2 = 0$. 
The left hand side of Figure  \ref{fig:cremona} shows the compactification of this curve $\overline{\C} \subset \CC P^2$ given by the standard simplex $\Delta$. The curve $\overline{\C}$ is a conic.  The right hand side of the same figure shows the compactification $\overline{\C}^{\prime} \subset \CC P^2$ given by $\Delta^{\prime}$. In this compactification $\overline{\C}^{\prime}$ is a line. 
It follows from the next lemma that  $\deg_{\Delta}(C) = 2$ and $\deg_{\Delta^{\prime}}(C) = 1$, where $C$ is the tropicalisation of $\C$. 
\end{example}

\begin{lemma}\label{lem:degree}
Let $\P  \subset (\CC^*)^N$ be a non-degenerate  plane, let $\C \subset (\CC^*)^N$ be a complex algebraic curve, and $\Delta \subset \RR^N$ a 
primitive $N$-simplex 
giving a degree $1$ compactification of $\P$. Then $\deg(\overline{\C}) = \deg_{\Delta}(\trp(\C))$, where
$\overline{\C}$ is the closure of $\C$ in the toric 
 compactification
of $(\CC^*)^N$ to $\CC P^N$
given by $\Delta$. 
\end{lemma}

\begin{proof}
Let $H$ be the tropicalisation of a hypersurface $\H$ of $(\CC^*)^N$ with
Newton polygon $\Delta$.  Then let  $\overline{\H} \subset \CC P^N$ 
be the closure of $\H$ in the toric compactification given by $\Delta$. Then $\overline{\H}$ is a
hyperplane, thus of degree 1. 
Given $i\in\{0,\ldots,N\}$,
there is a translation $H^{\prime} = H + v_i$
where $v_i \in \RR^N$, 
such that $H^{\prime}$ intersects the tropical curve 
$C$
in a finite number of points and only in the face 
of $H^{\prime}$ orthogonal to $u_i$. 
Then each such edge $e$  intersecting $H^{\prime}$ does so with
tropical multiplicity   
$w_e r_i(v_e)$, 
in the notation of Definition \ref{def:degree}. 
Therefore by Definition  \ref {def:degree} $\deg_{\Delta}(C) = \deg(H^{\prime}.C) = \deg(H.C)$, 
where $H^{\prime}.C$ is the stable intersection  from  \cite{St2}.
 
The translation $H^{\prime}$  is approximated by the  family $\H_t =
\{ (t^{v_1}z_1, \dots t^{v_N}z_N \ | \ (z_1, \dots , z_N) \in  \H\}$
in the sense that $\lim_{t \to \infty} \Log_t (\H_t) = H^{\prime}$. 
The
intersection of $H'$ and $C$  is proper, so  by 
{\cite[Theorem 8.8]{KatzTool}},  for $t$ large enough, 
$\deg(H^{\prime}.C)$
is the intersection number of 
$\H_t$
and
$\C$. 
Compactifying $\R^N$ to $\TT P^N$ 
by way of $\Delta$, 
the closures
$\overline{H^{\prime}}$ and $\overline{C}$ do not intersect at the boundary
for a generic translation. Therefore $\deg_{\Delta}(C) = \deg(H^{\prime}.C) =
\deg(\overline{\H}.\overline\C) = \deg(\overline{\C})$, and the lemma is
proved. 
\end{proof}

We define next
the tropical
intersection multiplicity of two tropical curves $C_1$ and $ C_2$
contained in a fan tropical plane.  
By {\cite[Theorem 8.10]{FranRau}} this definition is equivalent to the one from \cite{AlRa1},  \cite{Shaw}, and so it
is independent of the choice of $\Delta$. 

\begin{definition}\label{def:localInt}
Let $\P\subset(\CC^*)^N$ be a non-degenerate  plane
and $\Delta$ a 
primitive $N$-simplex 
giving a degree $1$ compactification of $\P$. 
 Given  two fan tropical curves
$C_1$ and $C_2$ in $\trp(\P)$, we define their  tropical intersection number 
in 
$\trp(\P)$ as
$$C_1.C_2= \deg_{\Delta}(C_1).\deg_{\Delta}(C_1) - \sum_{\bold p_I \in \bold
  p(\A)}(\overline{C}_1.\overline{C}_2)_{p_I}.$$ 
\end{definition}

In order to interpret the above tropical intersection number in the case of two approximable tropical curves, we must  first describe an
appropriate compactification of  the open space $\P$, 
in general
distinct from $\overline \P = \CC P^2$. 
Recall that in Section \ref{section:definitions curves} a
compactification $\X$ of 
$(\CC^*)^N$ 
 is called
\textbf{compatible} 
with a complex curve $\C$
if $\X$ is a toric compactification of
$(\CC^*)^N$ such that for any two irreducible boundary divisors $\D_i$ and
$\D_j$ of $\X$,  the intersection $\D_i \cap \D_j \cap \C$ is empty.
For two complex algebraic curves $\C_1$ and $\C_2$
in a plane $\P \subset (\CC^*)^N$, we call
a compactification $\X$ of $(\CC^*)^N$  \textbf{compatible} with
$\C_1, \C_2$ and  $\P$  if it is compatible with both
curves and the compactification $\tilde{\P} \subset \X$ of $\P$ is a
non-singular surface.
We refer to \cite{Tev1} for more details about tropical  compactification of
subvarieties of $(\CC^*)^N$.

\begin{example}\label{ex:comp}
Here we describe a compactification of $(\CC^*)^N$ compatible with two
given complex curves $\C_1$ and $\C_2$ in a plane $\P\subset (\CC^*)^N$ that
we will use in the rest of the article.
First let  $\Sigma \subset \RR^N$ be the complete unimodular fan dual to a 
primitive $N$-simplex  $\Delta$
giving a degree $1$ compactification of $\P$. 
Let $\tilde \Sigma$ be a unimodular completion of
$\Sigma \cup \trp(\C_1) \cup \trp(\C_2) \subset \R^N$, and $\tilde{\X}(\Delta)$ 
denote the corresponding  compactification of $(\CC^*)^N$. Then $\tilde{\X}(\Delta)$ is
compatible with $\C_1, \C_2$
and $\P$. This compactification  $\tilde \P$ of $\P$ 
is the minimal compactification of $(\CC^*)^N$
compatible with $\C_1, \C_2$
and $\P$. 
Moreover, 
it is
obtained
from  $\CC P^2$ by blowing up points in $\textbf{p}(\A)$ and points
above them which are 
intersection of boundary divisors
(see Figure \ref{fig:compatibleblowup}).  
\end{example}
 
The following theorem states the  correspondence between the complex and tropical intersection numbers  for curves.

\begin{figure}
\includegraphics[scale=0.3]{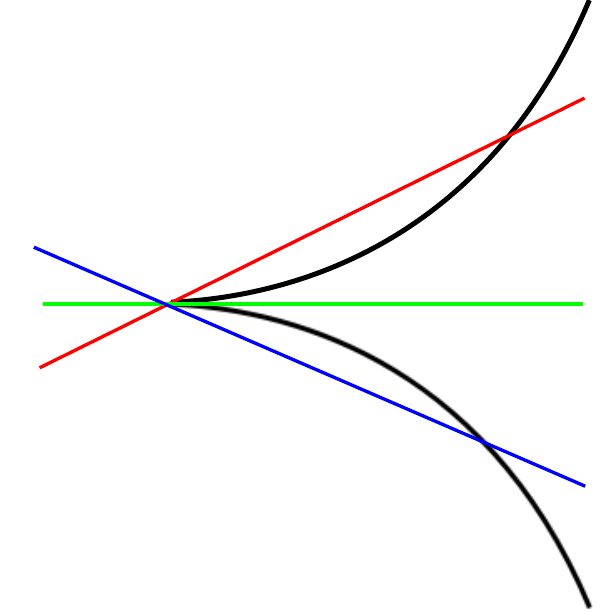}
\hspace{2cm}
\includegraphics[scale=0.25]{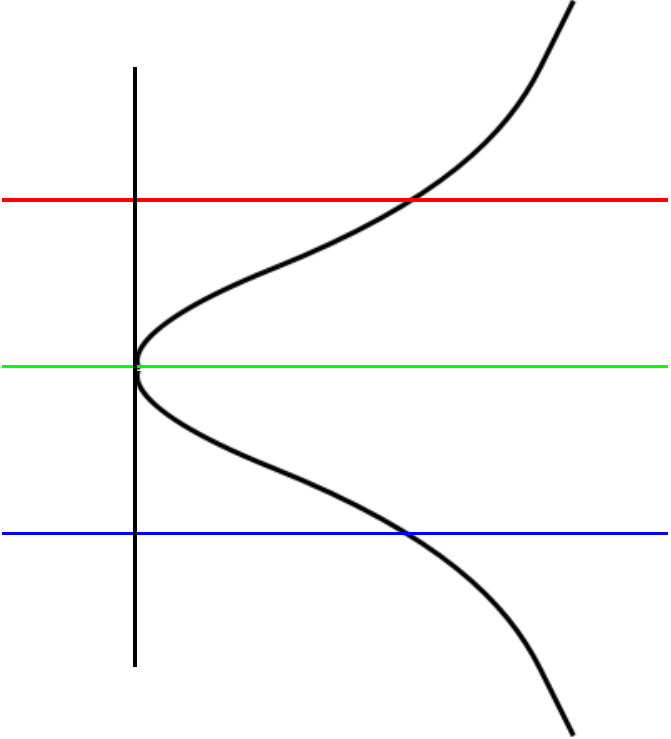}
\hspace{2cm}
\includegraphics[scale=0.25]{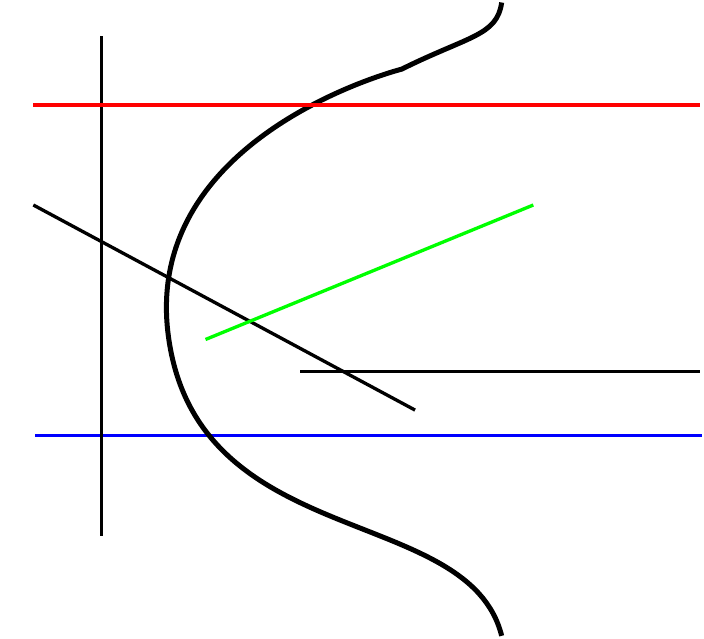}
\put(-365, 55 ){\small{$\textbf{p}_{I}$}}
\put(-315, 10 ){\small{$\C$}}
\put(-375, -5 ){\small{$\CC^2$}}
\put(-240, -5 ){\small{$Bl_{\textbf{p}_{I}}(\CC^2)$}}
\put(-85, -5 ){\small{$\tilde{\P}$}}
\put(-160, 10 ){\small{$\C$}}
\put(-228, 75){\small{$\E_I$}}
\put(-85, 75){\small{$\E_I$}}
\put(-95, 50){\small{$\E^{\prime \prime}$}}
\put(5, 30){\small{$\E^{\prime}$}}
\put(-25, 10 ){\small{$\C$}}
\vspace{0.5cm}

\includegraphics[scale=0.46]{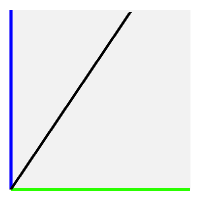}
\hspace{2cm}
\includegraphics[scale=0.35]{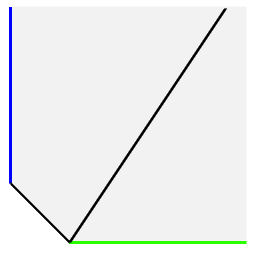}
\hspace{2.2cm}
\includegraphics[scale=0.33]{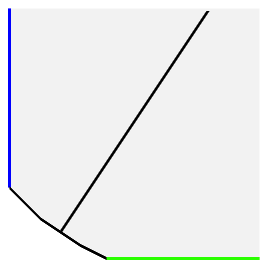}
\put(-345, 75){\small{$(2, 3)$}}
\put(-245, 8){\small{$E_I$}}
\put(-65, -2){\tiny{$E^{\prime}$}}
\put(-87, 15){\tiny{$E_I$}}
\put(-78, 4){\tiny{$E^{\prime \prime}$}}
\caption{An illustration of the compactification $\tilde{\P}$ at a single  corner point $\textbf{p}_{i,j,k}$, along with the tropicalisation of a projection $\pi_{i, j}$, and the toric blow ups.}
\label{fig:compatibleblowup}
\end{figure}
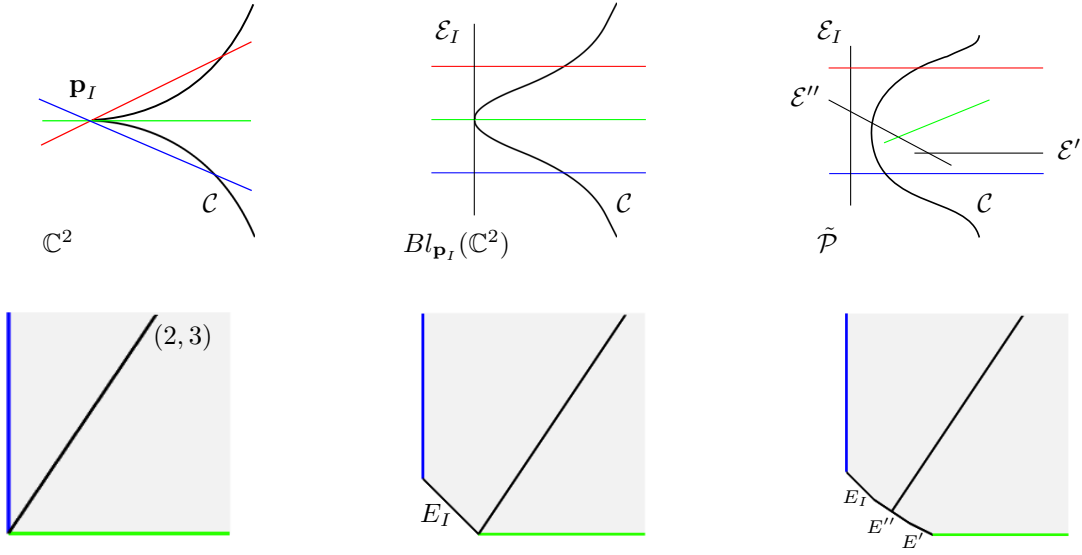

\begin{thm}\label{thm:realiseInt}
Let $\P\subset(\CC^*)^N$ be a non-degenerate  plane, let $\C_1$ and $\C_2$ be two
algebraic curves in $\P$, and let $C_i=\trp(\C_i)$. Then 
$$\tilde{\C}_1.\tilde{\C}_2 = C_1.C_2$$ 
where $\tilde{\C}_i$ is the compactification of $\C_i$ in the  compatible toric compactification $\tilde{\P} \subset \tilde{\X}(\Delta)$ from Example \ref{ex:comp}.
\end{thm}

The proof of Theorem \ref{thm:realiseInt}, uses the induced intersection theory of \cite{KatzInt} along with {\cite[Lemma 2.23]{Shaw}}. 
Recall that a tropical rational function on $\R^n$ is a tropical quotient of tropical polynomials, in other words it is a difference of convex piecewise integer affine functions. The divisor of a  tropical rational function restricted to a tropical fan plane produces a tropical $1$-cycle. For an introduction to the theory of tropical divisors see \cite{AlRa1}, \cite{Mik3}. 
Since here we are only considering  fan curves, we may  restrict to tropical rational functions which are simply piecewise linear. 

\begin{lemma}[{\cite[Lemma 2.23]{Shaw}}]\label{lem:divisor}
Suppose $P \subset \R^n$ is a fan tropical plane and $C \subset P$ a fan tropical curve, then there exists a tropical rational function $f: \R^n \longrightarrow \R$ such that $\div_P(f) = C$. 
\end{lemma}

\begin{proof}[Proof of Theorem \ref{thm:realiseInt}]  
Let $c_i \in A^1(\tilde{X}(\Delta))$ denote the cocyle associated to $\tilde{\C_i}$.  Then a tropical rational function $f_i$ such that $\div_P(f_i) = C_i$ from Lemma \ref{lem:divisor}  is a lift of $c_i$ in $\tilde{\P}$ in the sense of
{\cite[Definition 6.2]{KatzInt}}.  The equality of the intersection numbers then follows from 
{\cite[Theorem 6.3]{KatzInt}}.
\end{proof}

Note that the above theorem does not depend on the initial choice of simplex $\Delta$ used to construct the compatible compactification
in Example \ref{ex:comp}. 
Following Theorem \ref{thm:realiseInt} we present some examples and immediate corollaries.

\begin{example}\label{ex:CC<0}
Let $P$ be the 
tropicalisation of a 
uniform hyperplane
$\P \subset (\CC^*)^3$ and let $L \subset P$ be the
affine line in the direction $(1, 1, 0)$. It is the red curve depicted
in Figure \ref{fig:planecurves}.  Since $\P$ is uniform there is a unique 
simplex $\Delta$ yielding a degree $1$ compactification of $\P$. 
Using  Definition \ref{def:localInt} we compute $L.L = -1$.
This tropical line is approximable by a line $\L
\subset \P$. Viewing $\P$ as a complement of  four generic lines
$\L_1, \dots, \L_4$ in $\CP^2$, then for some labeling of these lines,
$\L$ is the line passing through the two points $\L_1 \cap \L_2$ and
$\L_3 \cap \L_4$.
Again see  Figure
\ref{fig:planecurves} for a real drawing of the five lines. The blow
up of $\CP^2$ at the two points $\L_1 \cap \L_2$ and $\L_3 \cap \L_4$
gives the desired compactification $\tilde{\P}$. The proper transform
of $\L$ in $\tilde{\P}$ is indeed a curve of self intersection $-1$.  
\end{example}

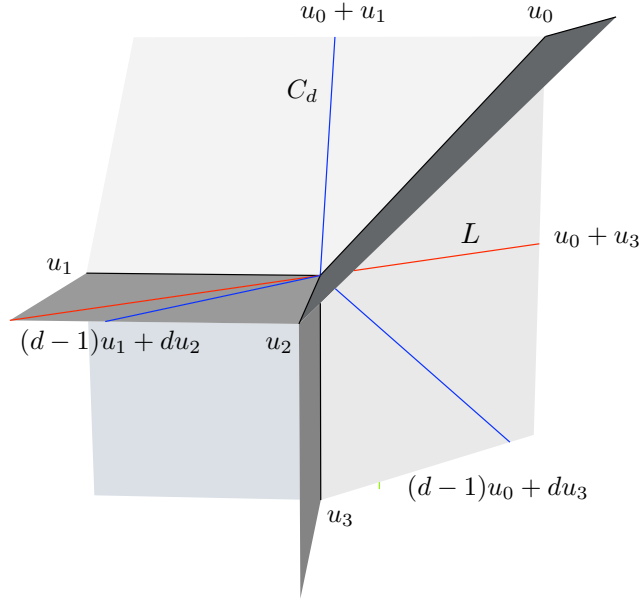
\begin{figure}
\includegraphics[scale=0.9]{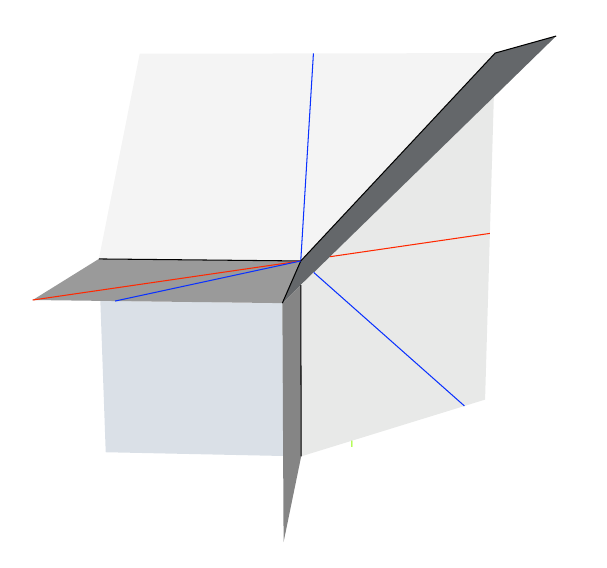}
\put(-45, 145){\small{$u_0+u_3$}}
\put(-80, 145){$L$}
\put(-100, 50){\small{$(d-1)u_0 + du_3$}}
\put(-245, 105){\small{$(d-1)u_1 + du_2$}}
\put(-140, 230){\small{$u_0+u_1$}}
\put(-55, 230){\small{$u_0$}}
\put(-235, 135){\small{$u_1$}}
\put(-153, 105){\small{$u_2$}}
\put(-130, 40){\small{$u_3$}}
\put(-145, 200){\small{$C_d$}}
\vspace{-0.5cm}
\caption{The curve $C_d$ and the line $L$ from Example \ref{ex:CD<0}.}
\label{fig:minusLC}
\end{figure}

\begin{example}\label{ex:CD<0}
We recall the example presented in  {\cite[Section $4$]{Shaw}}. 
Let us consider $P$ and $L$ as in Example \ref{ex:CC<0}, and let
$C_d$ 
be the trivalent
tropical curve of degree $d$ centered at the vertex of $P$ with weight
one rays in the directions 
$$u_0+u_1,  \quad (d-1)u_0 + du_3, \quad
\text{ and } \quad (d-1)u_1 + du_2, $$ 
see Figure \ref{fig:minusLC}.
These curves intersect at two corner points of
$\overline{P} \subset \TP^3$. Locally at these two corners the curves
appear as rays 
passing through
the corner of $\T^2$, $L$ in the
direction $(-1, -1)$ and $C_d$ in the direction $(-d, 1-d)$. Using
Definition \ref{def:localInt} we have $$C_d.L= d\cdot 1 - 2(d-1) =
-d+2.$$  
\end{example}

\begin{corollary}\label{cor:CD<0}
Let $\P$ be a non-degenerate plane and $C\subset\trp(\P)$ be 
a finely 
 approximable  
 fan tropical curve. If $D \subset \trp(\P)$ is 
 also 
 a 
 fan tropical curve such that 
 $D \neq C$ and $C.D <0$ then $D$ is not
 finely
 approximable. 
\end{corollary}

It follows from this corollary 
and Example \ref{ex:CC<0}
that the curve $C_d$ from Example
\ref{ex:CD<0}  is not approximable for  $d \geq 3$, this was already
shown in \cite{Shaw}. 

\begin{corollary}\label{cor:CC<0}
Let $\P$ be a non-degenerate plane and $C\subset\trp(\P)$ be 
a finely 
approximable  
 fan tropical curve. 
If 
$C.C <0$ then $C$ is 
finely
approximated by a unique complex curve $\C \subset \P$. 
\end{corollary}

\end{section}

\begin{section}{Obstructions coming from the adjunction formula}\label{sec:adj}
The adjunction formula for a non-singular curve $ \C$ in a
non-singular 
compact
complex 
surface $\X$ reduces to (see \cite{Sha})
$$g(\C) = \frac{K_{\X}. \C + \C^2 + 2}{2}$$
where $g(\C)$ is the genus of 
$\C$,
and $K_{\X}$ is the canonical class of $\X$. If $\C$ is singular but
reduced,  the right hand side of the  above formula defines the
arithmetic genus of the curve and we denote it by $g_a(\C)$.  
As before, denote by
$g(\C)$
the geometric genus of the curve
$\C$, 
i.e.~the genus of  
its normalisation.
If $\C$ is an irreducible curve, we have 
$0\le g(\C)\le g_a(\C)$.
We interpret this bound on the level of the tropical curve in order to
prove 
Theorem \ref{thm:Adj}
which appeared in a simplified version in Theorem \ref{thm:simpadjunction}. 

Beforehand, 
we need to introduce some more notation.
Let  $\P \subset (\CC^*)^N$ be a non-degenerate plane and $\Delta$ be a 
primitive $N$-simplex giving a degree $1$ compactification of $\P$. 
Recall that if the arrangement $\A$ corresponding to $\P \subset (\CC^*)^N$
contains a point $\textbf{p}_I$, then the fan tropical plane $P = \trp(\P) \subset \RR^N$ contains a ray 
 in the corresponding direction  $u_I = \sum_{i \in I} u_i$, where $u_0, \dots, u_N$ are the outgoing 
 primitive  normal vectors to the simplex $\Delta$.
Given a fan tropical curve $C \subset P$, let $w_I$ denote the weight
of the edge of $C$ in the direction  $u_I$,
with the convention that
$w_I=0$ 
if $C$ does not contain a ray in this direction.  Note that this definition depends 
on the directions of $u_I$, therefore it depends on the choice of the simplex $\Delta$
giving a degree $1$ compactification when a choice exists.

\begin{thm}\label{thm:Adj}
Let $\P\subset(\CC^*)^N$ be a non-degenerate  plane and $C \subset \trp(\P)$  a 
fan 
tropical curve. If $C$ is finely approximable by a complex curve
$\C\subset\P$, then for a primitive $N$-simplex $\Delta \subset \RR^N$ 
giving a degree $1$ compactification of $\P$
we have 
$$2 g(\C) \leq C^2 + (N-2)\deg_{\Delta}(C) - \sum_{e_i \in \Ed(C)} w_{e_i} - \sum_{\textbf{p}_I \in \textbf{p}(\A)}( |I| - 2)w_I + 2,$$ 
with equality if and only if  $\C$
is non-singular.
\end{thm}

Recall that for  a uniform plane, there is a unique degree $1$ compactification. Therefore the simplex 
 $\Delta$ is unique and the degree $d$ used in the statement of Theorem \ref{thm:simpadjunction} is well-defined. Moreover, all of the points of the corresponding arrangement are pairs $i, j$. 
This accounts for the simplified version of the above theorem
 stated in the introduction.

\begin{proof}
Let 
$\tilde{\P}$ and $\tilde{\C}$ be the closures of $\P$ and $\C$
in the  compatible compactification $\tilde{\X}(\Delta)$ of $(\CC^*)^N$
 described in Example \ref{ex:comp}. 
 Recall that $\tilde{\P}$ is a blow up of $\CP^2$, let  $\pi: \tilde{\P} \longrightarrow \CP^2$
denote the contraction map. 
The boundary $\partial \tilde{\P} = \tilde{\P} \backslash \P$ is a
collection of non-singular divisors consisting of the proper
transforms of the $N+1$ lines
in $\overline\P\setminus \P$
along with all 
exceptional divisors.  
Given $\textbf{p}_I\in\textbf{p}(\A)$, we denote by $\E_I$  
the  proper transform in $\tilde{\P}$ of the
exceptional divisor of the blowup of $\CP^2$ at the point
$\textbf{p}_I$,
 by
$\partial \tilde{P} $ the sum of all divisors in $ \tilde{\P}
\backslash \P$, and by $\L$ the divisor class of a line
in $\CC P^2$. Note that  the divisors
$\E_I$ are contained in the support of $\partial \tilde{P} $.
We will prove in Lemma \ref{KtropComp} that  
the canonical class of $\tilde{\P}$ can be written in the following way: 
$$K_{\tilde{\P}} = (N-2)\pi^*\L  -   \partial \tilde{\P} - \sum_{\textbf{p}_I \in \textbf{p}(\A)}( |I| - 2)\E_I.$$

With $K_{\tilde{\P}}$ written this way, we may calculate $K_{\tilde{\P}}.\C$ using just the tropical curve  $C= \trp(\C)$. Firstly, $\pi^*\L.\C = \deg_{\Delta}(C)$.  By definition of the weights of the edges of $\trp(C)$ we have 
$$\E_I.\tilde{\C} = w_I \quad \text{and} \quad  \partial \P . \tilde{\C} = \sum_{e \in \Ed(C)} w_{e}. $$
Therefore, 
$$K_{\tilde{\P}}.\C  = (N-2)\deg_{\Delta}(C) - \sum_{e\in \Ed(C)} w_{e} - \sum_{\textbf{p}_I \in \textbf{p}(\A)}( |I| - 2)w_{I}.$$
By Theorem \ref{thm:realiseInt} we have   $\tilde{\C}^2 =C^2$. Applying the adjunction formula for $\tilde{\C} \subset \tilde{\P}$ we obtain the claimed inequality. 
\end{proof}

\begin{corollary}\label{cor:neggenus}
Let $\P\subset(\CC^*)^N$ be a non-degenerate plane and $C \subset \trp(\P)$  a 
fan tropical  curve. If $C$ is finely approximable by a complex curve
$\C\subset\P$, then for a primitive $N$-simplex $\Delta \subset \RR^N $
giving a degree $1$ compactification of $\P$
 we have
$$C^2 + (N-2)\deg_{\Delta}(C) - \sum_{e\in\Ed(C)} w_{e} - \sum_{\textbf{p}_I \in \textbf{p}(\A) }( |I| - 2)w_I+ 2 \geq 0.$$ 
\end{corollary}

The following lemma completes the proof of Theorem \ref{thm:Adj}.

\begin{lemma}\label{KtropComp} 
Using the same notations as in the proof of Theorem \ref{thm:Adj},
we have
$$K_{\tilde{\P}} = (N-2)\pi^*\L  -   \partial \tilde{\P} - \sum_{\textbf{p}_I \in \textbf{p}(\A)}( |I| - 2)\E_I.$$ 
\end{lemma}

\begin{proof}
To see that the canonical class can be expressed as claimed we first start with
$$K_{\CC P^2} = -3\L  = -\sum_{i=0}^N \L_i + (N-2) \L,$$  
where the $\L_i$'s  are the  lines
in $\overline\P\setminus \P$.
If $\pi^{\prime}: \P^{\prime} \longrightarrow
\CP^2$ is the blowup of $\CP^2$ at 
the point $\textbf{p}_I$, 
the canonical classes
are related as follows,
$$K_{\P^{\prime}} = \pi^{\prime *} K_{\CP^2} + \E_I.$$ 
Then, 
$$K_{\P^{\prime}}  = (N-2) \pi^{\prime *}\L  -  \pi^{\prime *}(\sum_{i=0}^N \L_i) + \E_I 
    = (N-2) \pi^{\prime *}\L   - \sum_{i=0}^N \tilde{ \L_i } - |I| \E_I  + \E_I, 
$$
where $\tilde{\L_i}$ is the proper transform of $\L_i$.  
Moreover $\partial \P^{\prime} = \sum_{i=0}^N \tilde{\L_i}+ \E_I$, so
$$K_{\P^{\prime}}   = (N-2) \pi^{\prime *}\L   - \partial \P^{\prime} - (|I| -2) \E_I.$$

Blowing up further at points above $\textbf{p}_I$ that are the intersection of two boundary divisors, the exceptional divisor is  again a boundary divisor of the new surface. Continuing the process at each $\textbf{p}_I$ to obtain $\tilde{\P}$ and we have:

$$K_{\tilde{\P}}=  (N-2)\pi^{*} \L  - \partial \tilde{\P} -  \sum_{\textbf{p}_I \in \textbf{p}(\A)}( |I| - 2)\E_I, $$ 
which completes the proof.
\end{proof}

\end{section}

\renewcommand{\L}{{\mathcal L}}
\begin{section}{Obstructions coming from intersections the with Hessian}\label{sec:Hess}
Consider a plane algebraic curve $\C$ in  $\CC P^2$
given by the homogeneous  equation
$F(z,w,u)=0$.
 The Hessian of the polynomial
$F(z,w,u)$,
 denoted by 
 $Hess_F(z,w,u)$, 
is  the homogeneous
polynomial defined as
$$Hess_F(z,w,u)=\det
\left(\begin{array}{ccc}

\frac{\partial^2 F}{\partial ^2z}

           & \frac{\partial^2 F}{\partial z\partial w}

               & \frac{\partial^2 F}{\partial z\partial u}\\

        \frac{\partial^2 F}{\partial z\partial w}

           &\frac{\partial^2 F}{\partial ^2w}

               & \frac{\partial^2 F}{\partial w\partial u}\\

       \frac{\partial^2 F}{\partial z\partial u}

           & \frac{\partial^2 F}{\partial w\partial u}

               &\frac{\partial^2 F}{\partial ^2u}

   \end{array}\right). $$

If $Hess_F$
 is not the null polynomial, it
defines a curve $Hess_\C$  called the \textit{Hessian} of $\C$. 
The polynomial 
$Hess_ \C$ 
 of course depends on the chosen coordinate
system on $\CC P^2$, however the curve $Hess_\C$ does not, i.e.~the
curve $Hess_\C$ is invariant under projective change of coordinates in
$\CC P^2$.
Note that if $\C$ has degree $d$, then $Hess_\C$ has degree $3(d-2)$ and intersects $\C$ in finitely
many points
if $\C$ is reduced and does not contain a line as a
component. Intersecting a curve with its Hessian detects 
singularities and
inflection
points of the curve.
In particular,
if $\L$ is a line and $p\in\C$ are such that
$(\C.\L)_p=m$, then $(\C. Hess_\C)_p\ge m-2$.

Before we present the obstructions we fix some further notations. 
Given $\C$ an algebraic curve in affine space  $\CC^2$ defined by a polynomial
$\sum a_{i,j}z^iw^j$, 
we denote by $\Delta(\C)=Conv\{(i,j)\in\ZZ^2 \ | \ a_{i,j}\ne 0\}$
 its Newton polygon, and we define 
$$\Gamma(\C)=Conv(\Delta(\C)\cup \{(0,0\}), \quad \text{and}\quad
\Gamma^c(\C)=\Gamma(\C)\setminus \Delta(\C). $$
Once a coordinate
system is fixed
in $\CC^2$, 
the equation of an algebraic curve 
 is defined up to a non-zero multiplicative 
constant. In particular the polygons $\Delta(\C)$, $\Gamma(\C)$, and
$\Gamma^c(\C)$ do not depend on the particular choice of the defining
polynomial.
The latter definition translates literaly to tropical curves in $\TT^2$.
If $C$ is the tropicalisation of a projective plane curve $\C$ 
in the coordinates $(z,w)$, then we have
$$\Delta(\C)=\Delta(C),\quad
 \Gamma(\C)=\Gamma(C), \quad \text{and}\quad
\Gamma^c(\C)=\Gamma^c(C).$$ 
Finally,  given a polygon $\Delta$ in $\RR^2$ we will denote
by $A(\Delta)$ its lattice area, i.e. twice its euclidean area.

\begin{lemma}\label{lower hessian}
Let $\C$ be an algebraic curve in $\CC P^2$,
and let us fix a coordinate system on $\CC P^2$.
 If  $\C$ is reduced and
does not contain any line as a 
component, then
$$(\C. Hess_\C)_{[0:0:1]}\ge 3 A(\Gamma^c(\C)) + r_0(\C) - 2v_0(\C)
-2h_0(\C) $$
where 
\begin{itemize}
\item $r_0(\C)= Card(e\cap \ZZ^2) - 1$ if there exists an edge $e$ of
  $\Gamma^c(\C)$ of slope -1, and $r_0(\C)=0$ otherwise;
\item $v_0(\C)= Card(\Gamma^c(\C)\cap (\{0\}\times\ZZ)) - 1$;
\item $h_0(\C)= Card(\Gamma^c(\C)\cap (\ZZ\times\{0\})) - 1$.

\end{itemize}
\end{lemma}
\begin{proof}
The intersection multiplicity at $[0:0:1]$ of the curve $\C$ and
its Hessian is  bigger than the number of inflection points in
$\CC P^2$ of a curve with Newton polygon 
$\Gamma(\C)$ 
minus the number of 
inflection points in
$\CC P^2\setminus\{[0:0:1])\}$ of a curve with Newton polygon
$\Delta(\C)$.
Hence the result follows from {\cite[Proposition 6.1]{Br16}}.
\end{proof}

\begin{exa}\label{simple hessian}
We will use Lemma \ref{lower hessian} in the two following simple
situations:
\begin{itemize}
\item if $p$ is a non-degenerate node of a complex curve $\C$,
then  $(\C. Hess_\C)_p\ge 6$;
\item if the curve $\C$ has a unique
branch at a point $p$, then $(\C. Hess_\C)_p\ge 3M_pm_p-2M_p
-2m_p$ where $m_p$ is 
the multiplicity of $C$ at $p$, and $M$ is the maximal order of
contact of a line with $\C$ at $p$ (note that $M_p>m_p$ and that there
exists a unique line $\L$ such that $(\C.\L)_p>m_p$).
\end{itemize}
\end{exa}

Consider  a 
plane $\P\subset (\CC^*)^N$
 and a tropical morphism $f:C^{\prime} \to \trp(\P)$, denote the image by 
 $C = f(C^{\prime})$. 
As usual, given a primitive $N$-simplex $\Delta$ 
giving a degree $1$ compactification of $\P$, 
denote by   $u_0,\ldots,u_N$ 
 the outward primitive integer  vectors
normal to the faces of $\Delta$. 
We also denote by $\A$ the line arrangement $\overline \P\setminus \P$.
We define the three following subsets of $\Ed(C^{\prime})$:

\begin{itemize}
\item[]$Bis_I(C')=\{e\in\Ed(C^{\prime}) \ | \ u_{f,e}=
u_I
\}$,
\item[]$Bis(C')=\bigcup_{\bold{p}_I\in\bold{p}(\A)} Bis_I(C')$,
\item[]$K_w(C')=\{e\in\Ed(C^{\prime})\ | \ \exists i,\ u_{f,e}=u_i,\text{ and }w_{f,e}>1\}$,
\item[]$K_1(C')=\{e\in\Ed(C^{\prime})\ | \ \exists i,\ u_{f,e}=u_i,\text{ and }w_{f,e}=1\}$.
\end{itemize}

Note that 
the above defined sets  are dependent on the choice of $N$-simplex 
$\Delta$, used to obtain a degree $1$ compactification of $\P \subset (\CC^*)^N$ since they depend on the vectors
 $u_i$.
Finally, we denote by $m_I(C)$ the multiplicity
of the curve $C$ at the point $p_{I}$, i.e. its intersection
multiplicity at $p_{I}$ with the ray $u_I$.
For a simple expression of this multiplicity, 
let $Edge_I(C^{\prime}) = \{ e \in \Ed(C^{\prime}) \ | \  p_I \in  \overline{f(e)} \}$.
For an edge $e \in \Ed_I(C^{\prime})$, the 
vector $u_{f,e}$
has a unique expression 
$p_eu_I + q_eu_k$ for some $k \in I$. Then the multiplicity at point $p_I$ is 
given by, 
$$m_I(C) = \sum_{ e \in Edge_I(C^{\prime}) }  w_{f, e} p_e.$$
Again, the multiplicity $m_I(C)$ is dependent on the choice of $N$-simplex $\Delta$, when a choice exists. 

We can now state and prove  the main result of this section.

\begin{thm}\label{obstruction: hessian}
Let $\P\subset (\CC^*)^N$
 be 
a 
non-degenerate
plane, and $\Delta$ a primitive $N$-simplex
giving a degree $1$ compactification of $\P$. 
 Let $f:C^{\prime} \to \trp(\P)$
 be a tropical morphism such that 
 $C = f(C^{\prime}) $ 
 has
 $\deg_{\Delta}(C) >1$. 
If the morphism $f$  is finely approximable in $\P$,
then
$$3C^2 +2(N-2)\deg_{\Delta}(C) -  \sum_{\textbf{p}_{I}\in\textbf{p}(\A)} 2(|I|-2)m_I(C)
 -\sum_{e\in Bis(C')(C^{\prime})\cup K_w(C')}(3w_{f,e}-2) - 2|K_1(C')|\ge 0.  $$
\end{thm}
\begin{proof}
Suppose that $f:C^{\prime} \to \trp(\P)$ is finely approximable by $\F:\C\to\P$.
Since $\F(\C)$ is irreducible, we will identify $\C$ and
$\F(\C)$ in $\P$.
 Consider  $\overline \C$ the closure of $\C$
 in $\overline \P = \CC P^2$, and define $q_1,\ldots,q_s$ the 
 points in $\overline \C\cap\left(
\A\setminus\textbf{p}(\A)\right)$ for which the mulitplicity of
intersection is at least 2. We denote by $m_j$ this intersection multiplicity
at the point $q_j$.

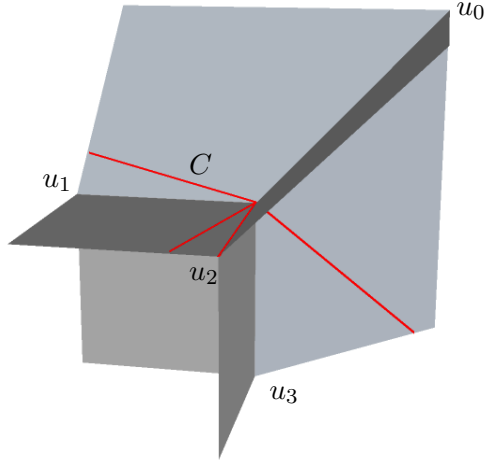
\begin{figure}
\includegraphics[scale=0.35]{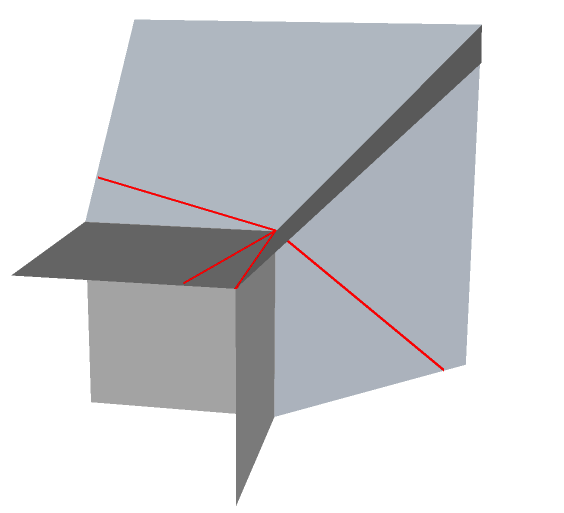}
\vspace{-0.3cm}
\put( -30, 170 ){$u_0$}
\put( -100, 25 ){$u_3$}
\put( -130, 70 ){$u_2$}
\put( -185, 105 ){$u_1$}
\put( -130, 110 ){$C$}
\caption{The $4$-valent curve from Corollary \ref{4-valent}.}
\label{fig:Vigeland4valent}
\end{figure}

To lighten slightly the notation, we denote by $d=deg_\Delta(C)$.
Since the number of intersection points of
$\overline \C$ with its 
Hessian in 
$\overline\P\setminus\{q_1,\ldots,
q_s,\textbf{p}_{I}\in\textbf{p}(\A)\}$ 
is non-negative and the Hessian is of degree $3(d-2)$, we have
\begin{equation}\label{equ hessian}
3d(d-2) -\sum_{\textbf{p}_{I}\in\textbf{p}(\A)} (\overline\C.
Hess_{\overline \C})_{\textbf{p}_I} -\sum_{j=1}^s (m_j-2) \ge 0.
\end{equation}
It is immediate from the definition of the weights of the tropicalisation of a 
morphism from Definition \ref{def:tropicalisationmorphism} that 
$$\sum_{j=1}^s (m_j-2) =  \sum_{e\in  K_w(C')}(w_{f,e}-2).$$

It remains to estimate the quantities $(\overline\C.
Hess_{\overline \C})_{\textbf{p}_{I}}$.   As before, we
denote  $\partial \overline \P=\overline \P\setminus \P$, 
and we claim that 
\begin{equation}\label{hessian local}
(\overline\C. Hess_{\overline \C})_{\textbf{p}_{I}}\ge  
3(\overline C^2)_{p_{I}}-
2(\overline \C. \partial \overline\P)_{\textbf{p}_{I}}
-2|Bis_{I}(C')| + 3\sum_{e\in Bis_I(C')}w_{f,e} + 2 (|I|-2)m_I(C).
\end{equation}
To prove inequality (\ref{hessian local}), 
we have to estimate the number of inflection points 
that are contained in   the Milnor fiber $F_I$ of $\overline\C$ at $\textbf{p}_{I}$. Let
us denote by 
$b_1,\ldots,b_k$ the local branches of $\overline\C$ at $p_I$. Note that these
  branches are in one to one correspondence with the edges
  of  $\Ed_I(C^{\prime})$.   
A small
  perturbation of $F_I$ can be constructed as follows: in
  a small Milnor ball centered at $\textbf{p}_{I}$, we translate
    each branch $b_i$ such that they intersect transversally; then
    we replace each singular point  by its Milnor
    fiber to obtain a surface $\Gamma$. 
Let us denote by $m_i$ the multiplicity of $b_i$ at $\textbf{p}_{I}$,
and by $M_i$  the maximal order of
contact of a line with $b_i$ at
$\textbf{p}_I$. 
Note that if $e_i\in Bis_I(C')$,
then $(b_i. \partial\overline \P)=|I|m_i$, and that  
if $e_i\notin Bis_I(C')$,
then $(b_i. \partial\overline \P)=(|I|-1)m_i+M_i$. In addition,  note
that if $e_i\in Bis_I(C')$, then $m_i=w_{f,e_i}$ and $M_i\ge m_i+1$.
Combining both the second and first part of Example \ref{simple hessian},
the number of inflection points contained in $\Gamma$ is at least, 
$$6\sum_{i\ne j}(b_i. b_j)_{\textbf{p}_{I}} + 
\sum_{e_i\notin  Bis_I(C')} (3M_im_i-2M_i-2m_i) +  
\sum_{e_i\in  Bis_I(C')} (3(m_i+1)m_i-4m_i-2). $$ 
By Definition \ref{def:localInt}, we have that 
$$\sum_{e_i\notin  Bis_I(C')} M_im_i + \sum_{e_i\in
  Bis_I(C')}m_i^2 + 2\sum_{i\ne j}(b_i. b_j)_{\textbf{p}_{I}}
=(\overline C^2)_{p_I}.$$
Combining the above two lines  we obtain Inequality (\ref{hessian local}).

Hence we get from Inequality
(\ref{equ hessian}) and (\ref{hessian local}) that
\begin{equation}\label{equ hessian2}
\begin{array}{l}
3d(d-2) -  \sum_{\textbf{p}_{I}\in\textbf{p}(\A)} 
\left(3 (\overline{C}^2)_{p_{I}} -
2(\overline \C. \partial \overline\P)_{\textbf{p}_{I}}
 + 2(|I|-2)m_I(C)\right)\\
\\ \quad  \quad \quad\quad \quad\quad\quad
 +2|Bis(C')| - 3\sum_{e\in Bis(C')}w_{f,e} -\sum_{e\in K_w(C')}(w_{f,e}-2)\ge 0
\end{array}
\end{equation} 
By Definition \ref{def:localInt}, we have
$$ \sum_{0\le i<j\le N}  (\overline{C}^2)_{p_{i,j}} = d^2 - C^2.$$
Summing up all intersection multiplicities of $\overline \C$ with
$\overline \P\setminus\P$, we get
$$\C.\partial \overline{\P}  =(N+1)d =
\sum_{\textbf{p}_{I}\in\textbf{p}(\A)} (\overline \C. \partial
\overline\P)_{\textbf{p}_{I}}
 + \sum_{e\in  K_w(C')}w_{f,e} + |K_1(C')|.$$
Plugging all of the latter equalities into  Inequality 
(\ref{equ hessian2}) we obtain the desired inequality.
\end{proof}

As an application, we prove the following 
corollary. 
\begin{cor}\label{4-valent}
Let $\P$ be a uniform plane in $(\CC^*)^3$, and denote 
by  $u_0,u_1,u_2$, and $u_3$ the primitive
integer directions of its four edges.
 Let $C$ be a fan 
tropical curve in $\trp(\P)$ with 
four rays of weight 1
in the
primitive integer directions 
$$u_1 + (d-1)u_2, \quad (d-1)u_1 +u_0 , \quad du_3+(d-1)u_0 , \quad \text{and} \quad
u_2, $$
(see Figure \ref{fig:Vigeland4valent}).
Then $C$ 
is approximable by 
a complex curve $\C\subset \P$ if and only if $d=1, 2$
or $3$. 
\end{cor}

\begin{figure}
\includegraphics[scale=0.3]{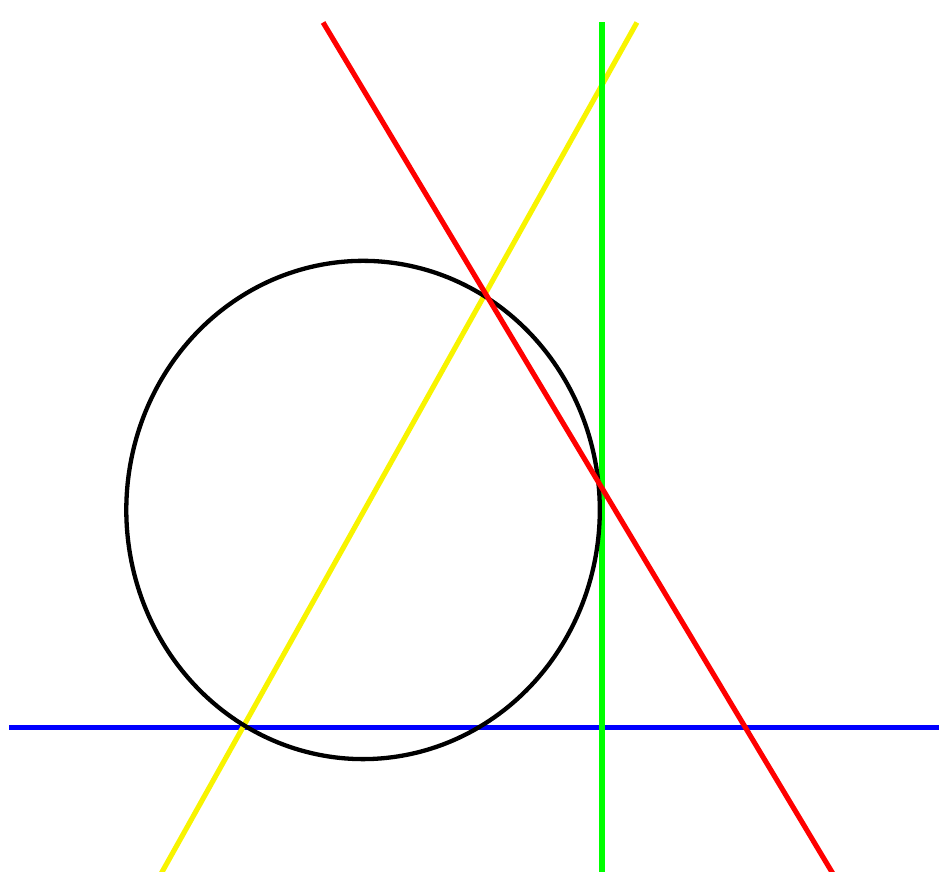}
\hspace{1cm}
\includegraphics[scale=0.4]{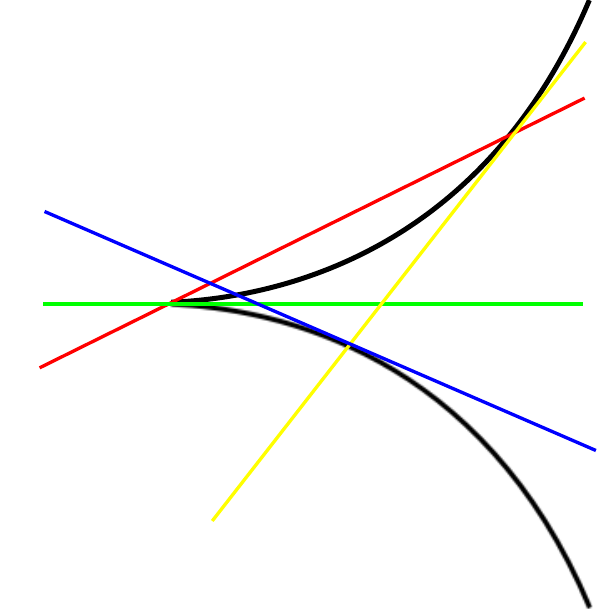}
\put(-198, 85){$\L_3$}
\put(-260, 0){$\L_1$}
\put(-252, 110){$\L_0$}
\put(-300, 20){$\L_2$}
\put(-280, 60){$\C$}
\put(-2, 55){$\L_3$}
\put(-90, 10){$\L_1$}
\put(-120, 40){$\L_0$}
\put(-122, 75){$\L_2$}
\put(-2, 5){$\C$}
\caption{The conic and the cuspidal cubic from Corollary \ref{4-valent} with their positions with respect to the $4$ lines.} 
\label{fig:Vigelandcusp}
\end{figure}
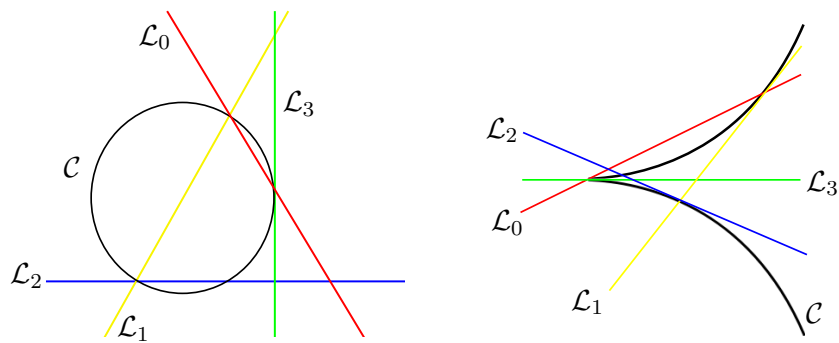

\begin{proof}
The tropical cycle induced by $C$ is irreducible, hence if an approximation
$\C$ exists it must be an irreducible and reduced curve. 
Also, as $\P$ is uniform 
there is a unique $N$-simplex $\Delta$ giving a degree $1$ compactification of $\P$. Moreover, $\deg_{\Delta}(C) = d$.

If $d = 1$ the tropical curve $C$ is the skeleton of the plane $\trp(\P)$ and it is approximated 
by a generic line $\L \subset \P$. 
For $d >1$ the closure of the tropical curve $\overline C \subset
\overline{\trp(\P)} \subset \TT P^3$ 
 contains the corner points $p_{i,j} \subset\overline{\trp(\P)}$ for $\{i, j\} = \{0, 1\}, \{0,3\}, \{1, 2\}$. 
 At these corner points we have the  polygons $$\Gamma^c_{i, j} = Conv\{ (0,0), (0,1), (d-1, 0)\}, \text{ for } \{i, j\} = \{0, 1\} \text{ and } \{1,2\}$$ 
 and $$\Gamma^c_{0, 3} = Conv\{(0,0), (0, d), (d-1, 0)\}.$$
It is immediate that $\C$ exists if $d= 2$ or
$3$; the curves relative to the line arrangements are drawn in Figure  \ref{fig:Vigelandcusp}.
For $d=3$ the curve has a cusp at $\textbf{p}_{0, 3}$ and simple tangencies  to $\L_1$ and  $\L_0$ at the points $\textbf{p}_{1, 2}$ and $\textbf{p}_{0, 1}$. 
For $d=2$, the curve is a conic 
 passing through the point $\textbf{p}_{0, 3}$ with a tangency
to the line $\L_3$, and also  
passing through  the points $\textbf{p}_{0, 1}$
and $\textbf{p}_{1, 2}$ having intersection multiplicity one  with 
the respective 
lines at these points.

 Let us now prohibit the remaining situations. An easy computation
 yields $C^2=2-d$, then according to Theorem  
\ref{obstruction: hessian},
  if $\C$ exists then 
$$3(2-d)+2d-2\ge 0$$
which reduces to $d\le 4$.
We are left to study by hand the case $d=4$. Choose 
coordinates on
$\overline {\P} = \CC P^2$  
such that
$\L_1, \L_0, \L_3$
are coordinates axes.
In this coordinate system,
$\Delta(\C)= Conv \{(0,4), (0,3), (1, 0)\}.$
It is not hard to
verify  (see for example {\cite[Lemma 6.4]{Br16}}) that such a curve cannot have
an inflection  point  on 
$\L_0\setminus\left(\L_1\cup\L_2 \right)$
\end{proof}

These curves will be revisited in Section \ref{sec:lines} as Vigeland
lines in a tropical surface.  

\end{section}

\renewcommand{\L}{{\mathcal L}}
\begin{section}{Trivalent fan tropical curves in planes.}\label{sec:aff}

 In this section we 
classify
all 
3-valent tropical curves
 $C\subset \trp(\P)\subset\RR^N$
which
 are finely approximable in a given non-degenerate plane $\P\subset(\CC^*)^N$.
  To this end we first focus 
on the case when $\P$ is a uniform hyperplane in $(\CC^*)^3$.

\subsection{Tropical curves in $\RR^3$ with $\aff_C\le 2$}
Given a fan tropical curve $C\subset \R^N$, we denote by
$\Aff(C)$ its 
affine span in $\RR^N$, and by $\aff_C$ the
dimension of $\Aff(C)$. The space
 $\Aff(C)$ has a natural tropical structure since
it is the tropicalisation of a binomial surface in $(\CC^*)^N$.

In \cite{BogKat}, T. Bogart and E. Katz gave some combinatorial obstructions
to the approximation of 
a fan tropical curve $C$ in a uniform tropical plane in $\RR^3$ satisfying
$\aff_C\le 2$.
Combining their
results with our own we obtain 
the complete classification of such approximable tropical curves.

\begin{thm}\label{plane curves}
Let $\P$ be a uniform plane in $(\CC^*)^3$, and let us denote by
$u_0,u_1,u_2$, and $u_3$ the four rays of $\trp(\P)$.
 Let 
$C\subset  \trp(\P)$
 be a  
reduced
fan tropical curve in $\RR^3$ with $\aff_{f(C)} \le
2$.  
Then 
the curve $C$
is finely approximable in $\P$
if and only if 
one of the two following conditions
hold
\begin{enumerate}
\item  
$C$
is equal to
  the tropical stable intersection of $\Aff(f(C))$ and $\trp(\P)$ (see \cite{St2});
\item 
$C$ has three edges $e_1$, $e_2$, and $e_3$ satisfying
$$w_{e_1}u_{e_1}=u_i + du_k, \quad
w_{e_2}u_{e_2}=u_j+du_l,\quad
w_{e_3}u_{e_3}=(d-1)(u_i + u_j)$$
with $\{i,j,k,l\}=\{0,1,2,3\}$.
In particular, if $d=1$ the fan tropical curve $C$ is $2$-valent 
with directions $u_i + u_j$, and $u_k + u_l$. 

Moreover in this case 
$C=f(C_0)$ where $C_0$ is a fan tropical curve with  $d+1$ edges, and
$f:C_0\to\trp(\P)$ is a tropical morphism finely approximable in
$\P$
by a rational curve with $d+1$ punctures. 
In particular, $f$ maps all edges of $C_0$ to $\trp(\P)$ with weight 1,
 and $(d-1)$
edges of $C_0$ are mapped to $e_3$ by $f$.
\end{enumerate}
\end{thm}

The proof of Theorem \ref{plane curves} decomposes into several
steps. 
Let us first recall 
two lemmas from \cite{BogKat}.

\begin{lemma}[Bogart-Katz, {\cite[Lemma 3.1]{BogKat}}]\label{bogkat 1}
Let $C$ be a tropical curve in $\RR^3$ such that $\aff_C\le 2$ and which
is 
approximable by a reduced and
irreducible complex algebraic curve $\C\subset(\CC^*)^3$.
Then there exists a reduced and irreducible
binomial algebraic surface $\H\subset (\CC^*)^3$ such that $\trp(\H)=\Aff(C)$
and $\C\subset \H$.
\end{lemma}

\begin{lemma}[Bogart-Katz, {\cite[Lemmas 8.5, 8.6, 8.7, and 8,8]{BogKat}}]\label{bogkat 2}
Let 
$\P\subset(\CC^*)^3$ be a uniform plane, and let
 $\H\subset (\CC^*)^3$
be a 
reduced and irreducible
binomial surface. Then either the curve  $\P \cap \H$ is non-singular
or it has a
unique singular point in $(\CC^*)^3$, which is a node.
Moreover, if $\P \cap\H$ has two
irreducible components $\C_1$ and $\C_2$, then 
the embedded tropical curve $\trp(\C_1)\cup\trp(\C_2)$ is 4-valent,   
the two 
tropical curves $\trp(\C_1)$ and 
$\trp(\C_2)$ are at most 3-valent, and at least one of them is 2-valent.
\end{lemma}

Lemma \ref{bogkat 2} implies immediately that a fan tropical
curve  $C\subset\RR^3$ with $\aff_C\le 2$ which is finely
 approximable in a uniform plane $\P\subset(\CC^*)^3$
and  not equal to the
tropical stable intersection of $\trp(\P)$ and $\Aff(C)$ must 
be either 2 or 3-valent.

\begin{lemma}\label{self int}
Let $\P\subset (\CC^*)^3$ be a uniform plane, and let $\H\subset
(\CC^*)^3$ be a reduced and irreducible binomial surface. We denote by 
$\Delta_{\P, \H}$  the Newton polytope of
the tropical surface $\trp(\P)\cup \trp(\H)$, and by $\X(\Delta_{\P, \H})$ the
toric variety 
defined by $\Delta_{\P, \H}$. 
 Let  $\overline \P$ and $\overline \H$ be respectively the closure of
 $\P$ and $\H$ in $\X(\Delta_{\P, \H})$, and let
$\overline \C=\overline \P\cap \overline \H$.
Then the curve $\overline \C$ is reduced and 
$\overline \C^2=0$ in $\overline \P$. Moreover, 
if  $\overline \C$ is reducible,  then $\overline \C$ has exactly
two irreducible components 
  $\overline \C_1$ and $\overline \C_2$, and 
$\overline \C_1^2=\overline \C_2^2=-1$ in $\overline \P$.

In particular, $\trp(\C_1)^2=\trp(\C_2)^2=-1$ in $\trp(\P)$.
\end{lemma}
\begin{proof}
We define $\C=\P\cap \H=\overline\C\cap(\CC^*)^3$. 
 According to Lemma \ref{bogkat 2}, the curve $\C$ 
has at most one
 singular point, so it has to be reduced and cannot have more
 than two irreducible components. 
Hence the same is true for  $\overline\C$.
Since $\overline\H^2=0$, we also have
 $\overline\C=0$ in $\overline\P$.
Suppose that $\overline \C$ has two
 irreducible components $\overline \C_1$ and $\overline \C_2$. 
Since 
$\overline\H^2=0$
we have
 $$\overline \C_1. \overline \C=\overline \C_2  .  \overline \C=0$$
which implies that 
$$\overline \C_1^2 +  \overline \C_1. \overline \C_2 = \overline
\C_2^2 + \overline  \C_1.\overline \C_2=0.$$
So we are left to show that $\overline  \C_1.\overline \C_2=1$. Since
the curve $\overline \C$ is reducible, it follows from 
 Lemma \ref{bogkat 2} that $\C$ has a unique singular point, which
 is a node, in $(\CC^*)^3$. Hence the result will follow from the fact
that $\overline \C$ intersects the boundary
 $\X(\Delta_{\P, \H})\setminus (\CC^*)^3$ transversally at non-singular points
 of  $\overline \C$.

To prove this last claim, we may assume that $\H$ is a subtorus of
$(\CC^*)^3$. In this case, there is a surjection
$\phi:\text{Hom}((\CC^*)^3,\CC^*)\otimes\RR\to
\text{Hom}(\H,\CC^*)\otimes\RR$. Moreover, if $Q\in Hom((\CC^*)^3,\CC^*)$ is
an equation of $\P$ in $(\CC^*)^3$, then $\phi(Q)$ is an equation of $\C$ in
$\H$. In particular, the
Newton polygon of $\phi(Q)$ is dual to the tropical curve
$\trp(\C)$, seen as a tropical curve in $\trp(\H)$. According to
Lemma \ref{bogkat 2}, the tropical curve $\trp(\C)$ is 4-valent, so
the Newton polygon of $\phi(Q)$ is a quadrangle. The polynomial $Q$ has exactly
4 monomials and $\phi(Q)$ has no fewer  monomials than $Q$, so we get that $\phi(Q)$
also has exactly 4 monomials. In particular, the only non-zero coefficients
of $\phi(Q)$ are the vertices of its Newton polygon. This implies that
$\overline\C$  intersects the boundary
 $\overline\H\setminus \H$ transversally at non-singular points
 of  $\overline \C$.
\end{proof}

To prove
Theorem \ref{plane curves}, 
we need
to list all possible 3-valent fan tropical  
plane
curves $C$ with $\aff_C\le 2$ 
and to compute $C^2$ for
each of them.  This is the content of the next lemma. 

\begin{lemma}\label{list}
Let $\P\subset(\CC^*)^3$
 be a uniform plane, and let $C\subset \trp(\P)$ be an irreducible 
3-valent fan tropical curve with $\aff_C\le
2$. 
Let $u_0,u_1,u_2$, and $u_3$ be the primitive integer directions of
the edges of $\trp(\P)$, and let 
$e_1,e_2$, and
  $e_3$  denote the edges of $C$.
Then, up
to the action of $\S_4$ on the rays of $\trp(\P)$ and re-ordering of the edges of
    $C$, 
the curve $C$
is one of the following types:
\begin{enumerate}
\item There exists $0\le \alpha,\beta$ with $\gcd(d,\alpha,\beta)=1$ 
and $\alpha+\beta\le d$, such that 
$$w_{e_1}u_{e_1}= du_1 + \alpha u_2,\quad 
w_{e_2}u_{e_2}= \beta u_2 + du_3,\quad \text{and}\quad
w_{e_3}u_{e_3}= (d-\alpha - \beta )u_2 + du_4,$$
see Figure \ref{fig:affcurves} (1).
In this case the curve $C$ is the tropical intersection of 
$\trp(\P)$ and $\Aff(C)$, and $C^2=0$;

\item There exists $0\le \alpha , \beta\le d$  with
  $\gcd(d,\alpha,\beta)=1$ 
such that 
$$w_{e_1}u_{e_1}= du_1 + \alpha u_2,\quad 
w_{e_2}u_{e_2}= (d-\alpha)u_2 + (d-\beta)u_3,\quad \text{and}\quad
w_{e_3}u_{e_3}= \beta u_3 + du_4,$$
see Figure \ref{fig:affcurves} (2).
In this case, $C^2=-\alpha\beta$;

\item There exists $0\le \alpha < \beta\le d$  with
  $\gcd(d,\alpha,\beta)=1$ 
such that 
$$w_{e_1}u_{e_1}= \alpha u_1 + \beta u_2,\quad 
w_{e_2}u_{e_2}= (d-\alpha)u_1 + (d-\beta)u_2,\quad \text{and}\quad
w_{e_3}u_{e_3}= du_3 + du_4,$$
 see Figure \ref{fig:affcurves} (3).
In this case, $C^2=-d^2 +\beta d -\alpha\beta$.
\end{enumerate}
\end{lemma}

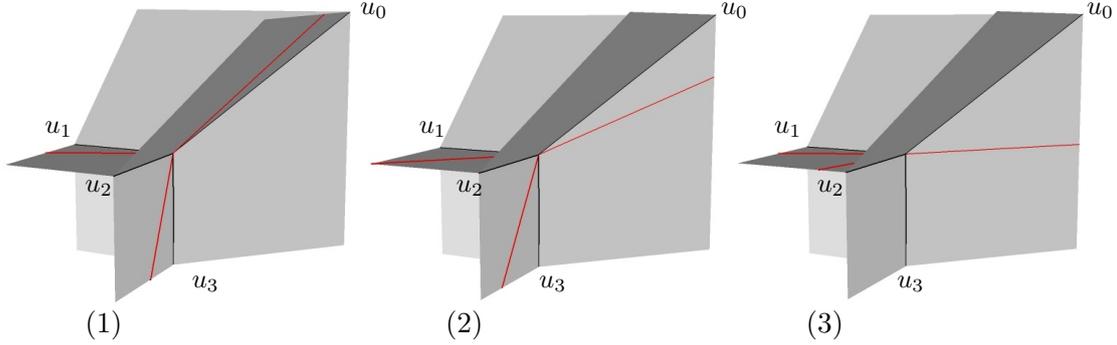
\begin{figure}
\includegraphics[scale=0.27]{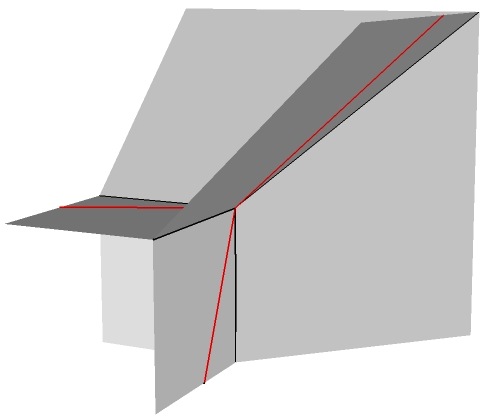}
\includegraphics[scale=0.27]{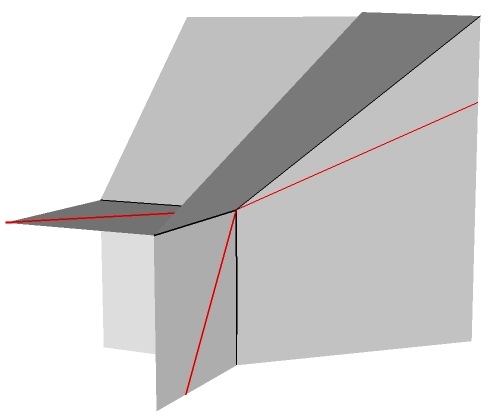}
\includegraphics[scale=0.27]{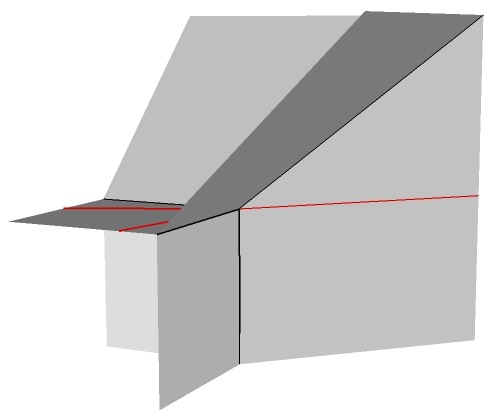}
\put(-375, -10){(1)}
\put(-240, -10){(2)}
\put(-105, -10){(3)}
\put(-335, 7){\small{$u_3$}}
\put(-375, 43){\small{$u_2$}}
\put(-390, 65){\small{$u_1$}}
\put(-272, 110){\small{$u_0$}}
\put(-205, 7){\small{$u_3$}}
\put(-236, 43){\small{$u_2$}}
\put(-250, 65){\small{$u_1$}}
\put(-137, 110){\small{$u_0$}}
\put(-71, 7){\small{$u_3$}}
\put(-101, 43){\small{$u_2$}}
\put(-117, 65){\small{$u_1$}}
\put(0, 110){\small{$u_0$}}
\caption{The three types of curves from Lemma \ref{list}.}
\label{fig:affcurves}
\end{figure}

Note that  cases $(1)$ and $(2)$ when $\alpha=\beta=0$ (and
consequently $d=1$) coincide with
the case $(3)$ when $\alpha=0$ and $\beta=d=1$.
In addition, notice that when $d >1$ and $\alpha, \beta >0$, an
approximation of a  curve 
from case $(1)$ must intersect the line arrangement
$\overline\P\setminus \P$
  in only 
three collinear points 
$\textbf{p}_{1, 2}, \textbf{p}_{2, 3}, \textbf{p}_{2, 4}$. Under the same conditions 
on $d, \alpha$ and $\beta$, an approximation
of a curve from 
case $(2)$ 
must intersect the arrangement in only the three points $\textbf{p}_{1, 2}, 
\textbf{p}_{2, 3}, \textbf{p}_{3, 4}$. Finally an approximation of a curve from 
case $(3)$ may only intersect the arrangement in two points $\textbf{p}_{1, 2}, \textbf{p}_{3, 4}$, if we assume
again the same restrictions on $d, \alpha$ and $\beta$.

\begin{proof}
The intersection numbers  follow from a direct computation. In case $(1)$, we
have to prove in addition that the curve $C$ is the tropical
intersection of 
$\trp(\P)$ and $\Aff(C)$,
which is non-trivial only for $\alpha\ne 0$
and $\beta\ne 0$. 
If $C'$ denotes this tropical intersection,
it is clear that $C$ and $C'$ have the same underlying
sets.
Since $C$ is irreducible, it remains to prove that $C'$ is also
irreducible. 

Without loss of generality, we can assume that 
$$u_0 =(1,1,1), \quad u_1=(-1,0,0),\quad u_2=(0,-1,0), \quad \text{and} \quad u_3=(0,0,-1). $$
The surface $\Aff(C)$ is given by a classical linear equation of the form
$$ax + by + cz = 0 \quad  \text{with}\quad  \gcd(a,b,c)=1.$$
Let us denote by $w_{1,2}$ (resp. $w_{2,3}$) the weight of the edge of $C'$
lying in the convex cone spanned by $u_1$ and $u_2$ (resp. $u_3$ and $u_2$).
A computation gives
$w_1=\gcd(a,b)$, and $w_2=\gcd(b,c)$. Hence $w_1$ and $w_2$ are relatively
prime and $C'$ is irreducible, which implies 
 $C^{\prime} = C$. This completes the proof. 
\end{proof}
\begin{rem}\label{rem stable int}
Note that the same proof gives that the tropical stable intersection of
\textit{any} tropical surface of degree 1 in $\R^3$ made of an edge
and 3 faces with \textit{any}
 non-singular 
binomial tropical surface 
in $\R^3$ is \textit{always}
 irreducible. 
\end{rem}

\begin{proof}[End of the proof of Theorem \ref{plane curves}:]
It follows from Lemma \ref{self int} that if the 
tropical curve $C$
is 
finely
approximable 
in $\P$, then $C^2=0$ or $-1$ in $\trp(\P)$.
Now, it follows from Lemma \ref{list} that
\begin{itemize}
\item   
the case when $C^2=0$
corresponds to 
the case $(1)$ of Theorem \ref{plane curves}. The approximation of
such a 
tropical curve
follows from the
fact that tropical stable intersections are 
always approximable (see {\cite[Theorem 5.3.3]{Pay1}}).

\item the case when 
$C^2 = -1$
corresponds to 
the case $(2)$ of Theorem \ref{plane curves}. Up to automorphism of $\CC P^2$
we are free to fix the $4$ generic lines of the arrangement corresponding to $\P$ 
as we wish. Let $\L_2$ be the line at infinity and 
$$\L_1 = \{ y = 0\},  \quad \L_3= \{ x=0 \}, \quad  \text{ and } \quad \L_4 = \{x+y-1 = 0  \}.$$
The existence of an
approximation in $\P$ of 
the
tropical morphism 
$f:C_0\to\trp(\P)$
is equivalent to the
existence of an irreducible complex algebraic curve $\C$ in $\CC^2$ defined
by the equation $yP(x) -1=0$ where $P(x)$ is a complex polynomial of
degree $d-1$ with no multiple root, and such that $\C$ has order of
contact $d$ at $(0,1)$ with the line $\L_4$. 
It is not hard to check that one has to have $P(x)=1+x+\ldots+x^{d-1}$.
The curve $\C$  has the correct order of contact with each of the lines $\L_i$ of 
$\overline \P \backslash \P$ in order to tropicalise to $C \subset P$,
and it is immediate 
that $\C$ is 
a rational curve with $d+1$ punctures.
\end{itemize}
This completes the proof of the theorem. 
\end{proof}

Notice that by Corollary \ref{cor:CC<0}, the above constructed complex curve $\C$ is
the unique curve which 
approximates a tropical curve $C$ from Case (2) of Theorem  \ref{plane curves}, 
 since $C^2 = -1$. 

\subsection{Classification of 3-valent fan tropical curves in a
  tropical plane}
The complete classification of 
fan tropical
curves $C\subset \RR^3$  
with $\aff_C \leq 2$ which are finely approximable 
in a non-degenerate plane in $(\CC^*)^3$
can be extended to classify all 
3-valent fan tropical curves 
finely
approximable in a given plane of any codimension. 
To this aim the next lemma determines whether a tropical curve $L
\subset \trp(\P)$ 
with $\deg_\Delta(L) = 1$ 
and any valency is approximable.

\begin{lemma}\label{lines}
Let $\P \subset (\CC^*)^N$ be a non-degenerate plane and $\Delta$ 
a primitive $N$-simplex 
giving a degree $1$ compactification of $\P$. 
Let $L \subset \trp(\P)
\subset \RR^N$ be a $k$-valent fan tropical curve with  $\deg_{\Delta}(L)
=1$. Then $L$ is approximated by 
some line
$\L \subset \P$  if and only if for some $0 \leq m \leq k$ the
arrangement determined by $\P$ contains $m$ collinear  points
$\textbf{p}_{I_1}, \dots , \textbf{p}_{I_m}$ such that the sets  $I_1,
\dots,  I_m$  are  disjoint and $I_1 \cup \dots \cup I_m = \{0, \dots
, N\} \backslash  J$ where $|J| = k -m$; 
 moreover, the $k$ rays of
$L$ are  in the directions
$$u_{I_1}, \dots u_{I_m}, \quad \text{and} \quad   \{u_i \ | \ i\in J\}, $$
and 
$\L$ is 
the 
line passing through the $m$ collinear points $\textbf{p}_{I_i}$.  
\end{lemma}
\begin{proof}
Suppose a curve $\L$ approximates $L$. Since $\deg_{\Delta}(L) = 1$,
the closure
$\overline{\L}$ of $\L$ must be a line in the compactification of $\P$ to
$\overline{\P} = \CC P^2$ 
given by $\Delta$. 
Therefore $\overline{\L}$  intersects each
line in the arrangement determined by $\P$ exactly once. Since
$L$ is $k$-valent and each edge is of weight one, the line $\overline{\L}$ may intersect
the arrangement in only $k$ points. 
These $k$ intersection points induce a partition of the set of lines
in the arrangement into $k$ subsets.
The subsets of size greater
than one correspond to points $\textbf{p}_{I_i}$  of the arrangement
through which  the line $\L$ passes. 
\end{proof} 

Before considering the general case we remark that not every 
tropical plane $P \subset \RR^N$ contains trivalent fan tropical 
curves. In fact, in order for $P$ to contain a trivalent curve there
must exist three sets  $I_1, I_2, I_3$
satisfying: $I_1 \cup I_2 \cup I_3 = \{0, \dots, N\}$ and if  $|I_i|>1$ then 
$p_{I_i}$ is a point of the corresponding line arrangement.

Given two line arrangements
$\A \subset \A^{\prime}$ 
in $\CC P^2$, 
there is a natural inclusion of their
respective planes $i :\P^{\prime} \hookrightarrow \P$.

\begin{lemma}\label{exceptional conics}
Let $\P \subset (\CC^*)^3$ 
be a uniform plane,
and 
denote
the lines of the associated arrangement by 
$\L_i,\L_j,\L_k,$ and $\L_l$.
In addition, let $\C_2 \subset \P$ be the degree $2$ curve
from part $(2)$ of Theorem \ref{plane curves}. Then the further 
fan tropical curves are finely approximated in the plane $\P'$ in
following cases (see Figure  \ref{fig:exceptional conics}):

\begin{enumerate}
\item  the plane $\P' \subset (\CC^*)^4$ corresponds to the arrangement
  of $5$ lines obtained by adding to the
  arrangement of $\P$ the unique line $\L_m$  which passes through the
  two points
  $\textbf{p}_{i,k}$, $\textbf{p}_{j, l}$;  
the three rays of $C \subset \trp(\P')$ are of weight one with
primitive integer directions $$u_i +u_j, \quad u_i + 2u_k + u_m,
\quad \text{and} \quad  u_j + 2u_l + u_m.$$  

\item the plane $\P' \subset (\CC^*)^4$ corresponds to the arrangement of $5$
  lines obtained by adding to the arrangement of
  $\P$ the unique line $\L_n$  which is tangent to $\C_2$ at the point
  $\textbf{p}_{i,j}$;  
the  three rays of $C_2 \subset \trp(\P')$ are of weight one with
primitive integer directions $$ u_i + u_j + 2u_n,  \quad u_i+  2u_k,
\text{ and }  \quad  u_j +2u_l .$$ 

\item the plane $\P' \subset (\CC^*)^5$ corresponds to the arrangement of $6$
  lines obtained from the arrangement of $\P$ by adding the lines
  $\L_m$ and $\L_n$ from parts $(1)$ and $(2)$;
the three rays of $C \subset \trp(\P')$ are of weight one and with
primitive integer directions $$u_i + u_j +2u_m, \quad u_i +2u_k +u_n,
\quad \text{and} \quad u_j + 2u_l + u_n.$$ 
\end{enumerate}
\end{lemma}
\begin{proof}
In each of the three above cases the tropical curves are approximated by the curve $\C = \C_2 \cap \P^{\prime}$. 
\end{proof}

The curve in case (1) of Lemma \ref{exceptional conics} has already appeared in Example \ref{ex:cremonaCurve}. It is remarked in this example that there is a choice of compactification of $\P$ given by the simplex  $\Delta^{\prime}$ so that the closure of the curve $\C$ approximating $C$ is a line. So in fact, case (1) above is also contained in Lemma \ref{lines}.

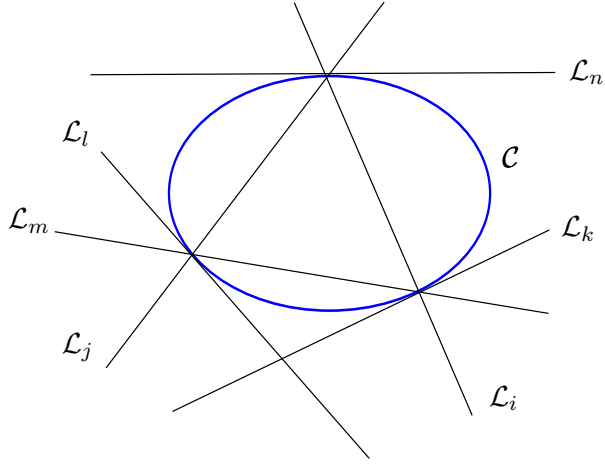
\begin{figure}
\includegraphics[scale=0.5]{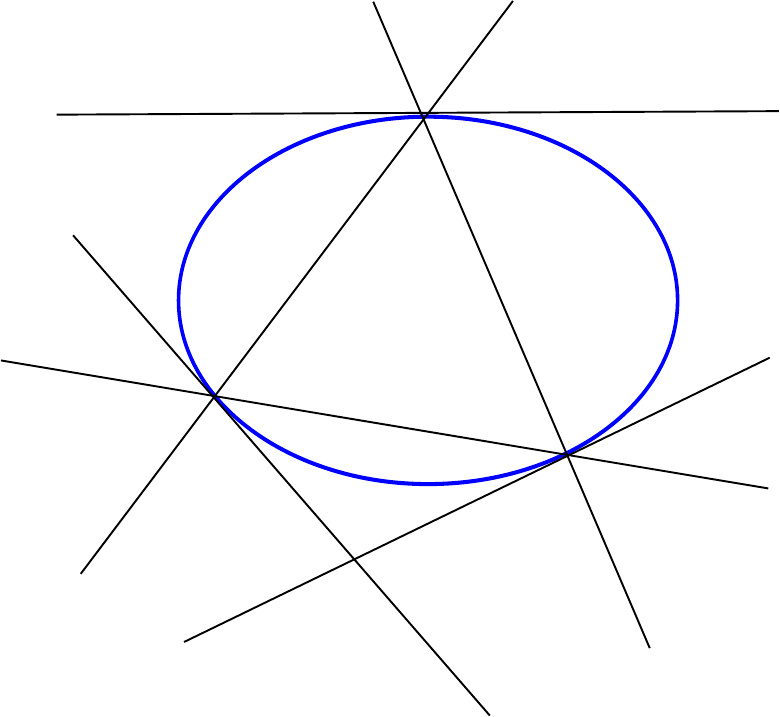}
\put(-25, 20){$\L_i$}
\put(-185, 40){$\L_j$}
\put(-205, 87){$\L_m$}
\put(5, 143){$\L_n$}
\put(-20, 110){$\C$}
\put(-20, 110){$\C$}
\put(2, 84){$\L_k$}
\put(-185, 120){$\L_l$}

\caption{The curve $\C$ from Lemma \ref{exceptional conics} with respect to the lines indexed by $i, j, k, l, m, n$.}
\label{fig:exceptional conics}
\end{figure}

\begin{thm}\label{trivalent thm}
Let  $N\ge 3$, let $\P \subset (\CC^*)^N$ be a non-degenerate plane,
and let $C \subset \trp(\P)$ be an irreducible 2 or 3-valent fan tropical
curve.  Then 
the curve $C$
is finely
 approximable 
in $\P$
if and only if we are in
one of 
the following cases:
\begin{enumerate}
\item there exists a primitive $N$-simplex $\Delta$ 
giving a degree $1$ compactification of $\P$ 
such that $\deg(C)_{\Delta} =1$ and $C$ and $\P$ satisfy Lemma \ref{lines}; 
\item $C$ and $\P$ satisfy one of  the three 
situations described in
Lemma \ref{exceptional conics};
\item 
the plane $\P\subset(\CC^*)^3$ is  non-uniform 
and $C$ is any irreducible trivalent fan tropical curve;
\item 
the plane $\P\subset(\CC^*)^3$ is  uniform 
and $C$ is a trivalent curve from part $(2)$ of Theorem \ref{plane
  curves} or part $(1)$ of Lemma \ref{list}. 
  
\end{enumerate}
\end{thm}

As a remark we mention that 
the trivalent lines
 with $N\ge 6$ in case $(1)$  of Theorem
\ref{trivalent thm} and 
the curves of case $(2)$ and $(3)$ of Lemma
\ref{exceptional conics} are  
\textbf{exceptional}, in the sense that for a generic choice of plane
$\P$ which tropicalises to the fans in each of these 
cases (i.e. whose
line arrangement has the right intersection lattice), the
corresponding tropical
curve  will not be approximable. 

\begin{proof}
All of the above 
tropical curves
were shown
to be approximable 
in the corresponding plane
in Lemmas \ref{lines}, \ref{exceptional conics}, Remark \ref{rem stable int},
and Theorem \ref{plane curves}.

Let $C$ be an irreducible 2 or 3-valent fan tropical
curve which is finely approximable by a curve $\C$ in some plane $\P$. 
It remains to show that 
the pair $(\P,\C)$ is one of those described in the theorem.
According to Lemma \ref{lines}, this is true if $\deg_{\Delta}(C)=1$, for 
some primitive $N$-simplex $\Delta$ 
giving a degree $1$ compactification of $\P$,
so let us
suppose that $\deg_{\Delta}(C)\ge 2$ for all such $\Delta$.

\vspace{2ex}
Suppose first that the arrangement determined by $\P$ contains a
uniform subarrangement $\A_0$ of $4$ lines 
$\L_i,\L_j,\L_k,$ and $\L_l$
yielding a plane $\P_0$.
  Then there is a natural inclusion $\P
\hookrightarrow \P_0$. If $\C \subset \P$ approximates  $C$, let
$\C_0 \subset \P_0$ be the closure of $\C$ in $\P_0$ and
$C_0\subset\trp(\P_0)$ its tropicalisation. 
Since the plane $\P_0$ is uniform there is a unique $3$-simplex $\Delta_0$ giving 
a degree $1$ 
compactification of  $\P_0$, and the degree of $C_0$ is well-defined, denote it $\deg(C_0)$. 
Then we may find a degree $1$ compactification of $\P$ given by a $N$-simplex $\Delta$ such that  $\deg(C_0) = \deg_{\Delta} (C)\ge 2$. 
This implies by Theorem \ref{plane curves} that the curve $C_0$ is
trivalent. 
Therefore $C_0$  is either
of type $(1)$ from Lemma \ref{list} or from case $(2)$ of Theorem
\ref{plane curves}. 

Suppose it is the former. 
Since $\deg(C_0)\ge 2$, up to relabeling the four lines in $\A_0$, 
it follows from Lemma \ref{list} that 
 $\C_0 \subset \P_0$
intersects this uniform arrangement 
in the three
points $\textbf{p}_{i, j},  \textbf{p}_{i, k},  \textbf{p}_{i, l}$.
Let $\L$ be another line of the arrangement $\A$
determined by $\P$. 
Since $C$ is
trivalent,  the line $\L$ must intersect $\C$ with multiplicity
$\deg(C)$ at one of these three points,
which is impossible according to Lemma \ref{list}.

If $C_0 \subset \trp(\P_0)$ satisfies part $(2)$ of Theorem
\ref{plane curves}, then $\C_0$ is the unique curve given  in the
proof of Theorem \ref{plane curves} (up to relabelling the four lines in 
$\A_0$). This curve intersects the
arrangement $\A_0$ in the points $\textbf{p}_{i,j}, \textbf{p}_{i, k},$ and
$\textbf{p}_{j,l}$.  
Note that the only singular point of $\overline \C\subset
\overline\P$ may be at the point $\textbf{p}_{i,j}$ and that 
 the tangent line to $\overline \C_0$ at the points 
$\textbf{p}_{i,k}$ and $\textbf{p}_{j,l}$ is already contained in the
arrangement
$\A_0$. Therefore, any
other line passing through  $\textbf{p}_{i,k}$ or
$\textbf{p}_{j,l}$ has intersection multiplicity $1$ with $\overline
\C_0$ at this point. 
Let $\L$ be a line in $\A\setminus\A_0$. Since the tropical curve
$C$
is 
trivalent, 
we are in one of the following two situations:
\begin{enumerate}
\item the line $\L$ passes  through the points $\textbf{p}_{i, k},
  \textbf{p}_{j,l}$, and 
the sum of the intersection multiplicities of $\overline \C$ and $\L$ at
these two points is equal $\deg(C)$; since the intersection
multiplicity of $\overline \C$ and $\L$ is 1 at these points, this
is possible only if 
$\deg_\Delta(C)=2$;

\item the line $\L$ passes  through the point $\textbf{p}_{i, j}$, 
$\textbf{p}_{i, k}$, or $\textbf{p}_{j, l}$,
 and intersects $\overline \C$ with
multiplicity $\deg(C)$ at this point;  since the intersection
multiplicity of $\overline \C$ and $\L$ is 1 at $\textbf{p}_{i, k}$
and $\textbf{p}_{j, l}$, the line $\L$ necessarily passes through
$\textbf{p}_{i, j}$, which is an ordinary point of multiplicity $d-1$
of $\overline\C$; since 
$C$ is 3-valent, the line $\L$ must have
the same intersection multiplicity with all local branches of
$\overline\C$ at $\textbf{p}_{i, j}$, which is possible only if 
$\deg_\Delta(C)=2$.
\end{enumerate}
Hence if we are not in cases (1) or (4) from the statement of the 
theorem, we are necessarily in
case (2).

\vspace{2ex}
If the arrangement $\A$ 
does not contain a
uniform subarrangement of $4$ lines then 
according to Lemma \ref{uniform subarr} below,
all
but one line of $\A$ must belong to the same pencil.
Then, $\Delta(\P)$ is $N-1$ dimensional and we must
choose an $N$-simplex $\Delta$ containing $\Delta(\P)$ and 
giving a degree $1$ compactification of $\P$. 
The arrangement $\A$
contains 
a subarrangement $\A'_0$ of $4$ lines, $3$ of which belong
to the same pencil, and defining a plane  
$\P_0^{\prime} \subset (\CC^*)^3$. As previously,  if $\C \subset \P$ approximates  $C$, let
$\C_0' \subset \P_0'$ be the closure of $\C$ in $\P_0'$ and
$C_0'\subset\trp(\P_0')$ its tropicalisation.
From the $N$-simplex $\Delta$ yielding a degree $1$ compactification of
$\P$ we may obtain a $3$-simplex $\Delta^{\prime}_0$ compactifying 
$\P^{\prime}_0$ so that $ \deg_{\Delta_0^{\prime}}(C'_0) = \deg_{\Delta}(C) \geq 2$.
Since $C$ is a trivalent curve and  $\deg_{\Delta}(C'_0)\ge 2$, 
 the curve $C'_0$ must also be  trivalent. 
Hence according to
Remark \ref{rem stable int}, $C'_0$ is the tropical stable
intersection of $\trp(\P'_0)$ and $\Aff(C)$. If $\C'_0$ does not pass
through the triple point of $\A'_0$, then since $\trp(\C)$ is trivalent we
must have  $\A=\A_0'$. 
If $\C'_0$  passes
through the triple point of $\A^{\prime}_0$,
then there exist
$|I| -2$ lines of $\A$ such that each branch
of $\C$ at $\textbf{p}_I$  has order of contact 
$\deg_{\Delta}(C)$
 with these
lines. This implies that $|I|=3$, which
completes the proof. 
\end{proof}

\begin{lemma}\label{uniform subarr}
If $\A$ is a line arrangement not containing a  uniform subarrangement of $4$  lines then all but one of the lines in $\A$ are contained in the same pencil.
\end{lemma}
\begin{proof}
By assumption not all lines of $\A$ belong to the same pencil,  so
there is a subarrangement of $3$ lines
 $\L_i,\L_j,$ and $\L_k,$
that is uniform. Every other line in $\A$ must belong to the pencil determined by a pair of lines in this subarrangement, 
otherwise there would be $4$ lines forming a uniform
subarrangement. If two of the additional lines indexed by $l, m$
belong to different pencils given by say  $\textbf{p}_{i, j}$ and
$\textbf{p}_{i,k}$  then the subarrangement given by $j, k, l, m$ is
uniform, and we obtain a contradiction.
\end{proof}

\begin{proof}[Proof of Theorem \ref{plane cycles}]
Point (1) is a consequence of Theorem \ref{plane curves} and Lemma
\ref{list}.
Point (2) is contained in Theorem \ref{trivalent thm}.
\end{proof}

\end{section}

\renewcommand{\L}{{\mathcal L}}

\begin{section}{Application to tropical lines in tropical surfaces}\label{sec:lines}

In the space of tropical
surfaces of degree $d$ in $\TT P^3$, 
there exists an open subset of surfaces which  contain 
(potentially
infinitely many)
tropical lines. 
In this section we prove that a generic non-singular tropical surface $S$
in $\RR^3$ of degree 3 contains finitely many
tropical lines $L$ such that the pair $(S,L)$ is approximable. In addition,  a 
generic non-singular tropical surface $S$ of degree greater than $3$ contains 
no tropical lines $L$ such that the pair $(S, L)$ is approximable. 
Therefore, if we restrict to approximable lines the tropical situation is 
 analogous to the classical algebro-geometric one. 
 
Examples of generic
non-singular  tropical surfaces in $\RR^3$ 
of any degree
containing infinitely many lines
were first constructed by
Vigeland
in \cite{Vig1}. 
In \cite{Vig2}, 
Vigeland 
later classified by combinatorial type
all tropical lines in generic non-singular
tropical surfaces. 
In this entire section, we denote by 
$\Delta_d$ the simplex 
$$\Delta_d=Conv\{(0,0,0),(d,0,0),(0,d,0),(0,0,d) \},$$ 
and
by $F_1,\ldots, F_4$ its facets.

\subsection{1-parametric families of lines}
A simplex $\Delta\subset \Delta_d$ with vertices in $\ZZ^3$ is said to 
be \textbf{$d$-pathological} 
if $\Delta$ is primitive 
and if $\Delta$ has
one edge in $F_i\cap F_j$, one edge in $F_k$, and one edge in $F_l$
for $\{i,j,k,l\}=\{1,2,3,4\}$. 
\begin{thm}[Vigeland, {\cite[Theorem 5.11]{Vig2}}]\label{vig thm}
Let $S$ be a generic non-singular
tropical surface of degree $d\ge3$ in $\RR^3$ with Newton polytope
$\Delta_d$. If $S$ contains a 1-parametric family of tropical lines,
then there exists a vertex $p$ of $S$, dual to a $d$-pathological
simplex, such that any line in the family is contained in the fan
$p+Star_p(S)$. 

Conversely, any compact tropical surface $S$ in $\RR ^3$
 containing a $d$-pathological
simplex in its dual subdivision contains infinitely
many tropical lines.
\end{thm}
Examples of tropical surfaces with infinitely many lines from
\cite{Vig1} were constructed using the $d$-pathological simplex with
vertices $(0,0,0),(1,0,0),(0,1,d-1),$ and $(d-1,0,1)$.
It is easy to see that 
any
$d$-pathological simplex defines a tropical fan containing
 a 1-parametric family of tropical
lines. This family consists of a unique  4-valent line and the remaining lines have two 3-valent 
vertices. 
Here we prove
that among all of 
these families, there is only a 
single line which is approximable. Before
giving the rigourous statement, lets us first describe in detail
tropical fans with Newton polygon a $d$-pathological simplex.

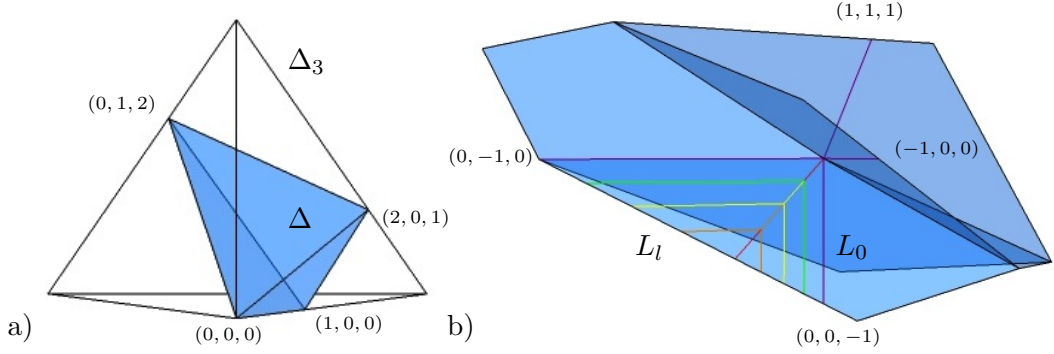
\begin{figure}
a)\includegraphics[scale=0.55]{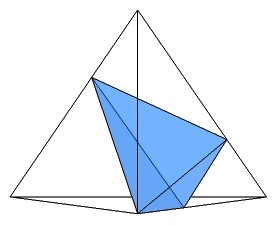}
b)\includegraphics[scale=0.45]{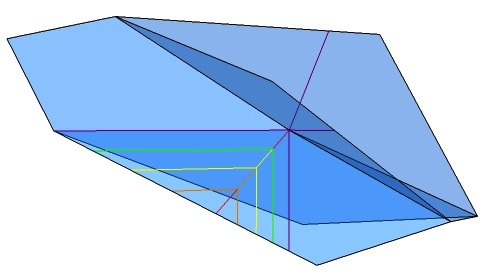}
\put(-290, 100){$\Delta_3$}
\put(-255,42){\tiny{$(2, 0, 1)$}}
\put(-365, 85){\tiny{$(0, 1, 2)$}}
\put(-280, 2){\tiny{$(1, 0, 0)$}}
\put(-325, -2){\tiny{$(0, 0, 0)$}}
\put(-290, 40){$\Delta$}
\put(-230, 65){\tiny{$(0, -1, 0)$}}
\put(-85, 30){$L_0$}
\put(-160, 30){$L_l$}
\put(-100, -3){\tiny{$(0, 0, -1)$}}
\put(-85, 120){\tiny{$(1, 1, 1)$}}
\put(-63, 68){\tiny{$(-1, 0,0 )$}}
\caption{a) A pathological $3$-simplex $\Delta$ from Theorem \ref{prohib Vigeland}
  drawn inside $\Delta_3$. b) The 
vertex of
  the surface 
dual to $\Delta$
and the infinite family of tropical lines. By Theorem \ref{prohib
    Vigeland} only $L_0$ is approximable in $S$.} 
\label{Delta in Deltad}
\end{figure}

Let $\Delta$ be such a 
$d$-pathological simplex with $d\ge 3$. 
Without loss of generality, we can suppose
that $\Delta$ has one edge in $F_1\cap F_2$. Hence either 
 $\Delta$ has one edge
in $F_3\cap F_4$, or  $\Delta$ has one edge in $F_3$ and the last one
in $F_4$. The first case is impossible since $\Delta$ would not be
primitive. Hence, up to permutation of the coordinates, the
vertices of $\Delta$ are
$$(0,0,0),\quad (1,0,0), \quad(0,d-\alpha,\alpha),\quad
(d-\beta-\gamma,\beta,\gamma)$$
 with the conditions that 
$$0<\alpha<d,\quad 0< \beta +\gamma<d,\quad
\left|\begin{array}{ccc}1&0&d-\beta-\gamma
\\ 0& d-\alpha & \beta \\ 0&\alpha&\gamma \end{array}
\right|=\gamma(d-\alpha) -\alpha\beta=1,$$
to ensure that $\Delta$ is primitive. See Figure \ref{Delta in Deltad} for an example.

Let $S$ be a tropical surface in $\RR^3$ 
with Newton polygon $\Delta$. Without loss of generality, we may
assume that the vertex of $S$ is the origin. 
By a computation
we get that $S$ has 4 rays with the 4 following
primitive outgoing directions 
$$u_0=(0,\alpha, \alpha-d),\quad u_1=(0,-\gamma,\beta),\quad
u_2=\left(-1,-\alpha(d-\beta-\gamma),
(d-\alpha)(d-\beta-\gamma )\right),$$
$$u_3=\left(1,\gamma+ \alpha(d-\beta-\gamma - 1),
-\beta-(d-\alpha)(d-\beta-\gamma - 1)\right). $$

The tropical surface $S$ 
contains the following one parameter family of tropical lines
$(L_{l})_{l\in\R_{\ge 0}}$: the tropical line $L_{l}$ has one
vertex $V_1$ at $(0,0,0)$
adjacent to 3 rays with outgoing directions $$U_1=(0,-1,-1), \quad
U_2=(-1,0,0), \quad \text{and} \quad U_3=(1,1,1),$$ and another vertex $V_2$ at
$(0,-l,-l)$ adjacent to 3 rays with outgoing directions $$-U_1, \quad 
U_4=(0,-1,0), \quad  \text{and}  \quad U_5=(0,0,-1).$$ If $l=0$, then $L_{0}$ is a
tropical line with one 4-valent vertex.
The tropical line $L_{l}$ is indeed in $S$ 
since we have (see Figure \ref{Delta in Deltad})
$$U_1=(\beta+\gamma)u_0 +du_1, \quad U_2=(d-\beta-\gamma)u_0 + u_2,$$ 
$$ U_3=(d-1)u_2 + du_3, \quad U_4=\beta u_0+ (d-\alpha )u_1, \quad 
U_5= \gamma u_0+ \alpha u_1. $$

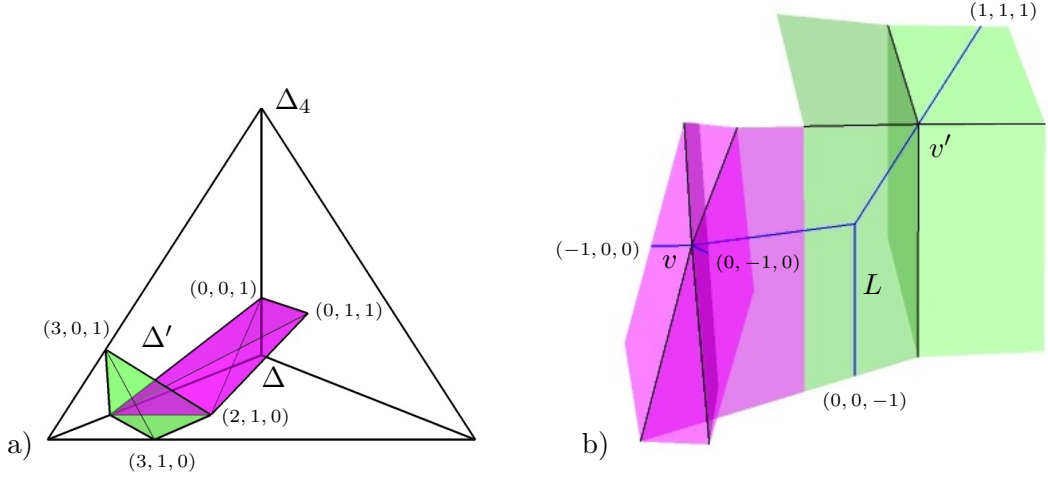
\begin{figure}
a)\includegraphics[scale=0.27]{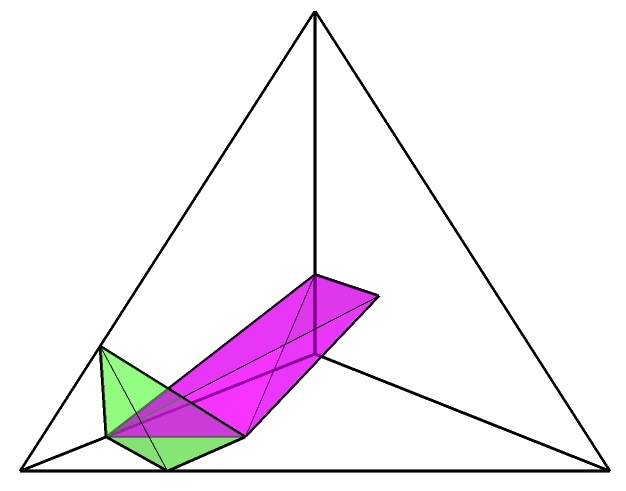}
\hspace{1cm}
b)\includegraphics[scale=0.5]{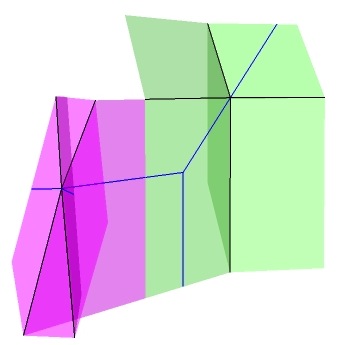}
\put(-295, 130){$\Delta_4$}
\put(-345, 40){$\Delta^{\prime}$}
\put(-300, 25){$\Delta$}
\put(-280, 52){\tiny{$(0, 1, 1)$}}
\put(-327, 60){\tiny{$(0, 0, 1)$}}
\put(-382, 45){\tiny{$(3, 0, 1)$}}
\put(-350, -5){\tiny{$(3, 1, 0)$}}
\put(-315, 12){\tiny{$(2, 1, 0)$}}
\put(-75, 60){$L$}
\put(-150, 70){$v$}
\put(-50, 112){$v^{\prime}$}
\put(-35, 165){\tiny{$(1, 1, 1)$}}
\put(-90, 18){\tiny{$(0, 0, -1)$}}
\put(-190, 75){\tiny{$(-1, 0, 0)$}}
\put(-130, 70){\tiny{$(0, -1, 0)$}}
\caption{a) A pair $(\Delta, \Delta^{\prime})$ of type $I$ in $\Delta_4$. 
b) The corresponding vertices $v, v^{\prime}$ along with the isolated line $L$.  \label{Vigtype1}}

\end{figure}

\begin{thm}\label{prohib Vigeland}
Let $\mathcal S\subset (\CC^*)^3$ be an algebraic surface with Newton
polytope a pathological $d$-simplex $\Delta$. 
The tropical line $Star_{(0,0,0)}(L_{l}) \subset S$ is approximable by a complex algebraic
line $\L_{l}\subset\mathcal S$ if and only if $l=0$ and $S$ has
Newton polytope 
 $Conv\{(0,0,0),(1,0,0),(0,1,2),(2,0,1)\}$.
\end{thm}
In the case of 3-valent tropical lines two instances of Theorem
\ref{prohib Vigeland} were already known, namely
the cases
 $(\alpha,\beta,\gamma)=(d-1,0,1)$ (\cite{BogKat}) and 
$(\alpha,\beta,\gamma)=(1,d-2,1)$ (\cite{Shaw}, see also Example
\ref{ex:CD<0}).  
\begin{proof}
The case $l > 0$ follows  from 
initial degeneration and
the classification given in
Theorem \ref{plane curves}. Suppose now that $l=0$.
We have that $C^2=-(\beta+\gamma)(d-1)+1$ which implies that
$$d + C^2 -\sum_{e\in\Ed(L_l)} w_e +2= -(d-1)(\beta+\gamma-1). $$
Hence, if $\beta+\gamma>1$, then the result follows from Theorem
\ref{thm:simpadjunction}. If $\beta+\gamma=1$, then since 
$\gamma(d-\alpha) -\alpha\beta=1$ we deduce that $\beta=0$ and
$\gamma=1$, and so $\alpha=d-1$, and  
these are the Vigeland lines.    The result now follows 
from  Corollary \ref{4-valent}.
\end{proof}

\subsection{Isolated lines}
A pair $(\Delta,\Delta')$ of simplices 
$\Delta$ and $\Delta'$ contained in $\Delta_d$ and  with vertices in $\ZZ^3$ is said to 
be \textbf{$d$-pathological} 
if 
\begin{enumerate}
\item $\Delta$ and $\Delta'$ are primitive and intersect along a common edge
  $e$;
\item $\Delta$ has 2 edges 
distinct from $e$ 
and
  contained in the faces $F_i$ and $F_j$ of $\Delta_d$;
\item  one of the two situations occurs:
\begin{enumerate}
\item the edge $e$ is contained in $F_k$, and the opposite edge of $\Delta'$ is
  contained in $F_l$; in this case we say that the pair
  $(\Delta,\Delta')$ is of type $I$ (see Figure \ref{Vigtype1});
\item the polytope $\Delta'$ has a face $F$ containing $e$ and intersecting
 $F_k$ and  $F_l$; in this case we say that the pair
  $(\Delta,\Delta')$ is of type $II$ (see Figure \ref{Vigtype2});
\end{enumerate}
\item the set $\{i,j,k,l\}$ is equal to the set $\{1,2,3,4\}$. 
\end{enumerate}

Isolated lines on a generic non-singular
tropical surface of degree $d\ge 4$ in $\RR^3$ have 
also been classified by combinatorial 
type by Vigeland.

\begin{thm}[Vigeland, {\cite[Theorem 5.11]{Vig2}}]\label{vig thm2}
Let $S$ be a generic non-singular
tropical surface of degree $d\ge 4$ in $\RR^3$ with Newton polytope
$\Delta_d$. If $S$ contains an isolated  tropical line $L$,
then $S$ contains two vertices $v$ and $v'$,  respectively dual to the
simplices $\Delta$
and $\Delta'$, such that the pair $(\Delta,\Delta')$ is $d$-pathological,
and such that $L$ has  one vertex at
 $v$ and 
\begin{itemize}
\item[(a)] passes through $v'$ if $(\Delta,\Delta')$ is of type $I$;
\item[(b)] has another vertex on the edge of $S$ dual to $F$ if  
$(\Delta,\Delta')$ is of type $II$.
\end{itemize}

\end{thm}

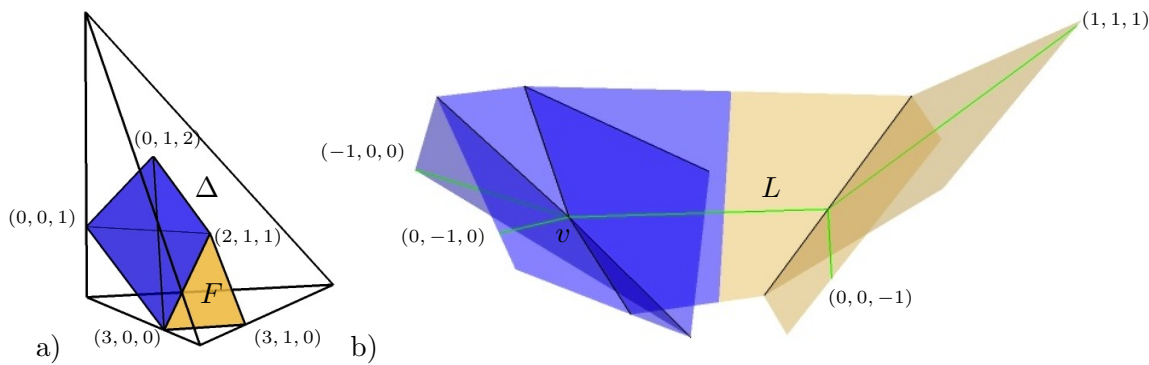
\begin{figure}
a) \includegraphics[scale=0.3]{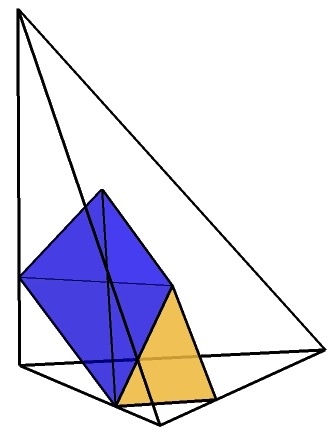}
b) \includegraphics[scale=0.45]{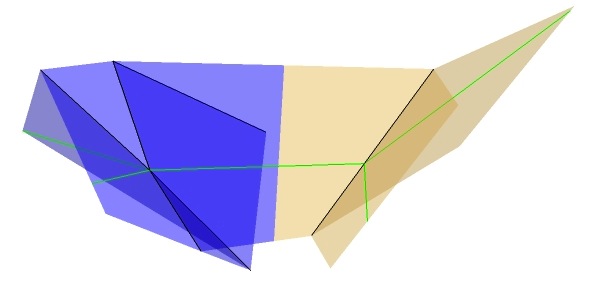}
\put(-342, 60){$\Delta$}
\put(-340, 20){$F$}
\put(-412, 50){\tiny{$(0, 0, 1)$}}
\put(-380, 5){\tiny{$(3, 0, 0)$}}
\put(-320, 5){\tiny{$(3, 1, 0)$}}
\put(-365, 80){\tiny{$(0, 1, 2)$}}
\put(-335, 43){\tiny{$(2, 1, 1)$}}
\put(-207, 43){$v$}
\put(-130, 60){$L$}
\put(-10 ,125){\tiny{$(1, 1, 1)$}}
\put(-105, 20  ){\tiny{$(0, 0, -1)$}}
\put(-295, 75){\tiny{$(-1, 0, 0)$}}
\put(-265, 43){\tiny{$(0, -1, 0)$}}
\caption{a) A pathological pair $(\Delta, \Delta^{\prime})$ of type $II$ in $\Delta_4$. b) The corresponding vertex and edge of the dual surface along with the isolated line $L$. \label{Vigtype2}     } 
\end{figure}

According to the next theorem, none of the tropical lines of Theorem
\ref{vig thm2} are approximable in $S$.
\begin{thm}\label{prohib Vigeland2}
Let $S$ be a generic non-singular
tropical surface of degree $d\ge 4$ in $\RR^3$ with Newton polytope
$\Delta_d$, and let $(\mathcal S_t)$ be a 1-parametric family of
algebraic surfaces in $(\CC^*)^3$ 
with Newton polytope $\Delta_d$
such that 
$$\lim_{t\to+\infty}\Log_t(\mathcal S_t)=S.$$
Suppose that 
 the tropical surface $S$ contains 
a tropical line $L$. Then there does not
 exist a 1-parametric family of lines
$\L_t\subset\mathcal S_t$ such that 
$$\lim_{t\to+\infty}\Log_t(\L_t)=L.$$
\end{thm}
\begin{proof}
Suppose that such a  1-parametric family of lines
$(\L_t)$ exists. It follows from Theorem \ref{prohib Vigeland} that the
tropical line $L$ is isolated. Hence the surface
$S$ contains  two vertices $v$ and $v'$,  respectively dual to the
simplices $\Delta$
and $\Delta'$, such the pair $(\Delta,\Delta')$ is $d$-pathological,
and the line $L$ is as described in 
 Theorem \ref{vig thm2}. By initial degeneration, 
 the family  $(\mathcal S_t,\L_t)$
produces an approximation of the pair $(Star_v(S),Star_v(L))$ by a
constant family. Let us denote by $u_0,u_1,u_2,$ and $u_3$ the
primitive integer directions of the rays of $Star_v(S)$.
Since $L$ is isolated 
 the fan tropical curve $ Star_v(L)$ cannot be
equal to the tropical stable intersection of $\Aff(Star_v(L))$ and
$Star_v(S)$. Moreover $L$ only has edges of weight 1, so 
according to
Theorem \ref{plane curves} the curve $ Star_v(L)$ has degree 2 in
$Star_v(S)$. Without loss of generality, we may assume that the $4$ rays of 
$Star_v(S)$ satisfy one of the two cases:
$$\text{Case 1: \quad}
2u_0+u_3=(-1,0,0),\quad 2u_1+u_2=(0,-1,0),\quad \text{and}\quad
u_2+u_3=(1,1,0), $$
$$\text{Case 2: \quad}
2u_0 + u_1 = (-1, 0, 0), \quad u_1 + u_3 = (0, -1, 0), \quad \text{and} \quad 2u_2 + u_3 = (1, 1, 0). 
$$ See Figure \ref{fig:twocaseIsolines} for the dual polytope $\Delta$ in each case.

Case $1$ corresponds to the situation when the
two edges of $\Delta$ contained in the 
faces $F_i, F_j$ are opposite edges of $\Delta$ and Case $2$ is when these two edges 
are adjacent, see Figure \ref{fig:twocaseIsolines}.

 \vspace{1ex}
In Case $1$ we may set
$$u_0=(a,b,c),\quad u_1=(-a-1,-b-1,-c), \quad u_2=(2a+2,2b+1,2c),
\quad \text{and}\quad u_3=(-2a-1,-2b,-2c).$$
Since $\Delta$ is primitive,
any three vectors among $u_0,u_1,u_2,$ and $u_3$ form a basis for the
lattice $\ZZ^3\subset\RR^3$, hence
$$\pm 1=\left| \begin{array}{ccc} 
a & -a-1& -2a-1
\\ b & -b-1 & -2b
\\ c &-c&-2c
\end{array} \right|$$
which gives $c=\pm1$.

Denote by $i_0,i_1,i_2,$ and $i_3$ the vertices of $\Delta$ in
such a way that $i_j$ is dual to 
the region of $\RR^3\setminus Star_v(S)$ which contains the vector
$-u_j$. The edge $e$ of $\Delta$ is dual to the face of
$Star_v(S)$ generated by $u_2$ and $u_3$, so that $e=[i_0;i_1]$. 
Then we have
$$i_0=(\alpha_0,0,\gamma_0), \quad i_1=(0,\beta_1,\gamma_1), \quad 
i_2=(0,\beta_2,\gamma_2),
\quad \text{and}\quad i_3=(\alpha_3,0,\gamma_3).$$
The edge $e$ is orthogonal to both $u_2$ and $u_3$, therefore it is in
direction $(-2c, 2c, 2a-2b+1),$ since $c = \pm 1$ the direction of $e$ 
is $(-2, 2, \pm(2a-2b+1))$. 
  From this we deduced that 
$$i_0 = (2, 0 , \gamma_0) \quad \text{and} \quad i_1 = (0 , 2, \gamma_0 \pm (2a-2b+1)).$$
Now we have to
distinguish the two cases depending on the type of the pair
$(\Delta,\Delta')$.

\begin{figure}
\includegraphics[scale=0.27]{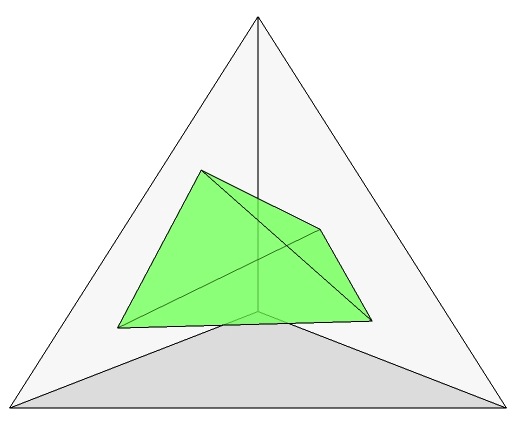}
\hspace{1cm}
\includegraphics[scale=0.27]{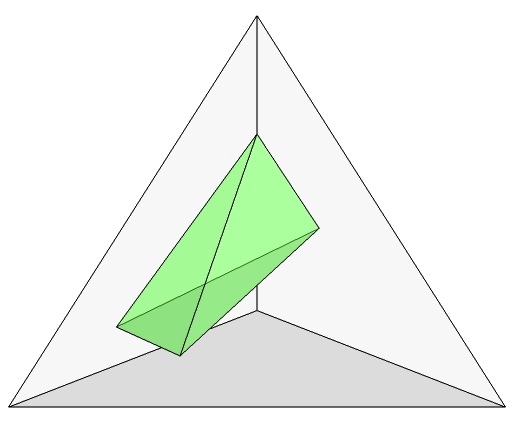}
\put( -290, 20){\tiny{$i_0 $}}
\put( -210, 25){\tiny{$i_1 $}}
\put( -225, 55){\tiny{$i_2 $}}
\put( -263, 71){\tiny{$i_3 $}}
\put( -118, 25){\tiny{$i_0 $}}
\put( -90, 10){\tiny{$i_1 $}}
\put( -65, 80){\tiny{$i_2 $}}
\put( -50, 55){\tiny{$i_3 $}}
\caption{ The two types of  polytopes $\Delta \subset \Delta_d$ dual to the vertex $v$ yielding 
the lines $Star_v(L)$. Case $ 1$ is on the left and Case $2$ on the right.  \label{fig:twocaseIsolines}}
\end{figure}
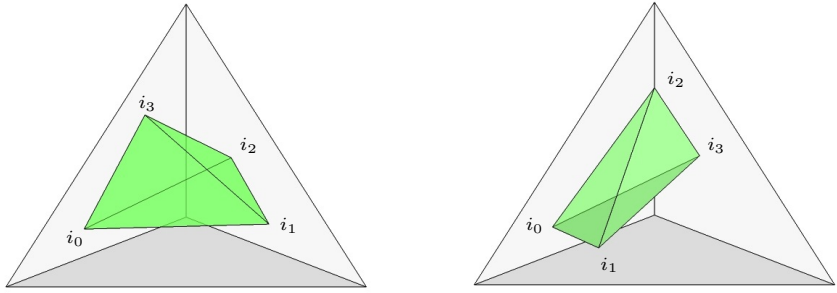

If $(\Delta,\Delta')$ is of type $I$, then the edge $e$ is contained
in a face of $\Delta_d$ which is neither $\{x=0\}$ nor  $\{y=0\}$. Moreover, 
$a, b$ are integers so $2b-2a-1 \neq 0$ for any choice of $a, b$. 
Therefore $e$ cannot be contained in the face $\{z =0\}$, and neither can 
it be contained in $\{x+y+z=d\}$
by symmetry.

If $(\Delta,\Delta')$ is of type $II$, then up to a change of
coordinates we have  
$$i_0=(2,0,0), \quad i_1=(0 ,2 ,d-2), \quad i_2=(0,\beta_2,\gamma_2),
\quad \text{and}\quad i_3=(\alpha_3,0,\gamma_3).$$ Moreover,
the third vertex of the face $F$ of $\Delta'$ has coordinates
$(\alpha,d-\alpha,0)$. Let us denote by $(\beta,\gamma,\delta)$ the
coordinates of the fourth vertex of 
 $\Delta'$. Since the polytope $\Delta^{\prime}$ is primitive, we must have
$$\pm 1=\left| \begin{array}{ccc}
-2 & \alpha-2&\beta-2
\\ 2 & d-\alpha & \gamma
\\ d-2 &0& \delta
\end{array} \right|
=(d-2) C, $$
where $C \in \Z$ is some constant. Therefore,  we obtain $d = 3$, so no isolated lines of this type 
exist when $d \geq 4$.

 \vspace{1ex}
Now consider Case $2$, here we may set:
$$u_0=(a,b,c),\quad u_1=(-2a-1,-2b,-2c), \quad u_2=(-a,-b+1,-c),
\quad \text{and}\quad u_3=(2a+1,2b-1,2c).$$
By a calculation of the determinant of the vectors $u_0, u_1, u_2$ we find again that $c = \pm 1$. 
Again the edge $e$ is orthogonal to both $u_2, u_3$ so that it 
is in the direction $(c, -c, b-a-1)$, since $c = \pm1$ this becomes, 
$(1, -1, \pm (b-a-1))$. 

In this case the vertices of the polytope $\Delta$ have coordinates:
$$i_0 = (\alpha_0, 0,  \gamma_1) , \quad i_1 = (\alpha_1, \beta_1,
\gamma_1), \quad  i_2 = (0, 0, \gamma_2), \quad \text{and} \quad i_3 =
(0, \beta_3, \gamma_3).$$
The edge $[i_0; i_2]$ 
is orthogonal to both $u_1, u_3$ by tropical duality.
Hence this edge has direction 
$ (\pm 2, 0, 1+2a)$, so that $\alpha_0=2$. 

Suppose the pair $(\Delta, \Delta^{\prime})$ is of type $I$, by a
change of coordinates we may assume the edge $e$ lies in the face $\{z
= 0\}$. 
Hence
$\gamma_0 = \gamma_1 = 0$ and the edge $e$ must have
direction $(1, -1, 0)$. Therefore, 
$i_0 = (2, 0, 0) $
and 
$i_1= (1, 1, 0)$.  
Now,  $\Delta^{\prime}$ has vertices 
$i_0$, $i_1$, 
$(\alpha, \beta, d-\alpha- \beta)$, and $(\gamma, \delta,
d-\gamma-\delta)$. The polytope  $\Delta^{\prime}$ is 
primitive
so, 
$$\pm 1=\left| \begin{array}{ccc}
-1 & \alpha-2 &\gamma-2 
\\ 1 & \beta & \delta
\\ 0 &d -\alpha - \beta & d  - \gamma - \delta
\end{array} \right|
=(d-2) C, $$
where $C \in \Z$ is some constant.
This implies $d = 3$, so again there are no isolated lines of this type for $d \geq 4$.  

Finally, if $(\Delta, \Delta^{\prime})$ is of type $II$, then  $i_0$
is in $\{y=z=0\}$ and  $i_1$
is in $\{x +y+z=d\}$, so we have 
$i_0 = (2, 0,  0)$ and 
$i_1 = (1, 1 , d-2)$. Now $\Delta^{\prime}$ has two
other vertices 
$(\alpha,d-\alpha,0)$
and $(\beta,\gamma,\delta)$, and
since $\Delta^{\prime}$ is primitive we have
$$\pm 1=\left| \begin{array}{ccc}
-1 & \alpha-2 &\beta-2 
\\ 1 & d -\alpha & \gamma
\\ d-2 &0& \delta
\end{array} \right|
=(d-2) C, $$
where $C \in \Z$ is some constant.
Once again we obtain $d = 3$ and there
 are no isolated lines of this type for $d \geq 4$.  
This completes the proof. 
\end{proof}

\begin{proof}[Proof of Theorem \ref{prohib Vigeland intro}]
According to {\cite[Theorem 5.6]{Vig2}},
such a generic tropical surface contains finitely
many isolated lines and finitely many 1-parametric families of
tropical lines. Now the result follows from initial degeneration,
Theorems \ref{vig thm}, \ref{vig thm2}, \ref{prohib Vigeland},
and \ref{prohib Vigeland2}. 
\end{proof}

\subsection{Singular tropical lines}\label{concluding remark}
To conclude this paper, 
let us point out a strange phenomenon in
tropical geometry that we were not aware of before starting this
investigation.
Let $S$ be a generic non-singular tropical surface of degree $d\ge 3$ in
$\TT P^3$ which has a vertex dual to a $d$-pathological
simplex with $(\alpha,\beta,\gamma)\ne (d-1,0,1)$ and $(1,1,0)$. 
Then any tropical line $L$ in the corresponding
1-parameter family is
singular when considered as a tropical curve in $S$. Indeed 
the self-intersection of such a line is $-(\beta+\gamma)(d-1)+1$
 in $S$, which is
far from being equal to  $2-d$, which is 
 the self-intersection of a complex algebraic line in a complex
algebraic surface of degree $d$ in $\CC P^3$. In particular, this
means that none of these tropical lines satisfy the adjunction formula
in $S$, i.e. they are singular.

It is however possible to give a sufficient condition for a tropical curve in a 
non-singular tropical surface to behave as a non-singular curve. 
If a  tropical curve $C$ in a non-singular compact tropical surface $X$  
 can be  locally given by a fan curve of degree $1$ in a tropical plane, 
then it is shown in {\cite[Section 3.4.6]{ShawTh}} that the 
curve satisfies the tropical adjunction formula, where the genus of 
the curve given by the first Betti number of $C$. This condition is however
not necessary for a curve to satisfy the adjunction formula.

\end{section}

\small
\def\rightmark{\em Bibliography}

\bibliographystyle{alpha}
\bibliography{Biblio.bib}

\end{document}